\documentclass{amsart}
\usepackage{amsmath,amsfonts,amsthm,amscd,amssymb,latexsym,comment,color}
\usepackage{enumerate,graphicx,psfrag,subfigure,changebar,oldgerm}
\usepackage[mathscr]{eucal}
\usepackage[all]{xy}
\theoremstyle{plain}

\newtheorem{thm}{Theorem}[section]
\newtheorem{lem}[thm]{Lemma}
\newtheorem{prop}[thm]{Proposition}
\newtheorem{cor}[thm]{Corollary}
\newtheorem*{thmge}{Theorem \ref{ge}}
\newtheorem*{thmase}{Theorem \ref{Ase}}
\newtheorem{prob}{Problem}
\newtheorem*{pr3}{Problem E5, \cite{grouptheory} and Problem 3, \cite{BMR1}}

\theoremstyle{definition}

\newtheorem{defn}[thm]{Definition}

\newtheorem*{NB}{Nota Bene}

\newtheorem{rem}{Remark}

\newcommand{\bS}{\mathbf{S}}
\newcommand{\cL}{\mathcal{L}}
\newcommand{\edge}{\mathbf{e}}
\renewcommand{\t}{\mathfrak{t}}
\renewcommand{\a}{\mathfrak{a}}
\renewcommand{\b}{\mathfrak{b}}
\newcommand{\p}{\mathfrak{p}}
\newcommand{\s}{\mathfrak{s}}
\renewcommand{\ss}{\mathbf{s}}

\DeclareMathOperator{\Cal}{Cay}

\newcommand{\cA}{\mathcal{A}}

\newcommand{\cc}{\mathfrak{c}}

\DeclareMathOperator{\dec}{dec}
\DeclareMathOperator{\sol}{sol}
\DeclareMathOperator{\Aut}{Aut}
\DeclareMathOperator{\lef}{left}
\DeclareMathOperator{\rig}{right}

\DeclareMathOperator{\sign}{sign}
\DeclareMathOperator{\Hom}{Hom}
\DeclareMathOperator{\id}{id}
\DeclareMathOperator{\Ker}{Ker}

\DeclareMathOperator{\ncl}{ncl}
\DeclareMathOperator{\itype}{i-type}
\DeclareMathOperator{\ttype}{t-type}
\DeclareMathOperator{\type}{type}
\DeclareMathOperator{\tp}{tp}
\DeclareMathOperator{\ET}{ET}
\DeclareMathOperator{\D}{D}
\DeclareMathOperator{\Sh}{Sh}
\DeclareMathOperator{\comp}{comp}

\newcommand{\VV}{\mathfrak{V}}
\renewcommand{\AA}{\mathfrak{A}}
\newcommand{\BB}{\mathfrak{B}}

\newcommand{\PT}{\mathcal{PT}}
\newcommand{\T}{\mathcal{T}}
\renewcommand{\P}{\mathcal{P}}

\newcommand{\nn}{\mathfrak{n}}

\newcommand{\m}{\mathfrak{m}}

\newcommand{\BC}{\mathcal{BC}}
\newcommand{\BS}{\mathcal{BS}}
\newcommand{\BD}{\mathcal{BD}}
\newcommand{\GE}{\mathcal{GE}}

\newcommand{\NN}{\mathcal{N}}
\newcommand{\M}{\mathcal{M}}
\newcommand{\B}{\mathcal{B}}

\newcommand{\N}{\mathbb{N}}
\newcommand{\Z}{\mathbb{Z}}

\newcommand{\factor}[2]{{\raise0.7ex\hbox{$#1$} \!\mathord{\left/ {\vphantom {#1 {#2}}}\right.\kern-\nulldelimiterspace} \!\lower0.7ex\hbox{${#2}$}}}

\title[Equations over Free Products]{On Systems of Equations over Free Products of Groups}

\author[M. Casals-Ruiz]{Montserrat Casals-Ruiz}\thanks{The first author is supported by Programa de Formaci\'{o}n
de Investigadores del Departamento de Educaci\'{o}n, Universidades e Investigaci\'{o}n del Gobierno Vasco}
\address{Montserrat Casals-Ruiz,
Department of Mathematics and Statistics, McGill University,
805 Sherbrooke St. West, Montreal,
Quebec H3A 2K6, Canada}
\email{casalsruiz@math.mcgill.ca}

\author[I. Kazachkov]{Ilya V. Kazachkov}\thanks{The second author is supported by Seymour Schulich Fellowship and le Bourse d'Excellence de l'Institut des Sciences Math\'{e}matiques}
\address{Ilya V. Kazachkov,
Department of Mathematics and Statistics, McGill University,
805 Sherbrooke St. West, Montreal,
Quebec H3A 2K6, Canada}
\email{kazachkov@math.mcgill.ca}

\begin{document}

\begin{abstract}
Using an analogue of Makanin-Razborov diagrams, we give a description of the solution set of systems of equations over an equationally Noetherian free product of groups $G$. Equivalently, we give a parametrisation of the set $\Hom(H, G)$ of all homomorphisms from a finitely generated group $H$ to $G$.
Furthermore, we show that every algebraic set over $G$ can be decomposed as a union of finitely many images of algebraic sets of NTQ systems.

If the universal Horn theory of $G$ (the theory of quasi-identities) is decidable, then our constructions are effective.
\end{abstract}
\maketitle
In this paper we study the structure of the set of solutions of an arbitrary system of equations over an equationally Noetherian free product of groups $G=G_1*G_2$. It is clear that the structure of the free product is at least as complicated as is the structure of its factors. But a general guideline is that it is not much more complicated. From this perspective, a lot of work has been done in order to understand the relations between the factors and the free product itself, i.e. in order to establish which properties of the free product reduce to the factors and, conversely, which properties are stable under taking free products.

Intuitively, one might expect that the properties of the factors that lift to the free product are those that are also satisfied by free groups. As an illustration, properties such as: to be finite, to be abelian (nilpotent, solvable), to be a torsion group, do not lift to the free product. On the other hand, properties such as: to be torsion free, to have unique roots, to be residually finite, to be CSA, to be (relatively) hyperbolic, to be CAT(0), to have decidable word, conjugacy and isomorphism problems lift to free products. However, as in any rule, there are exceptions, or at least subtleties: it was proved by  I.~Dey and H.~Neumann that the Hopf property is stable under taking free products in the case when the factors are finitely generated, see \cite{DN}, but if one does not assume the factors to be finitely generated, then the free product of Hopfian groups may be non-Hopfian, see \cite{NS}.

The interest in understanding the relations between first-order theories of the free product and of its factors, probably, goes back to A.~Tarski. It appears that R.~L.~Vaught, while doing his PhD under Tarski's supervision, was working on the following problem, see \cite{FV} and references there.
\begin{prob}\label{prob:1}
Does the elementary theory of the free product of groups reduce to the elementary theories of the factors? More precisely:
\begin{enumerate}
\item Let $G_1$ be elementarily equivalent to $H_1$, and $G_2$ be elementarily equivalent to $H_2$. Is it true that $G_1*G_2$ is elementarily equivalent to $H_1*H_2$?
\item Let the elementary theories of $G_1$ and $G_2$ be decidable, is the elementary theory of $G$ decidable?
\end{enumerate}
\end{prob}
It was communicated to us by V.~Remeslennikov that this problem was in the school of A.~Malcev since at least early 1960s. It also appears in Z.~Sela's list of open problems \cite{prsela}, where he attributes this problem to G.~Sabbagh, see Problem III.10.

Similar problems for other algebraic structures were studied by P. Olin, see \cite{Olin}, who proved that free product preserves  elementary equivalence of groupoids, but does not preserve the elementary equivalence of semigroups. For free products of groups, some partial results were obtained by V.~Diekert and M.~Lohrey. In their paper \cite{DL}, they proved that two free products of groups with universally (existentially) equivalent factors are universally (existentially) equivalent. Moreover, an analogous statement holds if the factors have the same positive theories. Furthermore, Diekert and Lohrey proved that the positive and universal (existential) theories of the free product are decidable, provided that the corresponding theories of the factors are decidable.

The possibility of approaching questions about the elementary theory of free products is suggested by the recent solution of Tarski's problems on the elementary theory of free groups. In a series of papers O.~Kharlampovich and A.~Myasnikov (see \cite{KhMTar} and references there), and Z.~Sela  (see \cite{SelaTar} and references there) proved that non-abelian free groups are elementarily equivalent.

In order to understand the elementary theory of a group, one has to study the structure of algebraic varieties and their projections. It is therefore natural to ask to which extent algebraic varieties over the factors determine the structure of algebraic varieties over the group.

The first description of algebraic varieties over a free group (viewed as solution sets of systems of equations), was given by A.~Razborov in \cite{Razborov1}, \cite{Razborov3}. In his work, Razborov developed Makanin's approach to the compatibility problem for systems of equations over free groups, see \cite{Mak82}, and gave a construction of what nowadays is called Makanin-Razborov diagrams. In this paper we present an analogue of Makanin-Razborov process and generalise some of the results from \cite{KMIrc} to free products. Therefore, this paper can be viewed as a first step towards solution of Problem \ref{prob:1}.

\medskip

 The main results of this paper are the following two theorems. Recall that for any finitely generated group $H=\langle x_1,\dots, x_n\mid S\rangle$ the set of homomorphisms $\Hom(H, G)$ is in one-to-one correspondence with the set of solutions of the system $S$ over $G$.

\begin{thmge}
Let $G=G_1\ast G_2$ be an equationally Noetherian free product of groups and let $H$ be a finitely generated {\rm(}$G$-{\rm)}group. Then the set of all {\rm(}$G$-{\rm)}homomorphisms  $\Hom(H,G)$ {\rm(}$\Hom_G(H,G)$, correspondingly{\rm)} from $H$ to $G$ can be described by a finite rooted tree. This tree is oriented from the root, all its vertices except for the root vertex are labelled by coordinate groups of generalised equations. To each vertex group we assign the group of automorphisms $A(\Omega_v)$. The leaves of this tree are labelled by groups of the form $F(X)\ast {G_1}_{R(S_1)}\ast {G_2}_{R(S_2)}$ {\rm(}$G\ast F(X)\ast {G_1}_{R(S_1)}\ast {G_2}_{R(S_2)}$, correspondingly{\rm)}, where ${G_i}_{R(S_i)}$ is a coordinate group over $G_i$.

Each edge of this tree is labelled by a {\rm(}$G$-{\rm)}homomorphism. Furthermore, every edge except for the edges from the root is labelled by a proper epimorphism. Each edge from the root vertex corresponds to a {\rm(}$G$-{\rm)}homomorphism from $H$ to the coordinate group of a generalised equation.

Every {\rm(}$G$-{\rm)}homomorphism from $H$ to $G$ can be written as a composition of the {\rm(}$G$-{\rm)}homomorphisms corresponding to the edges, automorphisms of the groups assigned to the vertices, and a specialisation of the variables of the generalised equation corresponding to a leaf of the tree.

This description is effective, provided that the universal Horn theory {\rm(}the theory of quasi-identities{\rm)} of $G$ is decidable.
\end{thmge}

\begin{thmase}
For any system of equations $S(X)=1$ over an equationally Noetherian free product of groups $G=G_1*G_2$, there exists a finite family of nondegenerate triangular quasi-quadratic systems $C_1,\ldots, C_k$ over the groups of the form $G \ast  {G_1}_{R(S_1)}\ast {G_2}_{R(S_2)}$, where ${G_i}_{R(S_i)}$
is a coordinate group over $G_i$, and morphisms of algebraic sets $p_i: V_G(C_i) \rightarrow V_G(S)$, $i = 1, \ldots,k$ such that for every $b \in V_G(S)$ there exists $i$ and $c \in V_G(C_i)$ for which $b = p_i(c)$, i.e.
$$
V_G(S) = p_1(V_G(C_1)) \cup \ldots \cup p_k(V_G(C_k)).
$$
If the universal Horn theory of $G$ is decidable, then the finite families $\{C_1,\ldots, C_k\}$ and $\{p_1,\dots, p_k\}$ can be constructed effectively.
\end{thmase}

The above theorems require the group $G$ to be equationally Noetherian.  Recall that a group $G$ is called equationally Noetherian if for any system of equations $S$ over $G$ there exists a finite subsystem $S_0$ of $S$ such that the sets of solutions of $S_0$ and $S$ coincide.

The notion of Noetherian ring is central in classical algebraic geometry. So, it is natural, when studying algebraic geometry over a group $H$, to require $H$ to be equationally Noetherian. The property of $H$ to be equationally Noetherian is equivalent to the Zariski topology on $H^n$ to be Noetherian for all $n$. It is therefore necessary in order that every closed set in the Zariski topology be a \emph{finite} union of irreducible algebraic sets, and every coordinate group be a subdirect product of \emph{finitely many} coordinate groups of irreducible algebraic sets. Furthermore, the property of a group  to be equationally Noetherian gives a lot of structure to the coordinate groups. As shown in \cite{BMR1}, \cite{AG2}, \cite{OH} and \cite{DMR}, it allows one to characterise coordinate groups from many different viewpoints: as groups that satisfy some residual properties; from the viewpoint of universal algebra; as direct limits of finite substructures; as algebraic limit groups; from model-theoretic viewpoint; from the viewpoint of the theory of atomic types.

Equationally Noetherian groups are abundant. It is known that every linear group (over a commutative, Noetherian, unitary ring) is equationally Noetherian (see \cite{Gub}, \cite{Br}, \cite{BMR1}). On the other hand, hyperbolic groups are equationally Noetherian (torsion free case is due to Z.~Sela, see \cite{SelaHyp} and torsion case is due to C.~Reinfeldt and R.~Weidmann), toral relatively hyperbolic groups and CAT(0) groups with isolated flats are also equationally Noetherian (due to D.~Groves, \cite{groves}), etc.  All these classes are closed under taking free products of groups. However, it is not known if the free product of equationally Noetherian groups is equationally Noetherian.
\begin{pr3}
Let $G_1$ and $G_2$ be two equationally Noetherian groups. Is the free product $G_1\ast G_2$ equationally Noetherian?
\end{pr3}

Furthermore, the class of equationally Noetherian groups is closed under taking subgroups, finite direct products and ultrapowers; it is also closed under universal equivalence and separation, see Theorem B2, \cite{BMR1}.

\bigskip

When this paper was in the final stage of preparation, a preprint \cite{JS} of E.~Jaligot and Z.~Sela addressing similar questions to the ones studied in this paper was published. Note that the general problem of describing the set of homomorphism from a \emph{finitely generated} group $H$ to an \emph{arbitrary} free product remains open. In their work Jaligot and Sela describe homomorphisms from finitely presented groups to (arbitrary) free products. In this paper, we require that the free product $G$ be equationally Noetherian, but give a description of the set of homomorphisms from any finitely generated ($G$-)group to $G$.

\section{Preliminaries}

\subsection{Free products}
Throughout this paper we denote by $G$ the free product of two fixed groups $G_1$ and $G_2$,  $G=G_1\ast G_2$. Note that we do not impose any restrictions on the finiteness properties (e.g. finite presentability or finite generation) and the cardinality of $G_1$ and $G_2$. Any element from $G$ can be represented as an alternating product of the elements from $G_1$ and $G_2$:
\begin{equation} \label{eq:nf}
g=g_{1,1}g_{2,1}\cdots g_{1,l}g_{2,l},
\end{equation}
where $g_{1,i}\in G_1$, $g_{2,j}\in G_2$, $g_{1,2},\dots, g_{1,l}\ne 1$, $g_{2,1},\dots, g_{2,l-1}\ne 1$ and $g_{1,1},g_{2,l}$ are possibly trivial. In the latter case we omit writing them. We think of elements $g_{1,i},g_{2,j}$ as \emph{elements} of the factors and presentation (\ref{eq:nf}) is unique for any $g\in G$, see \cite{LS}. In particular, we adopt the following convention: if $g_{i,j_1}$, $g_{i,j_2}$ are terms of a word and $g_{i,j_1}=g_{i,j_2}^{-1}$, we only use either $g_{i,j_1}$ and write $g_{i,j_2}$ as $g_{i,j_1}^{-1}$ or vice-versa. Note that we do not impose any restrictions on the groups $G_1$ and $G_2$, e.g. the groups $G_1$ and $G_2$ may be infinitely presented or infinitely generated.

We call elements written in the form (\ref{eq:nf}) \emph{reduced}. We refer to elements $g_{1,i}\in G_1$, $g_{2,j}\in G_2$ as to the \emph{terms} of $g$. Below if not stated otherwise we assume that all words are reduced.

Henceforth, by the symbol \glossary{name={`$\doteq$'}, description={graphical equality of words}, sort=Z}`$\doteq$' we denote graphical equality of words, i.e. for two words $g$, $g'$  of $G$ written in normal form (\ref{eq:nf}) we write $g\doteq g'$ if and only if $g=g_{1,1}g_{2,1}\cdots g_{1,l}g_{2,l}$, $g'=g_{1,1}g_{2,1}\cdots g_{1,l}g_{2,l}$.

The length $|g|$ of an element $g$ written in the form (\ref{eq:nf}) is the number of non-trivial terms in $g$.

We now define the initial and terminal types of an element. These functions determine the factors to which the initial and terminal letters of a word in $G$ belong. Define the functions $\itype$ and $\ttype$ from $G$ to the set $\{1,2\}$ as follows:
$$
\itype(g)=
\left\{
  \begin{array}{ll}
    1, & \hbox{if, when written in the form (\ref{eq:nf}), $g=g_{1,1}\cdots g_{2,l}$, $g_{1,1}\ne 1$;} \\
    2, & \hbox{if, when written in the form (\ref{eq:nf}), $g=g_{2,1}\cdots g_{2,l}$, $g_{2,1}\ne 1$;}
\end{array}
\right.
$$
$$
\ttype(g)=
\left\{
  \begin{array}{ll}
 1, & \hbox{if, when written in the form (\ref{eq:nf}) $g=g_{1,1}\cdots g_{1,l}$, $g_{1,l}\ne 1$;} \\
 2, & \hbox{if, when written in the form (\ref{eq:nf}) $g=g_{1,1}\cdots g_{2,l}$, $g_{2,l}\ne 1$.}
  \end{array}
\right.
$$
Set $\type(i,j)=\itype^{-1}(i)\cap\ttype^{-1}(j)$, $i,j\in \{1,2\}$, where $\itype^{-1}(i)$ and $\ttype^{-1}(j)$ are the full preimages in $G$ of $i$ and $j$ with respect to $\itype$ and $\ttype$ correspondingly. By the definition, set $\type(i,0)=\itype^{-1}(i)$ and $\type(0,j)=\ttype^{-1}(j)$. Finally, $\type(0,0)$ is the set of all non-trivial elements of the group $G$.

We say that an element $h\ne 1$ is a subword of $g$ if the normal form of $h$ is a subword of the normal form of $g$. In particular if $g\doteq g_1hg_2$, $|g|=|g_1|+|h|+|g_2|$ and if $h\in \type(i,j)$, then $g_1\in \type(k_1,3-i)$ and $g_2\in \type(3-j,k_2)$, where $i,j,k_1,k_2\in \{1,2\}$

We say that $g_1g_2$ is a subdivision of $g$ if $g_1$ and $g_2$ are subwords of $g$ and $g\doteq g_1g_2$. In particular, we have $|g|=|g_1|+|g_2|$ and if $g_1\in \type(i,j)$ then $g_2\in \type(3-i,k)$ and, moreover, $g\in\type(i,k)$, where $i,j,k\in \{1,2\}$. If $g\doteq g_1g_2$ is a subdivision of $g$, then we say that $g_1$ and $g_2$ are, respectively, initial and terminal subwords of $g$.

Let $g\in G$ and $|g| \ge 2$. Then $g$  is called \emph{cyclically reduced} if and only if, $\itype(g)\ne\ttype(g)$. An element $h\in G$ is called a \emph{cyclic permutation} of a cyclically reduced element $g\in G$ if and only if there exists a subdivision $g\doteq g_1g_2$ of $g$ such that $h\doteq g_2g_1$. Note that in this case $h$ is also cyclically reduced.

\begin{NB}
We would like to note that we assume that the group $G$ is a free product of two groups $G_1$ and $G_2$ for the sake of simplicity of notation only. All the definitions, proofs and results in this paper carry over in a straightforward way to the case when $G$ is a free product of \emph{finitely many} groups, $G=G_1*\dots *G_n$.
\end{NB}

\subsection{Graphs}

In this section we introduce notation for graphs that we use in this paper. Let $\Gamma=(V(\Gamma),E(\Gamma))$ be an oriented graph, where $V(\Gamma)$ is the set of vertices of $\Gamma$ and $E(\Gamma)$ is the set of edges of $\Gamma$. If an edge $e:v\to v'$ has \emph{origin} $v$ and \emph{terminus} $v'$, we sometimes write $e=v\to v'$. We always denote the paths in a graph by letters $\p$ and $\s$, and cycles by the letter $\cc$. To indicate that a path $\p$ begins at a vertex $v$ and ends at a vertex $v'$ we write $\p(v,v')$. If not stated otherwise, we assume that all paths we consider are simple. For a path $\p(v,v')=e_1\dots e_l$ by ${(\p(v,v'))}^{-1}$ we denote the reverse (if it exists) of the path $\p(v,v')$, that is a path $\p'$ from $v'$ to $v$, $\p'= e_l^{-1}\dots e_1^{-1}$, where $e_i^{-1}$ is the inverse of the edge $e_i$, $i=1,\dots l$.

Usually, the edges of the graph are labelled by certain words or letters. The \emph{label} of a path $\p=e_1\dots e_l$ is the concatenation of labels of the edges $e_1\dots e_l$.

Let $\Gamma$ be an oriented rooted tree, with the root at a vertex $v_0$ and such that for any vertex $v\in V(\Gamma)$ there exists a unique path $\p(v_0,v)$ from $v_0$ to $v$. The length of this path from $v_0$ to $v$ is called the \index{height!of a vertex}\emph{height of the vertex $v$}. The number $\max\limits_{v\in V(\Gamma)} \{\hbox{height of } v\}$, if it exists, is called the \index{height!of a tree}\emph{height of the tree $\Gamma$}, otherwise we say that the height of $\Gamma$ is infinite. We say that a vertex $v$ of height $l$ is \emph{above} a vertex $v'$ of height $l'$ if and only if $l>l'$ and there exists a path of length $l-l'$ from $v'$ to $v$.

\subsection{Elements of algebraic geometry over groups}
\label{se:2-4}

In \cite{BMR1} G.~Baumslag, A.~Miasnikov and V.~Remeslennikov lay down the foundations of algebraic geometry over groups and introduce group-theoretic counterparts of basic notions from algebraic geometry over fields. We now recall some of the basics of algebraic geometry over groups. We refer to \cite{BMR1} for details.

Algebraic geometry over groups centers around the notion of a \index{group@$G$-group}\emph{$G$-group}, where $G$ is a fixed group generated by a set $A$. These $G$-groups can be likened to algebras over a unitary commutative ring, more specially a field, with $G$ playing the role of the coefficient ring. We therefore, shall consider the category of $G$-groups, i.e. groups which contain a designated subgroup isomorphic to the group $G$. If $H$ and $K$ are $G$-groups then a homomorphism $\varphi: H \rightarrow K$ is a \index{homomorphism@$G$-homomorphism}\emph{$G$-homomorphism} if $\varphi(g)= g$ for every $g \in G$. In the category of $G$-groups morphisms are $G$-homomorphisms; subgroups are \index{subgroup@$G$-subgroup}$G$-subgroups, etc. By \glossary{name={$\Hom_G(H,K)$}, description={set of $G$-homomorphisms from $H$ to $K$}, sort=H} $\Hom_G(H,K)$ we denote the set of all $G$-homomorphisms from $H$ into $K$. A $G$-group $H$ is termed \index{group@$G$-group!finitely generated}{\em finitely generated $G$-group} if there exists a finite subset $C \subset H$ such that the set $G \cup C$ generates $H$. It is not hard to see that the free product $G \ast F(X)$ is a free object in the category of $G$-groups, where $F(X)$ is a free group with basis $X = \{x_1, x_2, \ldots  x_n\}$. This group is called a free $G$-group with basis $X$, and we denote it  by $G[X]$.

For any element $s\in G[X]$ the formal equality $s=1$ can be treated, in an obvious way, as an \index{equation over a group}\emph{equation} over $G$. In general, for a subset  $S \subset G[X]$ the formal equality $S=1$ can be treated as \index{system of equations over a group}{\em a system of equations} over $G$ with coefficients in $A$. In other words, every equation is a word in $(X \cup A)^{\pm 1}$. Elements from $X$ are called {\em variables}, and elements from $A^{\pm 1}$ are called {\em coefficients} or {\em constants}. To emphasize this we sometimes write $S(X,A) = 1$.

A \index{solution!of a system of equations} {\em solution} $U$ of the system $S(X) = 1$ over a group $G$ is a tuple of elements $g_1, \ldots, g_n \in G$ such that every equation from $S$ vanishes at $(g_1, \ldots, g_n)$, i.e. $S_i(g_1, \ldots, g_n)=1$ in $G$, for all $S_i\in S$. Equivalently, a solution $U$ of the system $S = 1$ over $G$ is a $G$-homomorphism \glossary{name={$\pi_U$}, description={homomorphism defined by the solution $U$}, sort=P}$\pi_U: G[X] \to G$ induced by the map $\pi_U: x_i\mapsto g_i$ such that $S\subseteq \ker(\pi_U)$. When no confusion arises, we abuse the notation and write $U(w)$, where $w\in G[X]$, instead of $\pi_U(w)$.

Denote by \glossary{name=$\ncl\langle S\rangle$, description={normal closure of $S$}, sort=N}$\ncl\langle S\rangle$ the normal closure of $S$ in $G[X]$. Then every solution of $S(X) = 1$ in $G$ gives rise to a $G$-homomorphism $\factor{G[X]}{\ncl\langle S\rangle )} \to G$, and vice-versa. The set of all solutions over $G$ of the system $S=1$ is denoted by \glossary{name={$V_G(S)$},description={algebraic set defined by $S$ over $G$}, sort=V}$V_G(S)$ and is called the \index{algebraic set}{\em algebraic set defined by} $S$.

For every system of equations $S$ we set the \index{radical!of a system}{\em radical of the system $S$}  to be the following subgroup of  $G[X]$:\glossary{name=$R(S)$, description={radical of the system $S$}, sort=R}
$$
R(S) = \left\{ T(X) \in G[X] \ \mid \ \forall g_1,\dots,\forall g_n  \left( S(g_1,\dots,g_n) = 1 \rightarrow T(g_1,\dots, g_n) = 1\right) \right\}.
$$
It is easy to see that  $R(S)$ is a normal subgroup of $G[X]$ that contains $S$. There is a one-to-one correspondence between algebraic sets $V_G(S)$ and radical subgroups $R(S)$ of $G[X]$. Notice that if $V_G(S) = \emptyset$, then $R(S) = G[X]$.

It follows from the definition that
$$
R(S)=\bigcap\limits_{U\in V_G(S)}\ker(\pi_U).
$$
The factor group \glossary{name=$G_{R(S)}$, description={coordinate group of the system $S$ over $G$}, sort=G}
$$
G_{R(S)}=\factor{G[X]}{R(S)}
$$
is called the \index{coordinate group!of an algebraic set}\index{coordinate group!of a system of equations}{\em coordinate group} of the algebraic set  $V_G(S)$ (or of the system $S$). There exists a one-to-one correspondence between the algebraic sets and  coordinate groups $G_{R(S)}$. More formally, the categories of algebraic sets and coordinate groups are equivalent, see Theorem 4, \cite{BMR1}.

A $G$-group $H$ is called \index{equationally Noetherian@$G$-equationally Noetherian group}{\em $G$-equationally Noetherian} if every system $S(X) = 1$ with coefficients from $G$ is equivalent over $G$ to a finite subsystem $S_0 = 1$, where $S_0 \subset S$, i.e. the system $S$ and its subsystem $S_0$ define the same algebraic set. If $G$ is $G$-equationally Noetherian, then we say that $G$ is equationally Noetherian. If $G$ is equationally Noetherian then the Zariski topology over $G^n$ is {\em Noetherian} for every $n$, i.e., every proper descending chain of closed sets in $G^n$ is finite. This implies that every algebraic set $V$ in $G^n$ is a finite union of irreducible subsets, called \index{irreducible!component}{\em irreducible components} of $V$, and such a decomposition of $V$ is unique. Recall that a closed subset $V$ is \index{irreducible!algebraic set}{\em irreducible} if it is not a union of two proper closed (in the induced topology) subsets.

\medskip

If $V_G(S) \subseteq G^n$ and $V_G(S') \subseteq G^m$ are algebraic sets, then a map $\phi: V_G(S) \to V_G(S')$ is a
\index{morphism of algebraic sets}\emph{morphism} of algebraic sets if there exist $f_1,\dots,f_m \in G[x_1,\ldots ,x_n]$ such that, for any $(g_1,\ldots,g_n) \in V_G(S)$,
$$
\phi(g_1,\dots,g_n)=(f_1(g_1,\dots,g_n),\ldots,f_m(g_1,\dots,g_n)) \in V_G(S').
$$
Occasionally we refer to morphisms of algebraic sets as \emph{word mappings}.

Algebraic sets $V_G(S)$ and $V_G(S')$ are called \index{isomorphism of algebraic sets}\emph{isomorphic} if there exist morphisms $\psi: V_G(S) \rightarrow V_G(S')$  and $\phi: V_G(S') \rightarrow V_G(S)$ such that $\phi \psi  = \id_{V_G(S)}$ and $\psi \phi = \id_{V_G(S')}$.

Two systems of equations $S$ and $S'$ over $G$ are called \index{equivalence!of systems of equations}\emph{equivalent} if the algebraic sets $V_G(S)$ and  $V_G(S')$ are isomorphic.

\begin{prop}
Let $G$ be a group and let $V_G(S)$ and $V_G(S')$ be two algebraic sets over $G$. Then the algebraic sets $V_G(S)$ and $V_G(S')$ are isomorphic if and only if the coordinate groups $G_R(S)$ and $G_{R(S')}$ are $G$-isomorphic.
\end{prop}

\subsection{Formulas in the languages $\cL_{A}$ and $\cL_{G}$}
\label{se:2-3}

In this section we recall some basic notions of first-order logic and model theory. We refer the reader to \cite{ChKe} for details.

Let $G$ be a group generated by the set $A$. The standard first-order language of group theory, which we denote by \glossary{name=$\cL$, description={first-order language of groups}, sort=L}$\cL$, consists of a symbol for multiplication $\cdot$, a symbol for inversion $^{-1}$, and a symbol for the identity $1$.  To deal with $G$-groups, we have to enrich the language $\cL$ by all the elements from $G$ as constants. In fact, as $G$ is generated by $A$, it suffices to enrich the language $\cL$ by the constants that correspond to elements of $A$, i.e. for every element of $a\in A$ we introduce a new constant $c_a$. We denote language $\cL$ enriched by constants from $A$ by \glossary{name={$\cL_A$}, description={first-order language of groups enriched by constants from $A$}, sort=L}$\cL_{A}$, and by constants from $G$ by \glossary{name={$\cL_G$}, description={first-order language of groups enriched by constants from $G$}, sort=L}$\cL_G$. In this section we further consider only the language $\cL_A$, but everything stated below carries over to the case of the language $\cL_G$.

A group word in variables $X$ and constants $A$ is a word $W(X,A)$ in the alphabet $(X\cup A)^{\pm 1}$. One  may consider the word $W(X,A)$ as a term in the language $\cL_A$. Observe that every term in the language $\cL_A$ is  equivalent modulo the axioms of group theory to a group word in variables $X$ and constants $A$. An \index{formula!atomic}{\em atomic formula}  in the language $\cL_A$ is a formula of the type $W(X,A) = 1$, where $W(X,A)$ is a group word in $X$ and $A$.  We interpret atomic formulas in $\cL_A$ as equations over $G$, and vice versa.  A {\em Boolean combination} of atomic formulas in the language $\cL_A$ is a disjunction of conjunctions of atomic formulas and their negations. Thus every Boolean combination $\Phi$  of atomic  formulas in $\cL_A$ can be written in the form $\Phi =  \bigvee\limits_{i=1}^m\Psi_i$, where each $\Psi_i$ has one of following form:
$$
\bigwedge\limits_{j = 1}^n(S_j(X,A) = 1),  \hbox{ or } \bigwedge\limits_{j =1}^n(T_j(X,A) \neq 1),  \hbox{ or }
\bigwedge\limits_{j = 1}^n (S_j(X,A) = 1) \ \wedge \  \bigwedge_{k = 1}^m (T_k (X,A) \neq 1).
$$

It follows from general results on disjunctive normal forms in propositional logic that every quantifier-free formula in the language $\cL_\cA$ is logically equivalent (modulo the axioms of group theory) to a Boolean combination of  atomic ones. Moreover, every formula $\Phi$ in $\cL_A$ with \emph{free variables} $Z=\{z_1,\ldots ,z_k\}$  is logically equivalent to a formula of the type
$$
Q_1x_1 Q_2 x_2 \ldots Q_l x_l \Psi(X,Z,A),
$$
where  $Q_i \in \{\forall, \exists \}$, and  $\Psi(X,Z,A)$ is a Boolean combination of atomic formulas in variables from $X \cup Z$.
Introducing fictitious quantifiers, if necessary, one can always rewrite the formula $\Phi$ in the form
$$
\Phi(Z)  = \forall x_1 \exists y_1 \ldots \forall x_k \exists y_k \Psi(x_1, y_1, \ldots, x_k, y_k, Z).
$$

A first-order formula $\Phi$ is called a \index{sentence}\emph{sentence}, if $\Phi$ does not contain free variables.

A sentence $\Phi$ is called \index{formula!universal}\emph{universal} if and only if $\Phi$ is equivalent to a formula of the type:
$$
\forall x_1 \forall x_2 \ldots \forall x_l \Psi(X,A),
$$
where $\Psi(X,A)$ is a Boolean combination of atomic formulas in variables from $X$. We sometimes refer to universal sentences as to universal formulas.

A \index{quasi identity}\emph{quasi identity} in the language $\cL_A$  is a universal formula of the form
\[
\forall x_1 \cdots\forall x_l \left(\mathop  \bigwedge \limits_{i = 1}^m \;r_i (X,A) = 1 \rightarrow s(X, A) = 1\right),
\]
where $r_i (X,A)$ and $S(X,A)$ are terms. The collection of all quasi identities satisfied by the group $G$ is called the \emph{universal Horn theory} of $G$.

\subsection{Quadratic equations and NTQ systems}
\begin{defn}
A standard quadratic equation over the group $G$ is an equation of the one of the following forms:
\begin{equation}\label{eq:st2}
\prod_{i=1}^{n}[x_i,y_i] \prod_{i=1}^{m}z_i^{-1}c_iz_i d = 1,\ n,m \geq 0, m+n \geq 1 ;
\end{equation}
\begin{equation}\label{eq:st4}
\prod_{i=1}^{n}x_i^2 \prod_{i=1}^{m}z_i^{-1}c_iz_i d = 1, \ n,m \geq 0, n+m \geq 1.
\end{equation}
Equations of the form (\ref{eq:st2}) are called {\em orientable} of genus $n$, and equations of the form (\ref{eq:st4}) are called {\em non-orientable} of genus $n$.
\end{defn}
Let $W$ be a strictly quadratic equation (that is an equation in which every variable occurs exactly twice) over a group $G$. Then there is a $G$-automorphism $\phi \in \Aut_G(G[X])$ such that  $\phi(W)$ is a standard quadratic equation over $G$.

To each quadratic equation one can associate a punctured surface. For example, the orientable surface associated to an Equation (\ref{eq:st2}) will have genus $n$ and $m+1$ punctures.

A system of equations $S = 1$  is called {\em triangular quasi-quadratic} over $G$ if it can be partitioned into the following subsystems
$$
\begin{array}{r}
S_1(X_1, X_2, \ldots, X_n,A) = 1,\\
 S_2(X_2, \ldots, X_n,A) = 1,\\
\cdots\phantom{1,} \\
 S_n(X_n,A) = 1,
 \end{array}
 $$
where for each $i$ one of the following holds:
\begin{enumerate}
    \item[(A)]  $S_i$ is quadratic  in variables $X_i$;
    \item[(B)] \label{it:2)} $S_i= \{[y,z]=1, [y,u]=1, [z,u]=1 \mid y, z \in X_i\}$, where $u$ is a group word in $X_{i+1} \cup  \ldots \cup X_n \cup A$. In this case we say that $S_i=1$ corresponds to an extension of a centraliser;
    \item[(C)] $S_i= \{[y,z]=1 \mid y, z \in X_i\}$;
    \item[(D)] $S_i$ is the empty equation.
\end{enumerate}

In the above notation, define $G_{i}=G_{R(S_{i}, \ldots, S_n)}$ for $i = 1, \ldots, n$ and put $G_{n+1}=G.$ The triangular quasi-quadratic system $S = 1$ is called {\em non-degenerate} or simply NTQ if the following conditions hold:
 \begin{enumerate}
    \item each system $S_i=1$, where $X_{i+1}, \ldots, X_n$ are viewed as the corresponding constants from $G_{i+1}$ (under the canonical maps $X_j \rightarrow G_{i+1}$, $j = i+1, \ldots, n$)  has a solution in $G_{i+1}$;
    \item the element in $G_{i+1}$ represented by the word $u$ from case (C) above, is  not a proper power in $G_{i+1}$.
  \end{enumerate}

\section{Reducing systems of equations over $G$  to generalised equations}
\label{se:4-1}

The notion of a generalised equation was introduced by Makanin in \cite{Makanin}. A generalised equation is a combinatorial object which encodes a system of equations. The motivation for defining a generalised equation is two-fold. One the one hand, it gives an efficient way of encoding all the information about a system of equations and, on the other hand, elementary transformations, that are essential for Makanin's algorithm, see Section \ref{se:5.1}, have a cumbersome description in terms of systems of equations, but admit an intuitive one in terms of graphic representations of combinatorial generalised equations. In this sense graphic representations of generalised equations can be likened to matrices. In linear algebra there is a correspondence between systems of equations over a field $k$ and matrices with elements from $k$. To describe the set of solutions of a system of equations, one uses Gauss elimination which is usually applied to matrices, rather than systems of equations.

In \cite{Mak82}, Makanin reduced the decidability of equations over a free group to the decidability of finitely many systems of equations over a free monoid, in other words, he reduced the compatibility problem for a free group to the compatibility problem for generalised equations. In fact, Makanin essentially proved that the study of solutions of systems of equations over free groups reduces to the study of solutions of generalised equations in the following sense: every solution of the system of equations $S$ factors trough one of the solutions of one of the generalised equations and, conversely, every solution of the generalised equation extends to a solution of $S$.

A crucial fact for this reduction is that the set of solutions of a given system of equations $S$ over a free group, defines only finitely many   different cancellation schemes (cancellation trees).  By each of these cancellation trees, one can construct a generalised equation.

The goal of this section is to generalise this approach to systems of equations over a free product of groups $G$.

Though the ideas of the reduction are easy, the formal definitions of generalised equation and partition table are technical. The reader non-familiar with the definitions may wish to consult some of the examples given in \cite{CK2}.

Let $G=G_1*G_2$ be the free product of two arbitrary groups. Let $X = \{x_1, \ldots, x_n\}$ be a set of variables, let $\cA= G_1\cup G_2$ be a set of constants from $G$ and let  $G[X] = G \ast F(X)$. Recall, that we assume that $G=G_1*G_2$ for simplicity of notation only. All the definitions, proofs and result carry over to the case when $G$ is the free product of finitely many groups.

\begin{NB}
It will be convenient for us to assume that \emph{all} equations of a system $S$ over  $G$ are triangular  i.e. every equation $S_i=1$ of $S$ has the form $xy=z$, where $x, y \in X^{\pm 1}$ and $z\in X^{\pm 1} \cup \cA$. The assumption that all equations in $S$ are triangular is used in Section \ref{se:4-1}. The only reason to triangulate $S$ is to simplify the notation. It is not hard to generalise the constructions given in Section \ref{se:4-1} to the case where $S$ is an arbitrary finite system of equations.
\end{NB}

\subsection{Definition of the generalised equation}

\begin{defn}[Combinatorial Generalised Equation]
A \emph{combinatorial generalised equation} $\Omega$ (with constants from $\cA$) consists  of the following objects:
\begin{enumerate}
\item A finite set of \emph{bases} $\BS = \BS(\Omega)$.  Every base is one of the following three types:
\begin{itemize}
    \item   a $G_1$-\emph{constant base}, a $G_2$-constant base,
    \item   a $G_1$-\emph{factor base}, a $G_2$-factor base,
    \item   or a \emph{variable base}.
\end{itemize}
If no confusion arises, we will further omit the prefix $G_1$- or $G_2$-. Each constant base is associated with exactly one element from $\cA$. The set of variable bases ${\M}$ consists of $2n$ elements ${\M} = \{\mu_1, \ldots, \mu_{2n}\}$. The set ${\M}$ comes equipped with two functions: a function $\varepsilon: {\M}\rightarrow \{-1,1\}$ and an involution $\Delta: {\M} \rightarrow {\M}$ (i.e. $\Delta$ is a bijection such that $\Delta^2$ is the identity on $\M$). Bases $\mu$ and $\Delta(\mu)$ are called {\em dual bases}.

The set of $G_j$-factor bases $\NN_j$ consists of $3m$ elements, comes equipped with a function $\varepsilon: \NN_j\rightarrow \{-1,1\}$ and is partitioned into disjoint subsets of $3$ elements each: $\NN_j=\NN_{j,1}\cup \dots \cup\NN_{j,k_j}$. The sets $\NN_{j,i}$ are called \emph{sets of dual bases}, bases belonging to $\NN_{j,i}$ are numbered sequentially: $\NN_{j,i}=\{\nu_{j,i,1},\nu_{j,i,2},\nu_{j,i,3}\}$. Given a $G_j$-factor base $\nu_{j,i,k}$, there are precisely two dual to $\nu_{j,i,k}$ $G_j$-factor bases.

\item A set of \emph{boundaries} $\BD = \BD(\Omega)$. The set $\BD$ is a finite initial segment of the set of positive integers  $\BD = \{1, 2, \ldots, \rho+1\}$.

\item Two functions $\alpha : \BS \rightarrow \BD$ and $\beta : \BS \rightarrow \BD$. We call $\alpha(\mu)$ and $\beta(\mu)$ the \emph{initial} and \emph{terminal} boundaries of the base $\mu$. These functions satisfy the following conditions: $\alpha(\eta) < \beta(\eta)$ for every base $\eta \in \BS$; if $b$ is a constant base then $\beta(\eta) = \alpha(\eta) +1$.

\item Two more functions $\a: \BD\to \{0,1,2\}$ and $\b: \BD\to  \{0,1,2\}$. We call $\a(i)$ and $\b(i)$ the \emph{initial and terminal terms} of the boundary $i$.

\item A finite set of \emph{boundary connections} $\BC = \BC(\Omega)$. A ($\mu$-)boundary connection is a triple $(i,\mu,j)$, where $i, j \in \BD$, $\mu \in {\M}$ such that $\alpha(\mu) <  i < \beta(\mu)$, $\alpha(\Delta(\mu)) <  j < \beta(\Delta(\mu))$ and $\a(i)=\a(j)$, $\b(i)=\b(j)$ if $\varepsilon(\mu)\varepsilon(\Delta(\mu))=1$, and $\a(i)=\b(j)$, $\b(i)=\a(j)$ if $\varepsilon(\mu)\varepsilon(\Delta(\mu))=-1$. We assume that if $(i,\mu,j) \in \BC$ then $(j,\Delta(\mu),i) \in \BC$. This allows one to identify the boundary connections $(i,\mu,j)$ and $(j,\Delta(\mu),i)$.
\end{enumerate}
\end{defn}
Though, by the definition, a combinatorial generalised equation is a combinatorial object, it is not practical to work with combinatorial generalised equations describing its sets and functions. It is more convenient to encode all this information in its graphic representation. We refer the reader to Figure \ref{fig:ET1} for an example of a graphic representation of a generalised equation. Further examples along with the formal construction of the graphic representation can be found in \cite{CK2}.

To a combinatorial generalised equation $\Omega$, one can canonically associate:
\begin{itemize}
    \item a system of equations in \emph{variables} $h_1, \ldots, h_\rho$, $\rho=\rho_\Omega$ and coefficients from $\cA$ (variables $h_i$ are often called \emph{items}),
    \item a system of equations over $G_1$,
    \item a system of equations over $G_2$,
    \item and constraints.
\end{itemize}
This collection is called a \emph{generalised equation}, and (abusing the language) we denote it by the same symbol $\Omega$. By a combinatorial generalised equation, the equations and constrains of the generalised equation  $\Omega$ are constructed as follows.
\begin{enumerate}
    \item Each pair of dual variable bases $(\lambda, \Delta(\lambda))$ gives rise to an equation $s_\lambda=1$:
            $$
    {\left[h_{\alpha (\lambda )}h_{\alpha (\lambda )+1}\cdots h_{\beta (\lambda )-1}\right]}^ {\varepsilon (\lambda)}= {\left[h_{\alpha (\Delta (\lambda ))}h_{\alpha (\Delta (\lambda ))+1} \cdots h_{\beta (\Delta (\lambda ))-1}\right]}^{\varepsilon (\Delta (\lambda))}.
            $$
            These equations are called \emph{basic equations}.

    \item   Every set of dual bases $\NN_{j,i}$, $i=1,\dots, k_j$, $j=1,2$ gives rise to the following \emph{$G_j$-factor equation}:
            $$
            {h_{\alpha(\nu_{j,i,1})}}^{\varepsilon(\nu_{j,i,1})} \cdot  {h_{\alpha(\nu_{j,i,2})}}^{\varepsilon(\nu_{j,i,2})} \cdot
            {h_{\alpha(\nu_{j,i,3})}}^{\varepsilon(\nu_{j,i,1})}=1.
            $$
    \item For each constant base $b$ we write a \emph{coefficient equation}:
            $$
            h_{\alpha (b)}= a,
            $$
            where $a \in \cA$ is the constant associated with $b$.
    \item Every boundary connection $(p,\lambda,q)$ gives rise to a \emph{boundary equation}
            $$
            [h_{\alpha (\lambda )}h_{\alpha (\lambda)+1}\cdots h_{p-1}]= [h_{\alpha (\Delta (\lambda ))}h_{\alpha (\Delta (\lambda ))+1} \cdots h_{q-1}],
            $$
            if $\varepsilon (\lambda)= \varepsilon (\Delta(\lambda))$ and
            $$
            [h_{\alpha(\lambda )}h_{\alpha (\lambda )+1}\cdots h_{p-1}]= [h_{q}h_{q+1}\cdots h_{\beta (\Delta (\lambda))-1}]^{-1},
            $$
            if $\varepsilon(\lambda)= -\varepsilon (\Delta(\lambda))$.
    \item Every item $h_j$, $j=1,\dots,\rho$ gives rise to a \emph{type constraint}
    $$
    h_j \in \type(\a(j),\b(j+1)).
    $$
\end{enumerate}

\begin{rem}
We assume that every generalised equation comes associated with a combinatorial one.
\end{rem}

\begin{defn}
Let
$$
\Omega = \left\{
\begin{array}{lll}
L_1(h)=R_1(h),\ldots, L_k(h) = R_k(h) & - &\hbox{basic and coefficient equations}; \\
V_1(h)=1, \dots, V_m(h)=1 & - &\hbox{factor equations};\\
T_{i,j}(h)\subset \type(i,j), i,j=0,1,2 & - & \hbox{type constraints}.
\end{array}
\right\}
$$
be a generalised equation in variables $h = \{h_1, \ldots,h_{\rho}\}$ with constants from $\cA$ (here $L_1(h),R_1(h), \dots, L_k(h), R_k(h)$, $V_1,\dots,V_m$ are words in $h_1,\dots,h_\rho$ and $T_{i,j}(h)$ are subsets of the sets $\{h_1,\dots, h_\rho\}$).

A tuple $H = (H_1, \ldots, H_{\rho})$ of nontrivial elements of $G$  written in the form (\ref{eq:nf}) is a \index{solution!of a generalised equation}{\em solution} of $\Omega $ if:
\begin{enumerate}
\item  all words $L_i(H), R_i(H)$ are reduced as written (i.e. written in the form (\ref{eq:nf}));
\item  $L_i(H) \doteq  R_i(H)$ in $G$, $i = 1, \ldots k$;
\item  $V_i(H)=1$ in the corresponding factor, $i=1,\dots,m$;
\item  the elements $H_1, \ldots, H_{\rho}$ satisfy the type constraints, i.e. $T_{i,j}(H)\subset \type(i,j)$.
\end{enumerate}

The \index{length!of a solution of the generalised equation}\emph{length} of a solution $H$ is defined to be
\glossary{name={$|H|$}, description={length of a solution of a generalised equation, $|H|=\sum\limits_{i=1}^\rho |H_i|$}, sort=H}
$$
|H|=\sum\limits_{i=1}^\rho |H_i|.
$$
\end{defn}
The notation $(\Omega, H)$ means that $H$ is a solution of the generalised equation $\Omega$.

We now introduce a number of notions that we use throughout the text. Let $\Omega$ be a generalised equation.

\begin{defn}[Glossary of terms]
A boundary $i$ \index{boundary!intersects a base}\emph{intersects} the base $\mu$ if $\alpha (\mu)<i<\beta (\mu)$. A boundary $i$ \index{boundary!touches a base}\emph{touches} the base $\mu$ if $i=\alpha (\mu)$ or $i=\beta (\mu)$. A boundary is said to be \index{boundary!open}\emph{open} if it intersects at least one base, otherwise it is called \index{boundary!closed}\emph{closed}. We say that a boundary $i$ is \index{boundary!tied}\emph{tied} in a base $\mu$ (or is \index{boundary!$\mu$-tied}\emph{$\mu$-tied}) if there exists a boundary connection $(p,\mu,q)$ such that $i = p$ or $i = q$. A boundary is \index{boundary!free}\emph{free} if it does not touch any base and it is not tied by a boundary connection.

An item $h_i$ \index{item!belongs to a base}\emph{belongs} to a base $\mu$ or, equivalently, $\mu$ \index{base!contains an item}\emph{contains} $h_i$, if $\alpha (\mu)\leq i\leq \beta (\mu)-1$ (in this case we sometimes write $h_i\in \mu$). An item $h_i$ is called a \index{item!constant}\emph{$G_k$-constant} item if it belongs to a constant base $b$ and the constant $a$ associated to the base $b$ belongs to $G_k$, $k=1,2$. The item $h_i$ is called a \emph{$G_k$-factor} item if it belongs to a $G_k$-factor base, $k=1,2$, and is called a \index{item!free}\emph{free} item if it does not belong to any base. By \glossary{name={$\gamma(h_i)$, $\gamma_i$},description={the number of bases which contain $h_i$},sort=G}$\gamma(h_i) =\gamma_i$ we denote the number of bases which contain $h_i$, in this case we also say that $h_i$ is \index{item!covered $\gamma_i$ times}\emph{covered} $\gamma_i$ times. An item $h_i$ is called \index{item!linear}\emph{linear} if $\gamma_i=1$ and is called \index{item!quadratic}\emph{quadratic} if $\gamma_i=2$.

Let $\mu, \Delta(\mu)$ be a pair of dual bases such that $\alpha (\mu)=\alpha (\Delta(\mu))$ and  $\beta (\mu)=\beta (\Delta(\mu))$ in this case we say that bases $\mu$ and $\Delta(\mu)$ form \index{pair of matched bases}\emph{a pair of matched bases}. A base $\lambda$ is \index{base!contained in another base}\emph{contained} in a base $\mu$ if $\alpha(\mu) \leq \alpha(\lambda) < \beta(\lambda) \leq \beta(\mu)$. We say that two bases $\mu$ and $\nu$ \index{base!intersects another base}\emph{intersect} or \index{base!overlaps with another base}\emph{overlap}, if $[\alpha(\mu),\beta(\mu)]\cap [\alpha(\nu),\beta(\nu)]\ne \emptyset$. A base $\mu$ is called \index{base!linear}\emph{linear} if there exists an item $h_i\in\mu$ so that $h_i$ is linear.

A set of consecutive items \glossary{name={$[i,j]$}, description={section $\{h_i,\ldots, h_{j-1}\}$ of a generalised equation}, sort=S}$[i,j] = \{h_i,\ldots, h_{j-1}\}$ is called a \index{section}\emph{section}. A section is said to be \index{section!closed}\emph{closed} if the boundaries $i$ and $j$ are closed and all the boundaries between them are open. If $\mu$ is a base then by \glossary{name={$\sigma(\mu)$}, description={the section $[\alpha(\mu),\beta(\mu)]$}, sort=S}$\sigma(\mu)$ we denote the section $[\alpha(\mu),\beta(\mu)]$ and by \glossary{name={$h(\mu)$}, description={the product of items $h_{\alpha(\mu)}\ldots h_{\beta(\mu)-1}$}, sort=H}$h(\mu)$ we denote the product of items $h_{\alpha(\mu)}\ldots h_{\beta(\mu)-1}$. In general for a section $[i,j]$ by \glossary{name={$h[i,j]$}, description={the product of items $h_i \ldots h_{j-1}$}, sort=H}$h[i,j]$ we denote the product  $h_i \ldots h_{j-1}$. A base $\mu$ \index{base!belongs to a section}\emph{belongs} to a section $[i,j]$ if $i\le\alpha(\mu)<\beta(\mu)\le j$. Similarly an item $h_k$ \index{item!belongs to a section}\emph{belongs} to a section $[i,j]$ if $i\le k<j$. In these cases we write $\mu\in [i,j]$ or $h_k\in [i,j]$.

Let $H=(H_1,\dots,H_{\rho})$ be a solution of a generalised equation $\Omega$ in variables $h=\{h_1,\dots, h_\rho\}$. We use the following notation. For any word $W(h)$ in $G[h]$ set \glossary{name={$W(H)$}, description={for any word $W(h)$ in $G[h]$ and any solution $H$ of a generalised equation, $W(H)=H(W(h))$}, sort=W}$W(H)=H(W(h))$. In particular, for any base $\mu$ (section $\sigma=[i,j]$) of $\Omega$,  we have \glossary{name={$H(\mu)$}, description={if $H$ is a solution  of a generalised equation, $H(h(\mu))=H_{\alpha(\mu)} \cdots H_{\beta(\mu)-1}$}, sort=H}$H(\mu)=H(h(\mu))=H_{\alpha(\mu)} \cdots H_{\beta(\mu)-1}$ (\glossary{name={$H[i,j]$, $H(\sigma)$},description={if $H$ is a solution  of a generalised equation, $H(\sigma)=H_i\cdots H_{j-1}$}, sort=H}$H[i,j]=H(\sigma)=H(h(\sigma))=H_i\cdots H_{j-1}$, respectively).
\end{defn}

We now formulate some necessary conditions for a generalised equation to have a solution.

\begin{defn} \label{defn:formcons}
A generalised equation $\Omega$ is called \emph{formally consistent} if it satisfies the following conditions.
\begin{enumerate}
    \item \label{item:fc1} If $\varepsilon (\mu)=-\varepsilon (\Delta (\mu))$, then the bases $\mu$ and $\Delta (\mu )$ do not intersect.
     \item \label{it:forcon2}  Given two boundary connections $(p,\lambda ,q)$ and $(p_1,\lambda ,q_1)$, if $p\le p_1$, then $q\le q_1$ in the case when $\varepsilon (\lambda)\varepsilon (\Delta (\lambda))=1$, and $q\ge q_1$ in the case when $\varepsilon (\lambda)\varepsilon (\Delta (\lambda))=-1$. In particular, if $p=p_1$ then $q = q_1$.
   \item  Let $\mu$ be a base such that  $\alpha (\mu)=\alpha (\Delta (\mu))$, in other words, let $\mu$ and $\Delta(\mu)$ be a pair of matched bases. If $(p,\mu ,q)$ is a $\mu$-boundary connection  then  $p=q$.
    \item \label{item:fc4} Let $\eta_1$ be either a $G_i$-constant base or a $G_i$-factor base, $i=1,2$. Similarly, let $\eta_2$ be either a $G_j$-constant base or a $G_j$-factor base, $j=1,2$. If $\alpha(\eta_1)=\alpha(\eta_2)$ then $i=j$.
    \item \label{item:clean} Let $\eta$ be a $G_j$-factor base or a $G_j$-constant base. If $\eta$ is contained in a base $\mu$ and $(i,\mu,q_1)$, $(i+1,\mu,q_2)$ are boundary connections, then $|q_1- q_2|=1$.
    \item  \label{item:fc6} Let $\mu$ be a variable base or a $G_j$-factor base. If $\varepsilon (\mu)=-\varepsilon (\Delta (\mu))$, then $\a(\alpha(\mu))=\b(\beta(\Delta(\mu)))$ and $\b(\beta(\mu))=\a(\alpha(\Delta(\mu)))$. If $\varepsilon (\mu)=\varepsilon (\Delta (\mu))$, then $\a(\alpha(\mu))=\a(\alpha(\Delta(\mu)))$ and $\b(\beta(\mu))=\b(\beta(\Delta(\mu)))$.
    \item  \label{item:fc7} Let $\nu$ be a $G_j$-factor base, $\nu\in \NN_{j,i}$. If $\varepsilon (\nu)=-\varepsilon (\nu')$, $\nu'\in\NN_{j,i}$, then $\a(\alpha(\nu))=\b(\beta(\nu'))$ and $\b(\beta(\nu))=\a(\alpha(\nu'))$. If $\varepsilon (\nu)=\varepsilon (\nu')$, then $\a(\alpha(\nu))=\a(\alpha(\nu'))$ and $\b(\beta(\nu))=\b(\beta(\nu'))$.
  \end{enumerate}
\end{defn}

\begin{lem} \label{le:4.1} \
\begin{enumerate}
    \item \label{it:fcl1} If a generalised equation $\Omega$ has a solution, then $\Omega$ is formally consistent;
    \item \label{it:fcl2} There is an algorithm to check whether or not a given generalised equation is formally consistent.
\end{enumerate}
\end{lem}
\begin{proof}
It is easy to see that if $\Omega$ has a solution then it satisfies conditions (\ref{item:fc1})-(\ref{item:fc4}), (\ref{item:fc6}), (\ref{item:fc7}) of Definition \ref{defn:formcons}. Condition (\ref{item:clean}) follows from statement (\ref{item:le14-3}) of Lemma \ref{le:R1}, since if for a given $G_i$-factor base or $G_i$-constant base $\nu$ we have $\beta(\nu)-\alpha(\nu)\ge 2$, the equality $U=P(H)$ is not graphical.
\end{proof}

\begin{rem}
We further consider only formally consistent generalised equations.
\end{rem}

\subsection{Reduction to generalised equations: constructing generalised equations by a system of equations over $G$} \label{sec:consgeneq}
In this section, we show how for a given finite system of triangular equations $S(X,\cA) = 1$ over $G$ one can canonically associate a finite collection  of generalised equations ${\GE}(S)$ with constants from $\cA$, which to some extent describes all solutions of the system $S(X,\cA) = 1$.

Let $S(X,\cA) = 1$ be a finite system of triangular equations $S_ 1 =1,\ldots, S_\m =1$  over the group $G$. Write $S(X,\cA) = 1$ in the form
\begin{equation}\label{*}
\begin{array}{c}
 r_{11}r_{12} r_{13}=1,\\
 r_{21}r_{22} r_{23}=1,\\
 \ldots \\
 r_{\m1}r_{\m2}r_{\m3}=1,\\
\end{array}
\end{equation}
where $r_{ij}$ are letters in the alphabet $ X^{\pm 1}\cup \cA^{\pm 1}$.

We aim to define a combinatorial object called a partition table, that encodes a particular type of cancellation that happens when one substitutes a  solution $U\in G^n$ into $S(X,\cA)=1$ and then reduces the words in $S(U, \cA)$ to the empty word. To acquire an intuition on how the partition tables are constructed, the reader may consult Section \ref{sec:relsol2ge}.

Informally, a partition table is a collection of words. Each word corresponds to an occurrence of a variable in the system $S$. Letters from $Z$ (see the definition below) encode the pieces of the variable that cancel freely. The letters from $T_i$ encode cancellation in the factor $G_i$.

\begin{defn} \label{defn:partab}
A \emph{partition table} $\T$  of $S(X,\cA)$  is  a set of reduced words
$$
\T = \{V_{i,j}(Z, T_1, T_2)\}, \ 1\leq i\leq \m,\ 1\leq j\leq 3
$$
from the free group $F(Z\cup T_1 \cup T_2)$, where $Z = \{z_1,\ldots ,z_p\}$, $p\leq 3{\m}$,
$T_1=\{t_{1,1,1},t_{1,1,2},t_{1,1,3},t_{1,2,1}\ldots,t_{1,\m,1},t_{1,\m,2}, t_{1,\m,3}\}$, $T_2=\{t_{2,1,1},t_{2,1,2},t_{2,1,3},t_{2,2,1}\ldots,t_{2,\m,1},t_{2,\m,2}, t_{2,\m,3}\}$, which satisfy the conditions below.
(Note, that the triple index $(i,j,k)$ of $t$ is decoded as follows:
$$
\hbox{(factor $G_1$ or $G_2$, sequential number of the triangle, number of the edge in the triangle)},
$$
see Section \ref{sec:relsol2ge} for details.)
\begin{enumerate}
    \item \label{it:defpt1} The equality $V_i\doteq V_{i1}V_{i2}V_{i3}=1$, $1\leq i\leq \m$ holds in $\factor{F(Z\cup T_1\cup T_2)}{\ncl\langle T_1, T_2\rangle}$.
    \item \label{it:defpt2} Every letter from $Z^{\pm 1}\cup T_1 \cup T_2$ occurs at most once in the word $V$. Moreover, if a letter $l \in Z^{\pm 1}$ occurs in $V_{ij}$, then the letter $l^{-1}$ does not occur in $V_{ij}$. The letters from $T_1^{-1}\cup T_2^{-1}$ do not occur in $V$.
    \item \label{it:defpt3} There are no subwords of the form $t_{k,j_1,i_1}t_{k,j_2,i_2}$ in the words $V_{ij}$.
    \item \label{it:defpt4} For every subword of $V_i$ of the form $t_{k_1,j_1,i_1} W t_{k_2,j_2,i_2}$, where $W \in F(Z)$, the following holds
    $$
    (k_1=k_2 \hbox{ and } j_1=j_2) \hbox{ if and only if } W=1.
    $$
    \item \label{it:defpt5} if $r_{ij} =a \in \cA$, then $|V_{ij}|=1$.
\end{enumerate}
\end{defn}

\begin{rem}
Since the number of letters in $\{Z\cup T_1\cup T_2\}$ is bounded, conditions (\ref{it:defpt2}), (\ref{it:defpt3}) and (\ref{it:defpt4}) of Definition \ref{defn:partab} imply that the length $|V_{ij}|$ of $V_{ij}$ is bounded above by a function of the system $S(X,\cA)=1$.
\end{rem}

\begin{lem} \label{le:4.2}
Let $S(X,\cA) = 1$ be a finite system of equations over $G$. Then
\begin{enumerate}
    \item the set $\PT(S)$ of all partition tables  of the system $S$ is finite, and its cardinality is bounded above by a number which depends only on $\m$;
    \item one can effectively construct the set $\PT(S)$.
\end{enumerate}
\end{lem}
\begin{proof}
Since the words $V_{ij}$ have  bounded length,  one can effectively construct the finite set  of all collections of words $\{V_{ij}\}$ in $F(Z\cup T_1\cup T_2)$ which satisfy conditions (\ref{it:defpt2}), (\ref{it:defpt3}), (\ref{it:defpt5}) of Definition \ref{defn:partab}.  Now for each such collection $\{V_{ij}\}$,  one can effectively check  whether or not conditions (\ref{it:defpt1}) and (\ref{it:defpt4}) hold. This allows one to list all partition tables for $S(X,\cA) = 1$.
\end{proof}

\medskip

By a partition table $\T=\{V_{ij}\}$ we construct a generalised equation $\Omega _\T$ in the following way. Consider the following word $V$ in the free monoid $M(\cA \cup Z^{\pm 1}\cup T_1^{\pm 1}\cup T_2^{\pm 1})$
$$
V\doteq V_{11}V_{12} V_{13}\cdots V_{\m 1}V_{\m 2}V_{\m 3} = y_1 \ldots y_\rho,
$$
where $y_i \in \cA \cup Z^{\pm 1}\cup T_1^{\pm 1}\cup T_2^{\pm 1}$ and $\rho = |V|$ is the length of $V$. Then the generalised equation $\Omega_\T = \Omega_\T(h)$ has $\rho + 1$ boundaries and $\rho$ variables $h_1,\ldots ,h_{\rho}$, $h = \{h_1,\ldots,h_{\rho}\}$.

Now we define the bases of $\Omega_\T$ and the functions $\alpha, \beta, \a, \b, \varepsilon$.

Put $\a(i)=\b(i)=0$ for all $i$.

Let $z \in Z$.  For the (unique) pair of distinct occurrences $y_i$ and $y_j$, $i<j$ of  $z$ in $V$
$$
y_i = z^{\epsilon _i}, \quad y_j = z^{\epsilon _j} \quad \epsilon _i, \epsilon _j \in \{-1,1\}
$$
we introduce a pair of dual variable bases $\eta_{z,i,j}, \Delta(\eta_{z,i,j})$ . Put
\begin{gather}\notag
\begin{split}
\alpha(\eta_{z,i,j}) = i, \quad &\beta(\eta_{z,i,j}) = i+1, \quad \varepsilon(\eta_{z,i,j}) = \epsilon _i;\\
& \alpha(\Delta(\eta_{z,i,j})) = j, \quad \beta(\Delta(\eta_{z,i,j})) = j+1, \quad \varepsilon(\Delta(\eta_{z,i,j})) = \epsilon _j.
\end{split}
\end{gather}
The basic equation that corresponds to this pair of dual bases is $h_{i}^{\epsilon_i}=h_{j}^{\epsilon _j}$.

Let $x \in X$.  For any two distinct occurrences of  $x$ in $S(X,\cA) = 1$:
$$
r_{i,j} = x^{\epsilon_{ij}}, \ r_{s,t} = x^{\epsilon_{st}}; \ \epsilon _{ij}, \epsilon _{st} \in \{-1,1\},
$$
so that $(i,j)$ precedes $(s,t)$ in left-lexicographical order, we introduce a pair of dual bases $\mu_{x,q}$ and $\Delta(\mu_{x,q})$, $q=(i,j,s,t)$.
Now suppose that $V_{ij}$ and $V_{st}$ occur in the word $V$ as subwords
$$
V_{ij} = y_{c_1} \ldots y_{d_1}, \quad V_{st} = y_{c_2} \ldots y_{d_2} \hbox{ correspondingly}.
$$
Then we put
\begin{gather}\notag
\begin{split}
\alpha(\mu_{x,q}) = {c_1}, \quad &\beta(\mu_{x,q}) = d_1+1, \quad \varepsilon(\mu_{x,q}) = \epsilon_{ij},\\
&\alpha(\Delta(\mu_{x,q})) = {c_2}, \quad \beta(\Delta(\mu_{x,q})) = d_2+1, \quad \varepsilon(\Delta(\mu_{x,q})) = \epsilon_{st}.
\end{split}
\end{gather}
The basic equation which corresponds to this pair of dual bases is
$$
\left(h_{\alpha(\mu_{x,q})}\ldots h_{\beta(\mu_{x,q})-1}\right)^{\epsilon _{ij}}= \left(h_{\alpha(\Delta(\mu_{x,q}))}\ldots h_{\beta(\Delta(\mu_{x,q}))- 1}\right)^{\epsilon _{st}}.
$$

For the (unique) triple of occurrences
$$
y_{n_1} = t_{k,j,1},\  y_{n_2} = t_{k,j,2}, \  y_{n_3} = t_{k,j,3},
$$
of $t_{k,j,1}$, $t_{k,j,2}$ and $t_{k,j,3}$ in $V$, we introduce a set of dual $G_k$-factor bases $\nu_{t_{k,j,1},n},\nu_{t_{k,j,2},n},\nu_{t_{k,j,3},n}$, where $n=(n_1,n_2,n_3)$.

Put
\begin{gather}\notag
\begin{split}
\alpha(\nu_{t_{k,j,m},n}) = n_m, &\quad \beta(\nu_{t_{k,j,m},n})= n_m+1,\quad \varepsilon(\nu_{t_{k,j,m},n})= 1, \\
&\a(n_m)=k, \b(n_m)=3-k, \quad \a(n_{m}+1)=3-k, \b(n_{m}+1)=k, \hbox{ where } m=1,2,3.
\end{split}
\end{gather}

The $G_k$-factor equation that corresponds to this set of dual bases is
$$
h_{n_1}h_{n_2}h_{n_3}=1.
$$

Let $r_{ij} = a \in \cA$ be an element of $G_k$.  In this case we introduce a constant base $\mu_{ij}$ with the \emph{label} $a$. If $V_{ij}$ occurs in $V$ as $V_{ij} = y_l$, then we put
\begin{gather}\notag
\begin{split}
\alpha(\mu_{ij}) = l, \ \beta(\mu_{ij}) = l+1,& \ \a(l)=k, \ \b(l)=3-k, \\
&\a(l+1)=3-k, \ \b(l+1)=k.
\end{split}
\end{gather}
The corresponding coefficient equation is $h_{c}=a$.

For every base $\mu$ so that $\a(\alpha(\mu))\ne 0$ (or $\b(\beta(\mu))\ne 0$) we set:
\begin{itemize}
    \item $\a(\alpha(\mu))=\b(\beta(\Delta(\mu)))$ (or $\b(\beta(\mu))=\a(\alpha(\Delta(\mu)))$), if $\varepsilon (\mu)=-\varepsilon (\Delta (\mu))$;
    \item $\a(\alpha(\mu))=\a(\alpha(\Delta(\mu)))$ (or $\b(\beta(\mu))=\b(\beta(\Delta(\mu)))$), if $\varepsilon (\mu)=\varepsilon (\Delta (\mu))$.
\end{itemize}

The list of boundary connections (and hence the boundary equations) is empty. This defines the generalised equation
$\Omega_\T$. Put
$$
{\GE}(S) = \{\Omega_\T \mid \T \hbox{ is a partition  table for } S(X,\cA)= 1 \}.
$$
Then ${\GE}(S)$ is a finite collection of generalised equations which can be effectively constructed for a given system $S(X,\cA) = 1$.

\subsection{Relating solutions of $S(X,\cA)=1$ to the solutions of generalised equations}
\label{sec:relsol2ge}
We shall need an analogue of the notion of a cancellation scheme for free groups.

Consider the free product of groups $G=G_1*G_2$. Take the multiplication table presentations of $G_1$ and $G_2$:
$$
G_1=\langle G_1 \mid R_1 \rangle; \quad G_2=\langle G_2 \mid R_2 \rangle.
$$
Then the Cayley graph $\Cal(G)$ of $G=\langle G_1, G_2 \mid R_1, R_2 \rangle$ can be obtained as follows. Take the Cayley graph $\Cal(G_1)$ of $G_1$. To each vertex $g_1$ of $\Cal(G_1)$ we attach the Cayley graph $\Cal(G_2)$ of $G_2$ by identifying $g_1$ and the vertex of $\Cal(G_2)$ corresponding to the identity in $G_2$. To each vertex $g_2\ne 1$ of each copy of $\Cal(G_2)$ just attached, we attach $\Cal(G_1)$ by identifying $g_2$ and the vertex of $\Cal(G_1)$ corresponding to the identity in $G_1$, and so on. Summarising, the Cayley graph $\Cal(G)$ is a tree-graded space with respect to the \emph{pieces}: copies of $\Cal(G_1)$ and $\Cal(G_2)$
(we refer the reader to \cite{TrGr} for definition and basic properties of tree-graded spaces).

Let $U\in G^n$, $U=(U_1,\dots, U_n)$ be a solution of $S(X)=S(X,\cA) =\{S_ 1 =1,\ldots, S_\m =1\}$. We write $S(X,\cA)$ in the form (\ref{*}) and substitute $U$ in (\ref{*}) graphically, denote this substitution by $\varphi: X\to U$. Then $S_i(U)$ represents the trivial element in $G$, thus describes a cycle $\cc_i$ in $\Cal(G)$ based at the identity element of $G$. Since $\Cal(G)$ is a tree-graded space, the intersection of the cycle $\cc_i$ with any piece of $\Cal(G)$ is either empty or a point or a non-trivial cycle.

Below we treat $\cc_i$ as a cycle labelled by $r_{i1}r_{i2}r_{i3}$, where the $r_{ij}$'s are letters in $X^{\pm 1}\cup \cA^{\pm 1}$. Every term in $\varphi(r_{ij})$ labels an edge of either a  $2$-cycle or a $3$-cycle. From the structure of a tree-graded space on $\Cal(G)$ and since the words $\varphi(r_{ij})$ are reduced, it follows that distinct edges of any cycle belong to different words $\varphi(r_{ij})$. Therefore, the structure of the cancellation in $\varphi(r_{i1}r_{i2}r_{i3})$ is described by one of the five cancellation schemes shown on Figure \ref{fig:cancsc}.

\begin{figure}[h!tb]
  \centering
   \includegraphics[keepaspectratio,width=6in]{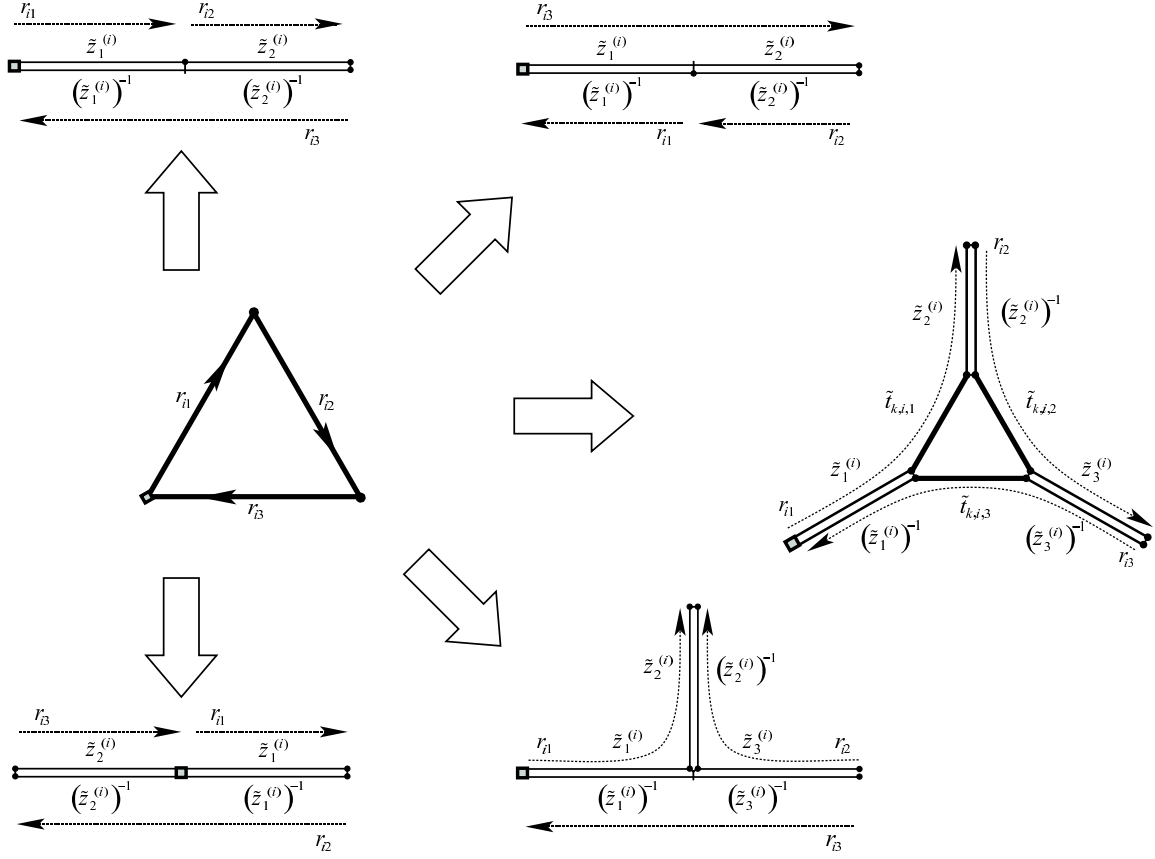}
\caption{Cancellation schemes for $\varphi(r_{i1}r_{i2}r_{i3})$.}\label{fig:cancsc}
\end{figure}

We denote the label (a $G_k$-term) of an edge of the triangle (see Figure \ref{fig:cancsc}) by the letter ${\tilde t_{k,i,l}}$, $l=1,2,3$. The parts of the word  $\varphi(r_{i1}r_{i2}r_{i3})$ that ``cancel freely'' when reducing the word $S_i(U)$ to the identity, are denoted by the letters ${\tilde{z}^{(i)}}_{j_1,j_2}$, $1\le j_1,j_2\le 3$ and their inverses, see Figure \ref{fig:cancsc} (here we write $^{(i)}$ to indicate that ${\tilde{z}^{(i)}_{j_1,j_2}}$ corresponds to the $i$-th equation of the system $S(X)$). Formally, if two terms label a $2$-cycle we call them \emph{dual}. These two terms belong to two distinct words $\varphi(r_{ij_1})$ and $\varphi(r_{ij_2})$, $j_1<j_2$ correspondingly. Denote by $\tilde{z}_{j_1,j_2}^{(i)}$ the maximal non-trivial sequence of consecutive terms of $\varphi(r_{ij_1})$ such that
\begin{itemize}
    \item every term of the sequence labels an edge of a $2$-cycle,
    \item their duals are terms of $\varphi({r_{ij_2}}^{(i)})$;
\end{itemize}
and by ${\left(\tilde{z}_{j_1,j_2}^{(i)}\right)}^{-1}$ the sequence of their duals in $\varphi({r_{ij_2}^{(i)}})$. Notice that the structure of a tree-graded space on $\Cal(G)$ and the fact that the words $\varphi(r_{ij})$ are reduced, imply that for every pair $\varphi(r_{ij_1})$ and $\varphi(r_{ij_2})$ there exists at most one maximal non-trivial sequence of consecutive terms of $\varphi(r_{ij_1})$ with duals in $\varphi(r_{ij_2})$.

It follows that we can write $\varphi(r_{ij})$ as a word in the $\tilde t_{k,i,l}$'s and the $\tilde{z}^{(i)}_{j_1,j_2}$'s and their inverses:
$$
\varphi(r_{ij})=V_{ij}(\tilde{z}^{(i)}_1, \tilde{z}^{(i)}_2, \tilde{z}^{(i)}_3, {\tilde t}_{1,i,j}, {\tilde t}_{2,i,j})
$$
Straightforward verification shows that the set $\T=\{V_{ij}\}$ is a partition table of $S(X,\cA)=1$.

Obviously,
$$
H = ({\tilde z}_1, \ldots, {\tilde z}_{3\m},{\tilde t}_{1,1,1},\ldots,{\tilde t}_{1,\m,3},{\tilde t}_{2,1,1},\dots, {\tilde t}_{2,\m,3})
$$
is a solution of the generalised equation $\Omega_\T$ induced by $U$.

\medskip

\begin{defn}
For a generalised equation $\Omega$ in variables $h$ we can consider the same system of equations over the group $G$. We denote this system by \glossary{name={${\Omega}^*$}, description={generalised equation $\Omega$ treated as a system of equations over the group $G$}, sort=O}
${\Omega}^*$. In other words, if
$$
\Omega = \left\{
\begin{array}{lll}
L_1(h)=R_1(h),\ldots, L_k(h) = R_k(h) & - &\hbox{basic and coefficient equations}; \\
V_1(h)=1, \dots, V_m(h)=1 & - &\hbox{factor equations};\\
T_{i,j}(h)\subset \type(i,j), i,j=0,1,2 & - & \hbox{type constraints}.
\end{array}
\right\}
$$
is a generalised equation, then by $\Omega^*$ we denote the system of equations
$$
\Omega^*= \{L_1(h)R_1(h)^{-1} = 1, \ldots, L_s(h)R_s(h)^{-1} = 1, V_1(h)=1, \dots, V_m(h)=1\}
$$
over the group $G$.

Let
\glossary{name={$G_{R(\Omega^\ast)}$}, description={coordinate group of the generalised equation $\Omega$}, sort=G}
$$
G_{R(\Omega^\ast)}= \factor{G\ast F(h)}{R({\Omega}^*)}\, ,
$$
where $F(h)$ is a free group with basis $h$. We call $G_{R(\Omega^\ast)}$ the \index{coordinate group!of the generalised equation}\emph{coordinate group} of the generalised equation $\Omega$.
\end{defn}

Note that every solution of a generalised equation  gives rise to a homomorphism from $G_{R(\Omega^\ast)}$ to $G$, but the converse is not true. Indeed, the variables that occur in $G_k$-factor equations must take values in the group $G_k$, $k=1,2$ and the items have to satisfy the type constraints - obviously, this is not true for an arbitrary homomorphism from $G_{R(\Omega^\ast)}$ to $G$.

Now we explain the relations between the coordinate groups of the system $S(X,\cA) = 1$ and of the generalised equation $\Omega_\T^\ast$.

We use the notation from Definition \ref{defn:partab}. For a letter $x$ in $X$, we choose an arbitrary  occurrence $r_{ij} = x^{\epsilon_{ij}}$ of $x$ in $S(X,\cA) = 1$. Let $\mu = \mu_{x,q}$, $q=(i,j,s,t)$ be the base that corresponds to the pair of occurrences $r_{ij} = x^{\epsilon_{ij}}$ and $r_{st} = x^{\epsilon_{st}}$ of $x$.  Then $V_{ij}$ occurs in $V$ as a subword
$$
V_{ij} = y_{\alpha(\mu)} \ldots y_{\beta(\mu) -1}.
$$
Define the word $P_x(h) \in G[h]$, where $h = \{h_1, \ldots,h_\rho\}$, as follows:
$$
P_x(h,\cA) ={( h_{\alpha(\mu)} \ldots h_{\beta(\mu)-1})}^{\epsilon_{ij}},
$$
and put
$$
P(h) = (P_{x_1}, \ldots, P_{x_n}).
$$
The tuple of words $P(h)$ depends on the choice of occurrences of letters from $X$ in $V$. It follows from the construction above that the map $X \to G[h]$ defined by $x \mapsto  P_x(h,\cA)$ gives rise to a $G$-homomorphism
\begin{equation}\label{eq:hompt}
\pi : G_{R(S)}\rightarrow G_{R(\Omega _\T^\ast)}.
\end{equation}
Indeed,  if $f(X) \in R(S)$ then $\pi(f(X))= f(P(h))$. It follows from the definition of partition table that $f(P(h)) = 1$ in $G_{R(\Omega_\T^\ast)}$, thus $f(P(h))\in R(\Omega_\T^*)$. Therefore $R(f(S))\subseteq R(\Omega_\T^*)$ and $\pi$ is a homomorphism.

Observe that the image  $\pi (x)$ in $G_{R(\Omega _\T^*)}$ does not depend on a particular choice of the occurrence of $x$ in $S(X,\cA)$ (the basic equations of $\Omega_\T$ make these images equal). Hence $\pi$ depends only on $\Omega_\T$. Thus, every solution $H$ of the generalised equation gives rise to a solution $U$ of $S$ so that $\pi_U=\pi\pi_H$

\begin{rem} \label{rem:epi}
Let $\T = \{V_{i,j}(Z, T_1, T_2)\}$ be a partition table of the  system of equations $S(X,\cA)$ that satisfies the following condition
\begin{equation}\label{3.31}
z_k, t_{l,m,n} \in \langle V_{i,j} \rangle,\ k=1,\dots,p;\ l=1,2;\  m=1,\dots \m;\  n=1,2,3.
\end{equation}
Then for the generalised equation $\Omega_\T$, the homomorphism $\pi_{\Omega_\T}:G_{R(S)} \to G_{R(\Omega^*_\T)}$ induced by the map $x_i\mapsto P_{x_i}(h,\cA)$ defined above is surjective.
\end{rem}

We summarise the discussion above in the following lemma.

\begin{lem} \label{le:R1}
For a system of equations $S(X,\cA)=1$ over the group $G$,  one can  effectively  construct a finite set
$$
{\GE}(S) = \{\Omega_\T \mid \T \hbox{ is  a partition table for } S(X,\cA)= 1\}
$$
of generalised equations  such  that
\begin{enumerate}
    \item if the set ${\GE}(S)$ is empty, then  $S(X,\cA)= 1$ has no solutions in $G$;
    \item for each $\Omega (h) \in {\GE}(S)$ and  for each $x \in X$  one can effectively find a word $P_x(h,\cA) \in G[h]$ of length at most $|h|$ such that the map $x\mapsto P_x(h,\cA)$, $x \in X$ gives rise to a $G$-homomorphism $\pi_\Omega: G_{R(S)}\rightarrow G_{R(\Omega^\ast)}$;
    \item \label{item:le14-3} for any solution $U  \in G^n$ of the system $S(X,\cA)=1$ there exists $\Omega (h) \in {\GE}(S)$ and a solution $H$ of $\Omega$ such that $U = P(H)$, where $P(h) = (P_{x_1}, \ldots, P_{x_n})$, and this equality is graphical, $U\doteq P(H)$;
\end{enumerate}
\end{lem}

\begin{cor} \label{co:R1}
In the notation of {\rm Lemma \ref{le:R1}},  for any solution $U  \in G^n$ of the system $S(X,\cA)=1$ there exist a generalised equation  $\Omega (h) \in {\GE}(S)$  and a solution $H$ of $\Omega(h)$ such that the following diagram commutes:
$$
\xymatrix@C-1em{
 G_{R(S)}  \ar[rd]_{\pi_U} \ar[rr]^{\pi_\Omega}  &  &G_{R(\Omega^* )} \ar[ld]^{\pi_H}
                                                                             \\
                               &  G &
}$$

Conversely, for every generalised equation $\Omega\in {\GE}(S)$  and a solution $H$ of $\Omega$, there exists a solution $U \in G^n$ of the systems $S(X, \cA)=1$ such that the above diagram commutes.
\end{cor}

\section{The Process} \label{se:5}

In the previous section we reduced the study of the set of solutions of a system of equations over a free product of groups to the study of solutions of generalised equations. In order to describe the solutions of generalised equations, in this section we describe a branching rewriting process.

The original version of this process for free monoids was introduced by G.~Makanin in his important papers \cite{Makanin} and \cite{Mak82}, where he solves the compatibility problem of systems of equations over a free monoid (over a free group). In his work \cite{Razborov1}, \cite{Razborov3}, A.~Razborov gave a complete description of all solutions of a system of equations over a free group. The process was further developed in many different directions by O.~Kharlampovich and A.~Miasnikov, see \cite{KMIrc}, \cite{IFT}, \cite{EJSJ}.

\bigskip

The process we describe is a rewriting system based on the ``divide and conquer'' algorithm design paradigm.

For a given generalised equation $\Omega_{v_0}$, this branching process results in a locally finite and possibly infinite oriented rooted at $v_0$ tree $T$, $T=T(\Omega_{v_0})$. The vertices of the tree $T$ are labelled by generalised equations $\Omega_{v_i}$. The edges of the tree $T$ are labelled by epimorphisms of the corresponding coordinate groups. Moreover, for every solution $H$ of $\Omega_{v_0}$, there exists a path in the tree $T$ from the root vertex to a vertex $v_l$ and a solution $H^{(l)}$ of $\Omega_{v_l}$ such that the solution $H$ is a composition of the epimorphisms corresponding to the edges in the tree and the solution  $H^{(l)}$.  Conversely, every path from the root to a vertex $v_l$ in $T$ and any solution $H^{(l)}$ of $\Omega_{v_l}$ give rise to a solution of $\Omega_{v_0}$, see Proposition \ref{prop:TO}.

The tree is constructed by induction on the height. Let $v$ be a vertex of height $n$. One can check under the assumptions of which of the 15 cases described in Section \ref{se:5.2} the generalised equation $\Omega_v$ falls. If $\Omega_v$ falls under the assumptions of Case 1 or Case 2, then $v$ is a leaf of the tree $T$. Otherwise, using the combination of elementary and derived transformations (defined in Sections \ref{se:5.1} and \ref{se:5.2half}) given in the description of the corresponding (to $v$) case, one constructs finitely many generalised equations and epimorphisms from the coordinate group of $\Omega_v$ to the coordinate groups of the generalised equations constructed.

We finish this section (see Lemma \ref{3.2}) by proving that infinite branches of the tree $T$, as in the case of free groups, correspond to one of the following three cases:
\begin{enumerate}
\item[(A)] Case 7-10: Linear case (Levitt type, thin type);
\item[(B)] Case 12: Quadratic case (surface type, interval exchange type);
\item[(C)] Case 15: General case (toral type, axial type).
\end{enumerate}

\subsection{Preliminary definitions}
\label{se:parametric}
\begin{defn} \label{de:gepar}
Let $\Omega$ be a generalised equation. If the set \glossary{name={$\Sigma$, $\Sigma(\Omega)$}, description={the set of all closed sections of a generalised equation $\Omega$}, sort=S}$\Sigma = \Sigma(\Omega)$ of all closed sections  of $\Omega$ is partitioned into a disjoint union of subsets
\glossary{name={$V\Sigma$}, description={the set of all variable sections of a generalised equation}, sort=V}
\glossary{name={$C\Sigma$}, description={the set of all constant sections of a generalised equation}, sort=C}
\glossary{name={$F\Sigma$}, description={the set of all free sections of a generalised equation}, sort=F}
\glossary{name={$G_i\Sigma$}, description={the set of all $G_i$-sections of a generalised equation}, sort=G}
$$
\Sigma(\Omega) = V\Sigma\cup G_1\Sigma\cup G_2\Sigma \cup F\Sigma.
$$
Sections from $V\Sigma$, $G_1\Sigma$, $G_2\Sigma$ and $F\Sigma$ are called correspondingly, \index{section!variable}\emph{variable}, \index{section!$G_i$-}$G_1$-, $G_2$-, and \index{section!free}\emph{free} sections. We sometimes refer to $G_1$- and $G_2$-sections as to \index{section!constant}\emph{constant} sections. Set $C\Sigma=G_1\Sigma\cup G_2\Sigma$.
\end{defn}

To organise the process properly, we partition the closed sections of $\Omega$ in another way into a disjoint union of two disjoint sets, which we refer to as \index{part of a generalised equation!non-active}\index{part of a generalised equation!active}\emph{parts}:
\glossary{name={$A\Sigma$}, description={the active part of a generalised equation}, sort=A}
\glossary{name={$NA\Sigma$}, description={the non-active part of a generalised equation}, sort=N}
\begin{equation}  \label{eq:geq1-1}
\Sigma(\Omega) = A\Sigma\cup NA\Sigma
\end{equation}
Sections from $A\Sigma$ are called \index{section!active}{\em active}, and sections from $NA\Sigma$ are called \index{section!non-active}{\em non-active}. We set
$$
C\Sigma \cup F\Sigma\subseteq NA\Sigma .
$$
If not stated otherwise, we assume that all sections from  $V\Sigma$ belong to the active part $A\Sigma$.

If $\sigma \in \Sigma$, then every item (or base) from $\sigma$ is called \index{item!$G_i$-}\index{base!$G_i$-} \index{item!free}\index{base!free} \index{item!active} \index{base!active}\index{item!non-active}\index{base!non-active} active, non-active, $G_1$-, $G_2$-, constant, free item (base), depending on the type of $\sigma$. Note the abuse of terminology: free item has already been defined as an item that does not belong to any base.

\begin{defn}
We say that a generalised equation $\Omega$ is in the \index{standard form of a generalised equation}\emph{standard form} if the following conditions hold.
\begin{enumerate}
    \item All non-active sections from $V\Sigma\cap NA\Sigma$ are located to the right of all active sections, all $G_1$-sections are located to the right  of all non-active sections from $V\Sigma\cap NA\Sigma$, all $G_2$-sections are located to the right of all $G_1$-sections and all free sections are located to the right of all $G_2$-sections. Formally, there are numbers
        \glossary{name={$\rho_A$}, description={the boundary between active and non-active parts of a generalised equation}, sort=R}
        $$
        1 \leq \rho_A  \leq  \rho_{NA} \leq \rho_{G_1} \leq\rho_{G_2} \leq \rho_F=\rho = \rho_\Omega
        $$
        such that $[1,\rho_A +1]$, $[\rho_A +1, \rho_{NA}+1]$, $[\rho_{NA}+1, \rho_{G_1} +1]$, $[\rho_{G_1} +1, \rho_{G_2} +1]$ and $[\rho_{G_2} +1, \rho+1]$ are, correspondingly, unions of all active, all variable non-active, all $G_1$-, all $G_2$-, and all free sections.
    \item All $G_i$-constant bases  and $G_i$-factor bases belong to $G_i\Sigma$, $i=1,2$, and for every letter $a \in \cA$ there is at most one constant base in $\Omega$ labelled by $a$, if there exists a base labelled by $a\in \cA$, then there are no bases labelled by $a^{-1}$.
    \item Every free variable $h_i$ of $\Omega$ belongs to a section from $F\Sigma$.
 \end{enumerate}
 \end{defn}

We will show in Lemma \ref{lem:stform} that every generalised equation can be taken to the standard form.

\subsection{Elementary transformations} \label{se:5.1}

In this section we describe \index{transformation of a generalised equation!elementary}{\em elementary transformations} of  generalised equations. Recall that we consider only formally consistent generalised equations. In general, an elementary transformation $\ET$ associates  to a generalised equation $\Omega$ a finite set of generalised equations $\ET(\Omega) = \left\{\Omega_1, \dots, \Omega_r\right\}$  and a collection of surjective homomorphisms $\theta_i: G_{R(\Omega^\ast)}\rightarrow G_{R({\Omega _i}^\ast)}$ such that for every pair $(\Omega,H)$ there exists a unique pair $(\Omega_i, H^{(i)})$ such that the following diagram commutes.
$$
\xymatrix@C3em{
 G_{R(\Omega^\ast)}  \ar[rd]_{\pi_H} \ar[rr]^{\theta_i}  &  &G_{R(\Omega_i^\ast )} \ar[ld]^{\pi_{H^{(i)}}}
                                                                             \\
                               &  G &
}
$$
Since the pair $(\Omega_i,H^{(i)})$ is defined uniquely, we have a well-defined map $\ET:(\Omega,H) \to (\Omega_i,H^{(i)})$. It is a good exercise to understand what how the elementary transformations change the system of equations corresponding to the generalised equation.

\subsubsection*{\glossary{name={$\ET 1$}, description={elementary transformation $\ET 1$}, sort=E}$\ET 1$: Cutting a base}
Suppose that $\Omega$ contains a boundary connection $(p,\lambda ,q )$.

The transformation $\ET 1$ carries $\Omega$ into a single generalised equation $\Omega_1$ which is obtained from $\Omega$ as follows. To obtain $\Omega_1$ from $\Omega$ we
\begin{itemize}
\item replace (cut in $p$) the base $\lambda$ by two new bases  $\lambda _1$ and $\lambda _2$, and
\item replace (cut in $q$) $\Delta(\lambda)$ by two new bases $\Delta (\lambda _1)$ and $\Delta (\lambda _2)$,
\end{itemize}
so that the following conditions hold.

If $\varepsilon(\lambda) =  \varepsilon(\Delta(\lambda))$, then
\begin{gather}\notag
\begin{split}
&\alpha(\lambda_1) = \alpha(\lambda), \ \beta(\lambda_1) = p, \quad \alpha(\lambda_2) = p, \  \beta(\lambda_2) = \beta(\lambda);
\\
\alpha(\Delta(\lambda_1))= \alpha(\Delta(\lambda)),&\ \beta(\Delta(\lambda_1))= q, \quad \alpha(\Delta(\lambda_2)) = q, \ \beta(\Delta(\lambda_2))= \beta(\Delta(\lambda));
\end{split}
\end{gather}
If  $\varepsilon(\lambda) = - \varepsilon(\Delta(\lambda))$, then
\begin{gather}\notag
\begin{split}
&\alpha(\lambda_1) = \alpha(\lambda), \ \beta(\lambda_1) = p, \quad \alpha(\lambda_2) = p, \  \beta(\lambda_2) = \beta(\lambda);
\\
\alpha(\Delta(\lambda_1)) = q,  &\ \beta(\Delta(\lambda_1)) =\beta(\Delta(\lambda)), \quad \alpha(\Delta(\lambda_2)) =\alpha(\Delta(\lambda)),  \  \beta(\Delta(\lambda_2)) = q;
\end{split}
\end{gather}

Put
$$
\varepsilon(\lambda_i) = \varepsilon(\lambda), \ \varepsilon(\Delta(\lambda_i)) = \varepsilon(\Delta(\lambda)), i= 1,2.
$$

Let $(p', \lambda, q')$ be  a boundary connection in $\Omega$. If $p' < p$,  then we replace $(p', \lambda, q')$ by  $(p',\lambda_1, q')$. If $p' > p$, then we replace   $(p', \lambda, q')$ by  $(p',\lambda_2, q')$.

Notice that from property (\ref{it:forcon2}) of Definition \ref{defn:formcons}, it follows that $(p',\lambda_1, q')$ (or  $(p',\lambda_2, q')$) is a boundary connection in the new generalised equation. The resulting generalised equation $\Omega_1$ is formally consistent.   Put $\ET 1(\Omega) =
\{\Omega_1\}$, see Figure \ref{fig:ET1}.

\begin{figure}[h!tb]
  \centering
   \includegraphics[keepaspectratio,width=6in]{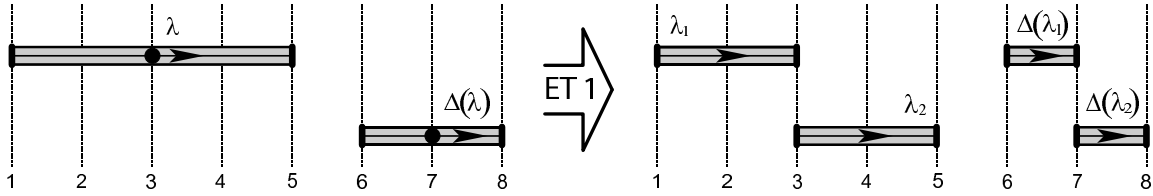}
\caption{Elementary transformation $\ET 1$: Cutting a base.} \label{fig:ET1}
\end{figure}

\subsubsection*{\glossary{name={$\ET 2$}, description={elementary transformation $\ET 2$}, sort=E}$\ET 2$: Transferring a base}
Let a base $\lambda$ of a generalised equation $\Omega$  be contained in the base $\mu$, i.e. $\alpha (\mu)\leq \alpha (\lambda)<\beta (\lambda)\leq\beta(\mu)$. Suppose that the boundaries $\alpha(\lambda)$ and $\beta(\lambda)$ are $\mu$-tied, i.e. there are boundary connections of the form $(\alpha (\lambda),\mu, q_1)$ and $(\beta (\lambda),\mu,q_2)$.  Suppose also that every $\lambda$-tied boundary is $\mu$-tied.

The transformation $\ET 2$ carries $\Omega$ into a single generalised equation $\Omega_1$ which is obtained from $\Omega$ as follows. We transfer $\lambda$ from the base $\mu$ to the base $\Delta (\mu)$ and adjust all the basic and boundary equations (see Figure \ref{fig:ET2}). Formally, we replace $\lambda$ by a new base $\lambda^\prime$ such that $\alpha(\lambda^\prime) = q_1, \beta(\lambda^\prime) = q_2$ and replace each $\lambda$-boundary connection $(p,\lambda,q)$ with a new one $(p^\prime,\lambda^\prime,q)$ where $p$ and $p^\prime$ are related by a $\mu$-boundary connection $(p,\mu, p^\prime)$.

Set $\ET 2(\Omega)=\{\Omega_1\}$.

\begin{figure}[h!tb]
  \centering
   \includegraphics[keepaspectratio,width=6in]{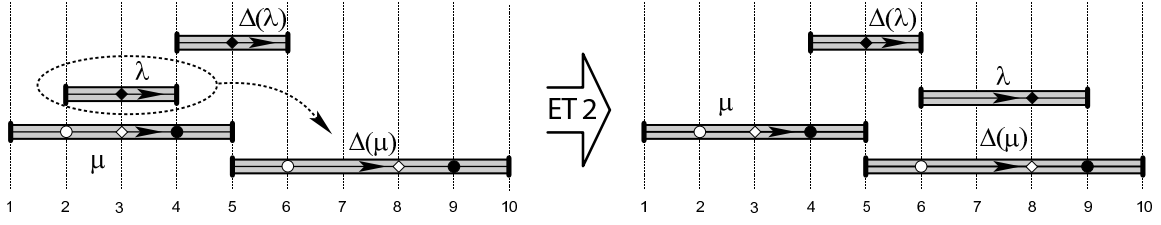}
\caption{Elementary transformation $\ET 2$: Transferring a base.} \label{fig:ET2}
\end{figure}

\subsubsection*{\glossary{name={$\ET 3$}, description={elementary transformation $\ET 3$}, sort=E}$\ET 3$: Removing a pair of matched bases}

Let $\mu$ and $\Delta(\mu)$ be a pair of matched bases in $\Omega$. Since $\Omega$ is formally consistent, one has $\varepsilon(\mu) = \varepsilon(\Delta(\mu))$, $\beta(\mu) = \beta(\Delta(\mu))$ and every $\mu$-boundary connection is of the form $(p,\mu,p)$.

The transformation $\ET 3$ applied to $\Omega$ results in a single generalised equation $\Omega_1$ which is obtained from $\Omega$ by
removing the pair of bases  $\mu, \Delta(\mu)$ with all the $\mu$-boundary connections, see Figure \ref{fig:ET3}.

\begin{figure}[h!tb]
  \centering
   \includegraphics[keepaspectratio,width=6in]{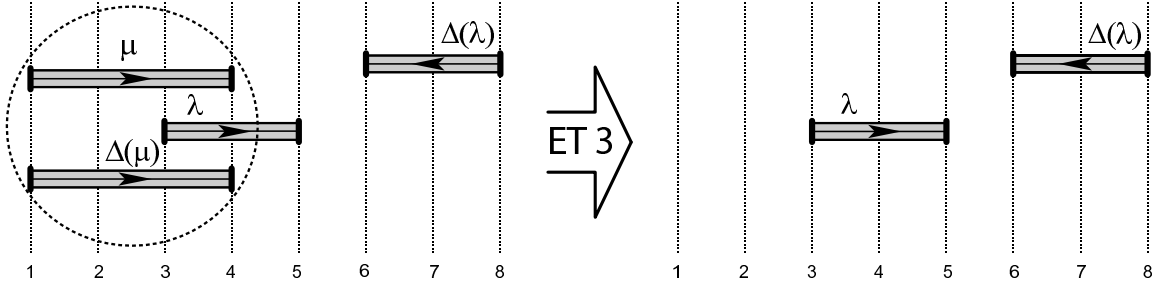}
\caption{Elementary transformation $\ET 3$: Removing a pair of matched bases.} \label{fig:ET3}
\end{figure}

\begin{rem} \label{rem:et123}
Observe that for $i = 1,2,3$, the set $\ET i(\Omega)$ consists of a single generalised equation $\Omega_1$,  such that  $\Omega$ and $\Omega_1$ have the same set of variables $h$. The identity map $G[h] \rightarrow G[h]$ induces a $G$-isomorphism $G_{R(\Omega^\ast )} \rightarrow G_{R({\Omega^{\prime}}^\ast)}$. Moreover, $H$ is a solution of $\Omega$ if and only if $H$ is a solution of $\Omega_1$.
\end{rem}

\subsubsection*{\glossary{name={$\ET 4$}, description={elementary transformation $\ET 4$}, sort=E}$\ET 4$: Removing a linear base}

Suppose that in $\Omega$ a variable base $\mu$ does not intersect any other base, i.e. the items $h_{\alpha(\mu)}, \ldots, h_{\beta(\mu)-1}$ are contained only in one variable base $\mu$. Moreover, suppose that all boundaries that intersect $\mu$ are $\mu$-tied, i.e. for every $i$, $\alpha(\mu)< i< \beta(\mu)$ there exists a boundary $\t(i)$ such that $(i,\mu ,\t(i))$ is a boundary connection in $\Omega$. Since $\Omega$ is formally consistent, we have $\t(\alpha(\mu)) = \alpha(\Delta(\mu))$ and  $\t(\beta(\mu)) = \beta(\Delta(\mu))$ if $\varepsilon(\mu)\varepsilon(\Delta(\mu)) = 1$, and $\t(\alpha(\mu)) = \beta(\Delta(\mu))$ and $\t(\beta(\mu)) = \alpha(\Delta(\mu))$ if $\varepsilon(\mu)\varepsilon(\Delta(\mu)) = -1$.

The transformation $\ET 4$ carries $\Omega $ into a single generalised equation $\Omega_1$ which is obtained from $\Omega$ by deleting the pair of bases $\mu$ and $\Delta(\mu)$; deleting all the boundaries $\alpha(\mu)+1, \ldots, \beta(\mu)-1$, deleting all $\mu$-boundary connections, re-enumerating the remaining boundaries and adjusting the initial and terminal terms as appropriate, see Remark \ref{rem:boundterm} (see Figure \ref{fig:ET4}).

\begin{figure}[!h]
  \centering
   \includegraphics[keepaspectratio,width=6in]{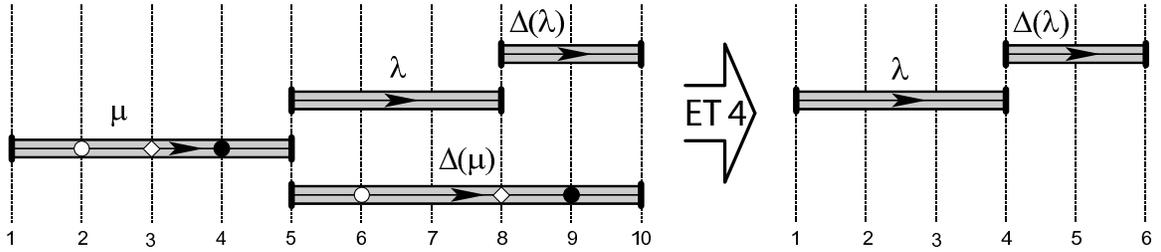}
\caption{Elementary transformation $\ET 4$: Removing a linear base.} \label{fig:ET4}
\end{figure}

We define the homomorphism $\theta_1: G_{R(\Omega^\ast )} \rightarrow G_{R({\Omega_1}^\ast)}$, where $h^{(1)}$ is the set of variables of $\Omega_1$, as follows:
$$
\theta_1(h_j)=
\left\{
  \begin{array}{ll}
h_j^{(1)},                            & \hbox{if $j<\alpha (\mu)$ or $j\ge \beta (\mu)$;} \\
h^{(1)}_{\t(j)} \dots h^{(1)}_{\t(j+1)-1},   & \hbox{if $\alpha +1 \leq j\leq\beta (\mu)-1$ and  $\varepsilon (\mu)=\varepsilon (\Delta(\mu))$;} \\
h^{(1)}_{\t(j-1)}\dots h^{(1)}_{\t(j+1)},  & \hbox{if $\alpha +1 \leq j\leq\beta (\mu)-1$ and $\varepsilon (\mu)=-\varepsilon (\Delta(\mu))$.}
\end{array}
\right.
$$
It is not hard to see that $\theta_1$ is a $G$-isomorphism.

\begin{rem} \label{rem:boundterm}
Every time when we transport or delete a closed section $[i,j]$ (see $\ET 4$, $\D 2$, $\D 5$, $\D 6$), we re-enumerate the boundaries and adjust the initial and terminal terms of the boundary $i$ as follows.

The re-enumeration of the boundaries is defined by the correspondence $\mathcal{C}:\BD(\Omega)\to \BD(\Omega_1)$, $\Omega_1\in\ET 4(\Omega)$:
$$
\mathcal{C}:k\mapsto \left\{
           \begin{array}{ll}
             k, & \hbox{if $k\le i$;} \\
             k-(j-i), & \hbox{if $k\ge j$.}
           \end{array}
         \right.
$$
Define $\a(\mathcal{C}(k))=\left\{
                              \begin{array}{ll}
                                \a(k), & \hbox{if $k\ne i,j$;} \\
                                \a(j), & \hbox{if $k=i$ or $k=j$;}
                              \end{array}
                            \right.$
and $\b(\mathcal{C}(k))=\left\{
                              \begin{array}{ll}
                                \b(k), & \hbox{if $k\ne i,j$;} \\
                                \b(i), & \hbox{if $k=i$ or $k=j$.}
                              \end{array}
                            \right.$
\end{rem}

\subsubsection*{\glossary{name={$\ET 5$}, description={elementary transformation $\ET 5$}, sort=E}$\ET 5$: Introducing a new boundary}

Suppose that a boundary $p$ intersects a base $\mu$ and $p$ is not $\mu$-tied.

The transformation $\ET 5$ $\mu$-ties the boundary $p$ in all possible ways, producing  finitely many different generalised equations. To this end, let $q$ be a boundary intersecting $\Delta (\mu)$. Then we perform one of the following two transformations (see Figure \ref{fig:ET5}):

\begin{figure}[!h]
  \centering
   \includegraphics[keepaspectratio,width=6in]{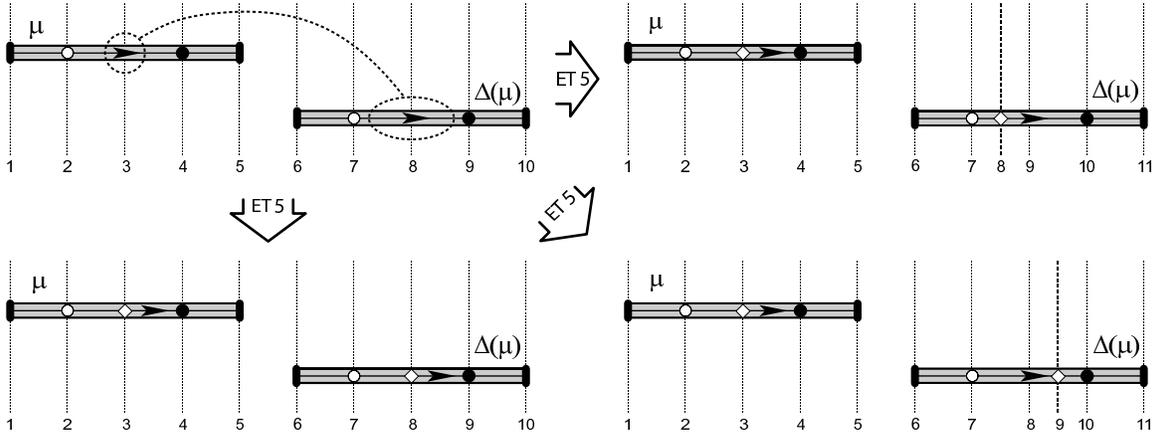}
\caption{Elementary transformation $\ET 5$: Introducing a new boundary.} \label{fig:ET5}
\end{figure}

\begin{enumerate}
\item Introduce the boundary connection $(p,\mu ,q)$, provided that the resulting equation $\Omega_q$ is formally consistent.

The identity map from $G[h]$ to $G[h]$ induces a $G$-epimorphism $\theta_q$ from $G_{R(\Omega^\ast)}$ to $G_{R({\Omega_q}^\ast)}$.

Observe that $\theta_q$ is not necessarily an isomorphism. More precisely, $\theta_q$ is not an isomorphism whenever the boundary equation corresponding to the boundary connection $(p,\mu ,q)$ does not belong to $R(\Omega^\ast)$. In this case $\theta_q$ is a proper epimorphism.

\item Introduce a new boundary $q^\prime$ between $q$ and $q+1$; introduce a new boundary connection $(p,\mu,q^\prime)$ in $\Omega$ (and introduce the initial and terminal terms of $q^\prime$ to satisfy the conditions of the definition of a boundary connection). Denote the resulting generalised equation by $\Omega_{q'}$. The corresponding $G$-homomorphism  $\theta_{q^\prime}: G_{R(\Omega^*)}\to G_{R(\Omega _{q^{\prime}}^*)}$ is induced by the map (below we assume that the boundaries of $\Omega_{q'}$ are labelled $1,2,\dots,q,q',q+1,\dots, \rho_{\Omega_{q^\prime}}+1$):
$$
\theta_{q'} (h_i)=
\left\{
  \begin{array}{ll}
    h_i, & \hbox{ if $i\neq q$;} \\
    h_{q'-1}h_{q'}, & \hbox{ if $i=q$.}
  \end{array}
\right.
$$
Observe that $\theta _{q^\prime}$ is a $G$-isomorphism.
\end{enumerate}

\begin{lem}\label{le:hom-check}
Suppose that the universal Horn theory (the theory of quasi-identities)  of $G$ is decidable. Let $\Omega_1\in \{\Omega_i\}=\ET(\Omega)$ be a generalised equation obtained from $\Omega$ by an elementary transformation $\ET$ and let $\theta_1: G_{R(\Omega^*)}\to G_{R(\Omega_1^*)}$ be the corresponding epimorphism. There exists an algorithm which determines whether or not the epimorphism $\theta_1$ is a proper epimorphism.
\end{lem}
\begin{proof}
The only non-trivial case is when $\ET = \ET 5$ and no new boundaries were introduced. In this case $\Omega_1$ is obtained from $\Omega$ by adding a new boundary equation $s = 1$, which is effectively determined by $\Omega$ and $\Omega_1$.
In this event, the coordinate group
$$
G_{R({\Omega_1}^\ast)} = G_{R({\{\Omega^* \cup \{s\}\}})}
$$
is a quotient of the group $G_{R(\Omega^\ast)}$. The homomorphism $\theta_1$ is an isomorphism if and only if $R(\Omega^\ast) = R({\Omega^* \cup \{s\}})$, or, equivalently, $s \in R(\Omega^\ast)$. The latter condition holds if and only if $s$ vanishes on all solutions of the system of equations $\Omega^\ast = 1$ in $G$, i.e. if the following universal formula (quasi identity) holds in $G$:
$$
\forall x_1 \ldots \forall x_\rho (\Omega^\ast (x_1, \ldots, x_\rho) = 1 \rightarrow s(x_1, \ldots, x_\rho) = 1).
$$
\end{proof}

\begin{rem}
Note that in \cite{DL} V.~Diekert and M.~Lohrey prove that if the universal theories of  the factors $G_1$ and $G_2$ are decidable, so is the universal theory of $G$.
\end{rem}

\subsection{Derived transformations} \label{se:5.2half}

In this section we describe several useful transformations of generalised equations. Some of them are finite sequences of elementary transformations, others result in equivalent generalised equations but cannot be presented as a composition of finitely many elementary transformations.

In general, a \index{transformation of a generalised equation!derived}\emph{derived transformation} $\D$ associates  to a generalised equation $\Omega$ a finite set of generalised equations $\D(\Omega) = \left\{\Omega_{1}, \dots, \Omega_{r}\right\}$,  and a collection of surjective homomorphisms $\theta_i: G_{R(\Omega^\ast)}\to G_{R({\Omega _i}^\ast)}$ such that for every pair $(\Omega, H)$ there exists a unique pair $(\Omega_i, H^{(i)})$ such that the following diagram commutes.
$$
\xymatrix@C3em{
 G_{R(\Omega^*)}  \ar[rd]_{\pi_H} \ar[rr]^{\theta_i}  &  &G_{R(\Omega_i^*)} \ar[ld]^{\pi_{H^{(i)}}}
                                                                             \\
                               &  G &
}
$$
Since the pair $(\Omega_i,H^{(i)})$ is defined uniquely, we have a well-defined map $\D:(\Omega,H) \to (\Omega_i,H^{(i)})$.

\subsubsection*{\glossary{name={$\D 1$}, description={derived transformation $\D 1$}, sort=D}$\D 1:$ Closing a section}

Let $\sigma=[i,j]$ be a section of $\Omega $. The transformation $\D 1$ makes the section $\sigma$ closed. To perform $\D 1$, using  $\ET 5$, we $\mu$-tie the boundary $i$ (the boundary $j$) in every base $\mu$ containing $i$ ($j$, respectively). Using $\ET 1$, we cut all the bases containing $i$ (or $j$) in the boundary $i$ (or in $j$), see Figure \ref{fig:D1}.

\begin{figure}[!h]
  \centering
   \includegraphics[keepaspectratio,width=6in]{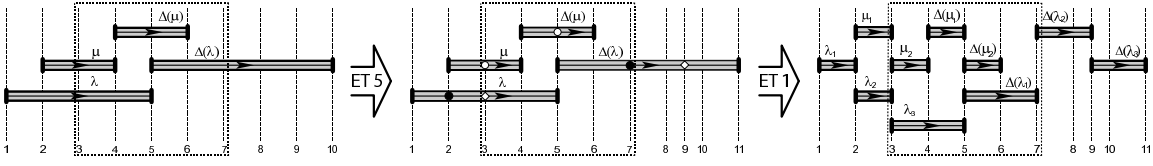}
\caption{Derived transformation $\D 1$: Closing a section.} \label{fig:D1}
\end{figure}

\subsubsection*{\glossary{name={$\D 2$}, description={derived transformation $\D 2$}, sort=D}$\D 2:$ Transporting a closed section}

Let $\sigma$ be a closed section of a generalised equation $\Omega$. The derived transformation $\D 2$ takes $\Omega$ to a single generalised equation $\Omega_1$ obtained from $\Omega$ by cutting $\sigma$ out from the interval $[1,\rho_\Omega+1]$ together with all the bases, and boundary connections on $\sigma$ and moving $\sigma$ to the end of the interval or between two consecutive closed sections of $\Omega$, see Figure \ref{fig:D2}. Then we re-enumerate all the items and boundaries of the obtained generalised equation and adjust the initial and terminal terms of the boundaries as appropriate, see Remark \ref{rem:boundterm}. Clearly, the original equation $\Omega$ and the new one $\Omega_1$ have the same solution sets and their coordinate groups are isomorphic (the isomorphism is induced by a permutation of the variables $h$).

\begin{figure}[!h]
  \centering
   \includegraphics[keepaspectratio,width=6in]{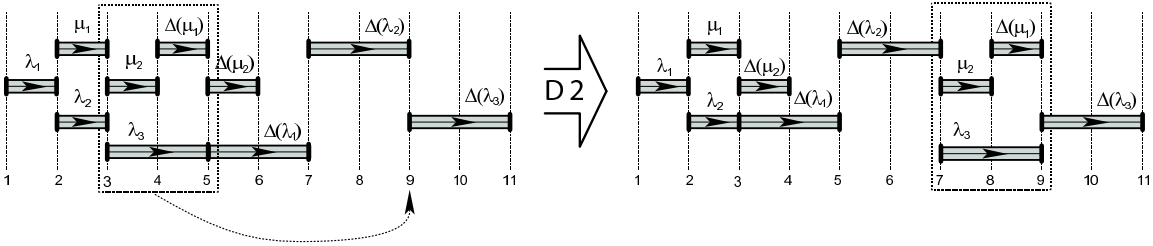}
 \caption{Derived transformation $\D 2$: Transporting a closed section.} \label{fig:D2}
\end{figure}

\subsubsection*{\glossary{name={$\D 3$}, description={derived transformation $\D 3$}, sort=D}$\D 3:$ Completing a cut}

Let $\Omega $ be a generalised equation. The derived transformation $\D 3$ carries $\Omega$ into a single generalised equation \glossary{name={$\widetilde{\Omega}$}, description={the generalised equation obtained from $\Omega$ by $\D 3$}, sort=O}$\widetilde{\Omega}=\Omega_1$ by applying $\ET 1$ to all boundary connections in $\Omega$. Formally, for every boundary connection $(p,\mu,q)$ in $\Omega$ we cut the base $\mu$ in $p$  applying $\ET 1$. Clearly, the generalised equation $\widetilde\Omega$ does not depend on the choice of a sequence of transformations $\ET 1$. Since, by Remark \ref{rem:et123}, $\ET 1$ induces the identity isomorphism between the coordinate groups, equations $\Omega$ and $\widetilde\Omega$ have isomorphic coordinate groups and the isomorphism is induced by the identity map $G[h] \to G[h]$.

\subsubsection*{\glossary{name={$\D 4$}, description={derived transformation $\D 4$}, sort=D}$\D 4:$ \index{kernel of a generalised equation}Kernel of a generalised equation}
The derived transformation $\D 4$ applied to a generalised equation $\Omega$ results in a single generalised equation \glossary{name={$\Ker(\Omega)$}, description={kernel of a generalised equation}, sort=K}$\Ker(\Omega)=\D 4(\Omega)$ constructed below.

Applying $\D 3$, if necessary, one can assume that a generalised equation $\Omega$ does not contain boundary connections. An active variable base $\mu  \in A\Sigma(\Omega)$ is called \index{base!eliminable}\emph{eliminable} if at least one of the following holds:
\begin{enumerate}
    \item [(a)] $\mu$ contains an  item $h_i$ with $\gamma(h_i)=1$;
    \item [(b)] at least one of the boundaries $\alpha (\mu),\beta (\mu)$ is different from $1,\rho +1 $ and it does not touch any other base (except $\mu$).
\end{enumerate}

An \index{elimination process for a generalised equation}\emph{elimination  process} for $\Omega$ consists of consecutive removals (\index{elimination of a base}\emph{eliminations}) of  eliminable bases until there are no eliminable bases left in the generalised equation. The resulting generalised equation is called the \emph{kernel} of $\Omega$ and we denote it by $\Ker(\Omega)$. It is easy to see that $\Ker(\Omega)$ does not depend on the choice of the elimination process. Indeed, if $\Omega$ has two different eliminable bases $\mu_1$, $\mu_2$, and elimination of  $\mu_i$ results in a generalised equation $\Omega_i$ then  by induction on the number of eliminations $\Ker(\Omega_i)$ is defined uniquely for $i = 1,2$. Obviously, $\mu_1$ is still eliminable in $\Omega_2$, as well as $\mu_2$ is eliminable in $\Omega_1$. Now eliminating $\mu_1$ and $\mu_2$ from $\Omega_2$ and $\Omega_1$ respectively, we get the same generalised equation $\Omega_0$. By induction $\Ker(\Omega_1) = \Ker(\Omega_0) = \Ker(\Omega_2)$, hence the result.

We say that a variable $h_i$ \index{item!belongs to the kernel}{\em belongs to the kernel}, if $h_i$ either belongs to at least one base in the kernel, or is constant.

We now establish a relation between the coordinate group of the generalised equation $\Omega$ and the coordinate group of its kernel $\Ker(\Omega)$.

For a generalised equation $\Omega$ we denote by $\overline{\Omega}$  the generalised equation which is obtained from $\Omega$ by deleting all free variables in $\Omega$. Obviously,
$$
G_{R(\Omega^\ast)} = G_{R({\overline {\Omega}}^\ast)} \ast F(Y)
$$
where $Y$ is the set of free variables in $\Omega$.

Let us consider what happens in the elimination process on the level of coordinate groups.

We start with the case  when just one base is eliminated. Let $\mu$ be an eliminable base in $\Omega = \Omega(h_1, \ldots, h_\rho)$. Denote by $\Omega_1$  the generalised equation obtained from $\Omega$ by eliminating $\mu$ .

\begin{enumerate}
\item  \label{item:7-10c1} Suppose  $h_i \in \mu$ and $\gamma(h_i) = 1$, i.e. $\mu$ falls under the assumption (a) of the definition of an eliminable base. Then the variable $h_i$ occurs only once in $\Omega$ (in the equation $h(\mu)^{\varepsilon(\mu)}=h(\Delta(\mu))^{\varepsilon(\Delta(\mu))}$ corresponding to the base $\mu$, denote this equation by $s_\mu$).  Therefore, in the coordinate group $G_{R(\Omega^\ast )}$ the relation $s_\mu$ can be written as $h_i = w_\mu$, where $w_\mu$ does not contain $h_i$. Consider the generalised equation $\Omega'$ obtained from $\Omega$ by deleting the equation $s_\mu$ and the item $h_i$. The presentation of the coordinate group of $\Omega'$ is obtained from the presentation of the coordinate group of $\Omega$ using a sequence of Tietze transformations, thus these groups are isomorphic.

It follows immediately that
    $$
    G_{R({\Omega_1}^\ast)} \simeq  G_{R({\Omega^\prime}^\ast)} \ast \langle h_i \rangle.
    $$
We therefore get that
    \begin{equation} \label{eq:ker1}
        G_{R(\Omega^\ast)} \simeq G_{R({\Omega^\prime}^\ast)} \simeq G_{R({\overline {\Omega}_1}^\ast)} \ast F(Z),
    \end{equation}
    where $F(Z)$ is the free group generated by the free variables of $\Omega_1$. Note that all the groups and equations which occur above can be found effectively.

\item \label{item:7-10c2} Suppose now that $\mu$ falls under the assumptions of case b) of the definition of an eliminable base, with respect to a boundary $i$. Then in all the equations of the generalised equation $\Omega$ but $s_\mu$, the variable $h_{i-1}$ occurs as a part of the subword $(h_{i-1}h_i)^{\pm 1}$. In the equation $s_\mu$ the variable $h_{i-1}$ either occurs once or it occurs precisely twice and in this event the second occurrence of $h_{i-1}$ (in $\Delta(\mu)$) is a part of the subword $(h_{i-1}h_i)^{\pm 1}$.
It is obvious that the tuple
    $$
    (h_1, \ldots, h_{i-2},h_{i-1}, h_{i-1}h_i, h_{i+1}, \ldots,h_\rho)=(h_1, \ldots, h_{i-2},h_{i-1}, h_i', h_{i+1}, \ldots,h_\rho)
    $$
forms a basis of the ambient free group generated by $(h_1,\ldots, h_\rho)$. We further assume that $\Omega$ is a system of equations in these variables. Notice, that in the new basis, the variable $h_{i-1}$ occurs only once in the generalised equation $\Omega$, in the equation $s_\mu$. Hence, as in case (\ref{item:7-10c1}), the relation $s_\mu$ can be written as $h_{i-1} = w_\mu$, where $w_\mu$ does not contain $h_{i-1}$.

Therefore, if we eliminate the relation $s_\mu$, then the set $Y = (h_1, \ldots, h_{i-2}, h_i', h_{i+1}, \ldots,h_{\rho})$ is a generating set of $G_{R(\Omega^\ast)}$. This shows that
    $$
    G_{R(\Omega^*)} \simeq \factor{G\ast F(Y \cup \cA)}{R(w_\lambda(Y) \mid \lambda \neq \mu)} \simeq G_{R({\Omega^\prime}^*)},
    $$
    where $\Omega^\prime$ is the generalised equation obtained from $\Omega_1$ by deleting the boundary $i$, and $G_{R(\Omega^*)}$, $G_{R({\Omega^\prime}^*)}$ are defined as in Section \ref{sec:relsol2ge}. Denote by $\Omega^{\prime \prime}$ an equation obtained from $\Omega^\prime$ by adding a free variable $z$ to the free part $F\Sigma$ of $\Omega^\prime$. It now follows that
    $$
        G_{R({\Omega_1}^\ast)} \simeq  G_{R({\Omega^{\prime \prime}}^\ast)} \simeq  G_{R({\Omega^\prime}^\ast)} \ast \langle z \rangle
    $$
    and
    \begin{equation}\label{eq:ker2}
        G_{R(\Omega^\ast)} \simeq G_{R({\overline{\Omega^{\prime\prime}}^\ast})} \ast F(Z),
    \end{equation}
where $F(Z)$ is the free group generated by the free items of $\Omega^{\prime\prime}$ distinct from $z$. Notice that all the groups and equations which occur above can be found effectively.
\end{enumerate}

By induction on the number of steps in the elimination process we obtain the following lemma.
\begin{lem}\label{7-10}
In the above notation,
$$
G_{R(\Omega^\ast)} \simeq G_{R({\overline {\Ker(\Omega)}}^\ast)} \ast F(Z),
$$
where $F(Z)$ is the free group generated by a set of free items from ${\Ker(\Omega)}$. Moreover, all the groups and equations which occur above can be found effectively.
\end{lem}
\begin{proof} Let
$$
\Omega = \Omega_0 \rightarrow \Omega_1 \rightarrow \ldots \rightarrow \Omega_l = \Ker(\Omega),
$$
where $\Omega_{j+1}$ is obtained from $\Omega_j$ by eliminating an eliminable base, $j = 0, \ldots,l-1$. It is easy to see (by induction on $l$)  that for every $j = 0, \ldots,l-1$
$$
\overline{\Ker(\Omega_j)} = \overline{\Ker (\overline{\Omega_j})}.
$$
Moreover, if $\Omega_{j+1}$ is obtained from $\Omega_j$ as in case (\ref{item:7-10c2}) above, then (in the above notation)
$$
\overline{(\Ker (\Omega_j))_1} = \overline{\Ker(\Omega_j^{\prime\prime})} .
$$
The statement of the lemma now follows from the remarks above and Equations (\ref{eq:ker1}) and (\ref{eq:ker2}). Note that the cardinality of the set $Z$ is the number of free variables of $\Ker(\Omega)$ minus $l$.
\end{proof}

\subsubsection*{\glossary{name={$\D 5$}, description={derived transformation $\D 5$}, sort=D}$\D 5:$ \index{transformation of a generalised equation!entire}Entire transformation}

In order to define the derived transformation $\D 5$ we need to introduce several notions.  A base $\mu$ of the generalised equation $\Omega$ is called a \index{base!leading}{\em leading} base if $\alpha(\mu)=1$. A leading base is called a \index{base!carrier}\index{carrier}{\em carrier} if for any other leading base $\lambda$ we have $\beta (\lambda)\leq \beta (\mu)$. Let $\mu $ be a carrier base of $\Omega$. Any active base $\lambda \neq \mu$ with $\beta(\lambda )\leq \beta (\mu )$ is called a \index{base!transfer}{\em transfer} base (with respect to $\mu$).

Suppose now that $\Omega$ is a generalised equation with $\gamma(h_i)\geq 2$ for each $h_i$ in the active part of $\Omega$. The \emph{entire transformation} is the following sequence of elementary transformations.

We fix a carrier base $\mu$ of $\Omega$. For any transfer base $\lambda$ we $\mu$-tie (applying $\ET 5$) all boundaries that intersect $\lambda$. Using $\ET 2$ we transfer all transfer bases from $\mu$ onto $\Delta (\mu)$. Now, there exists some $k<\beta (\mu)$ such that $h_1,\ldots ,h_k$ belong to only one base $\mu,$ while  $h_{k+1}$ belongs to at least two bases. Applying $\ET 1$ we cut $\mu$ in the boundary $k+1$. Finally, applying $\ET 4$, we delete the section $[1,k+1]$, see Figure \ref{quadratic}.

Let $\rho_A$ be the boundary between the active and non-active parts of $\Omega$, i.e. $[1,\rho_A]$ is the active part $A\Sigma$ of $\Omega$ and $[\rho_A,\rho+1]$ is the non-active  part $NA\Sigma$ of $\Omega$.

As we will see later, roughly speaking, the process has two main sub-processes: first it eliminates redundant equations (cleaning), and then applies the entire transformation, which is the motor of the process. We now introduce the excess of a solution of the generalised equation and prove that it is invariant under the entire transformation. This fact will be used later, in Section \ref{5.5.3}.

\begin{defn}\label{defn:excess}
For a pair $(\Omega,H)$, we introduce the following notation
\glossary{name={$d_{A\Sigma}(H)$}, description={length of the active part of the solution, $d_{A\Sigma}(H)=\sum \limits_{i=1}^{\rho_A-1}|H_i|$}, sort=D}
\begin{equation}\label{3.12}
d_{A\Sigma}(H)=\sum \limits_{i=1}^{\rho_A-1}|H_i|,
\end{equation}
\glossary{name={$\psi_{A\Sigma}(H)$}, description={excess of the solution, $\psi_{A\Sigma}(H)=\sum \limits_{\mu\in \omega_1}|H({\mu})|-2d_{A\Sigma}(H)$}, sort=P}
\begin{equation}\label{3.13}
\psi_{A\Sigma}(H)=\sum \limits_{\mu\in \omega_1}|H({\mu})|-2d_{A\Sigma}(H),
\end{equation}
where \glossary{name={$\omega_1$}, description={the set of all variable bases $\nu $ for which either $\nu$ or $\Delta (\nu)$ belongs to the active part of a generalised equation}, sort=O}$\omega _1$ is the set of all variable bases $\nu$ for which either $\nu$ or $\Delta (\nu)$ belongs to the active part $[1,\rho_A]$ of $\Omega$.

We call the number $\psi_{A\Sigma}(H)$ the \index{excess of the solution}\emph{excess} of the solution $H$ of the generalised equation $\Omega$.
\end{defn}

Since, by assumption, every item $h_i$ of the section $[1,\rho_A]$ belongs to at least two bases and each of these bases belongs to $A\Sigma$, hence $\psi_{A\Sigma}(H)\ge 0$.

Notice, that if for every item $h_i$ of the section $[1,\rho_A]$ one has $\gamma(h_i)=2$, for every solution $H$ of $\Omega$ the excess $\psi_{A\Sigma}(H)=0$.  Informally, in some sense, the excess of $H$ measures how far the generalised equation $\Omega$ is from being \index{generalised equation!quadratic}quadratic (every item in the active part is covered twice).

\begin{lem}\label{lem:excess}
Let $\Omega$ be a generalised equation such that  every item $h_i$ from the active part $A\Sigma=[1,\rho_A]$ of $\Omega$ is covered at least twice. Suppose that $\D 5(\Omega,H)=(\Omega_i,H^{(i)})$ and the carrier base $\mu$ of $\Omega$ and its dual $\Delta(\mu)$ belong to the active part of $\Omega$.
Then the excess of the solution $H$ equals the excess of the solution $H^{(i)}$, i.e. $\psi_{A\Sigma}(H)=\psi_{A\Sigma}(H^{(i)})$.
\end{lem}
\begin{proof}
By the definition of the entire transformation $\D 5$, it follows that the word $H^{(i)}[1,\rho_{\Omega_i}+1]$ is a terminal subword of the word $H[1,\rho+1]$, i.e.
$$
H[1,\rho+1]\doteq U_iH^{(i)}[1,\rho_{\Omega_i}+1], \hbox{ where } \rho=\rho_\Omega.
$$
On the other hand, since $\mu, \Delta(\mu)\in A\Sigma$, the non-active parts of $\Omega$ and $\Omega_i$ coincide. Therefore, $H[\rho_A,\rho+1]$ is the terminal subword of the word $H^{(i)}[ 1,\rho _{\Omega_i}+1]$, i.e. the following graphical equality holds:
$$
H^{(i)}[1,\rho_{\Omega_i}+1]\doteq V_i H[\rho_A,\rho +1].
$$
So we have
\begin{equation}\label{3.22}
d_{A\Sigma}(H)-d_{A\Sigma}(H ^{(i)})= |H[1,\rho_A]|-|V_{i}|=|U_{i}|= |H({\mu})|-|H^{(i)}({\mu})|,
\end{equation}
where, in the above notation, if $k=\beta(\mu)-1$ (and thus $\mu$ has been completely eliminated in $\Omega^{(i)}$), then  $H^{(i)}({\mu})=1$; otherwise, $H^{(i)}({\mu})=H[k+1,\beta(\mu)]$.

From (\ref{3.13}) and (\ref{3.22}) it follows that $\psi _{A\Sigma}(H)=\psi _{A\Sigma}(H^{(i)})$.
\end{proof}

\subsubsection*{\glossary{name={$\D 6$}, description={derived transformation $\D 6$}, sort=D}$\D 6:$ Identifying closed constant sections}

Let $\lambda$ and $\mu$ be two $G_i$-constant bases in $\Omega$ with labels $a^{\epsilon_\lambda}$ and $a^{\epsilon_\mu}$, where $a \in \cA$ and $\epsilon_\lambda, \epsilon_\mu \in \{-1,1\}$. Suppose that the sections $\sigma(\lambda) = [i,i+1]$ and $\sigma(\mu) = [j,j+1]$ are closed.

The transformation $\D 6$ applied to $\Omega$ results in a single generalised equation $\Omega_1$ which is obtained from $\Omega$ in the following way, see Figure \ref{constants}. Introduce a new variable base $\eta$ with its dual $\Delta(\eta)$ such that
$$
\sigma(\eta) = [i,i+1], \ \sigma(\Delta(\eta)) = [j,j+1],\  \varepsilon(\eta) = \epsilon_\lambda, \ \varepsilon(\Delta(\eta)) = \epsilon_\mu.
$$
Then we transfer all bases from $\eta$ onto $\Delta(\eta)$ using $\ET 2$, remove the bases $\eta$ and $\Delta(\eta)$, remove the item $h_i$, re-enumerate the remaining items and adjust the initial and terminal terms of the boundaries as appropriate, see Remark \ref{rem:boundterm}.

The corresponding homomorphism $\theta_1:G_{R(\Omega^*)}\to G_{R(\Omega_1^*)}$ is induced by the composition of the homomorphisms defined by the respective elementary transformations. Obviously, $\theta_1$ is an isomorphism.

\bigskip

We now exhibit the behaviour of the length of a solution of a generalised equation under elementary and derived transformations.

\begin{lem} \label{lem:solleng}
Let $\Omega_i$ be a generalised equation and $H^{(i)}$ be a solution of $\Omega_i$ so that
$$
\ET:(\Omega,H)\to (\Omega_i,H^{(i)}) \hbox{ or } \D:(\Omega,H)\to (\Omega_i,H^{(i)}).
$$
Then one has $|H^{(i)}|\le |H|$. Furthermore, in the case that $\ET=\ET 4$ or $\D=\D 5$  this inequality is strict $|H^{(i)}|<|H|$.
\end{lem}
\begin{proof}
Proof is by straightforward examination of descriptions of elementary and derived transformations.
\end{proof}

\begin{lem} \label{lem:stform}
Using derived transformations, every generalised equation can be taken to the standard form.

Let $\Omega_1$ be obtained from $\Omega$ by an elementary or a derived transformation and let $\Omega$ be in the standard form. Then $\Omega_1$ is in the standard form.
\end{lem}
\begin{proof}
Proof is straightforward.
\end{proof}

\subsection{Construction of the tree  $T(\Omega)$} \label{se:5.2}

In this section we describe a branching process for rewriting a generalised equation $\Omega$. This process results in a locally finite and possibly infinite tree \glossary{name={$T(\Omega)$}, description={the infinite, locally finite tree of the process}, sort=T}
$T(\Omega)$. In the end of the section we describe infinite paths in $T(\Omega)$. We summarise the results of this section in the proposition below.

\begin{prop} \label{prop:TO}
For a generalised equation $\Omega=\Omega_{v_0}$, one can effectively construct a locally finite, possibly infinite, oriented rooted at $v_0$ tree $T$, $T=T(\Omega_{v_0})$, such that:
\begin{enumerate}
\item The vertices $v_i$ of $T$ are labelled by generalised equations $\Omega_{v_i}$.
\item The edges $v_i\to v_{i+1}$ of $T$ are labelled by epimorphisms
$$
\pi(v_i,v_{i+1}):G_{R(\Omega_{v_i}^\ast)}\to G_{R(\Omega_{v_{i+1}}^\ast)}.
$$
The edges $v_k\to v_{k+1}$, where $v_{k+1}$ is a leaf of $T$ and $\tp(v_{k+1})=1$, are labelled by proper epimorphisms. All the other epimorphisms $\pi(v_i,v_{i+1})$ are isomorphisms, in particular, edges that belong to infinite branches of the tree $T$ are labelled by isomorphisms.
\item Given a solution $H$ of $\Omega_{v_0}$, there exists a finite path $v_0\to v_1\to \dots \to v_l$ and a solution $H^{(l)}$ of $\Omega_{v_l}$ such that
    $$
\pi_H=\pi(v_0,v_1)\cdots \pi(v_{l-1},v_l)\pi_{H^{(l)}}.
    $$
    Conversely, for every finite path $v_0\to v_1\to \dots \to v_l$ in $T$ and every solution $H^{(l)}$ of $\Omega_{v_l}$ the homomorphism
    $$
    \pi(v_0,v_1)\cdots \pi(v_{l-1},v_l)\pi_{H^{(l)}}
    $$
    gives rise to a solution of $\Omega_{v_0}$.
\end{enumerate}
\end{prop}

We begin with a general description of the tree $T(\Omega)$ and then construct it using induction on its height.  For each vertex $v$ in $T(\Omega)$ there exists a generalised equation $\Omega_v$ associated to $v$. Recall, that we consider only formally consistent generalised equations. The initial generalised equation $\Omega$ is associated to the root $v_0$, $\Omega_{v_0} = \Omega$. For each edge $v\to v'$ there exists a unique surjective homomorphism $\pi(v,v'):G_{R(\Omega _v^* )}\rightarrow G_{R(\Omega _{v'}^* )}$ associated to $v\to v'$.

If
$$
v\rightarrow v_1\rightarrow\ldots\rightarrow v_s\rightarrow u
$$
is a path in $T(\Omega )$, then by $\pi (v,u)$ we denote the composition of corresponding homomorphisms
$$
\pi (v,u) = \pi (v,v_1)   \cdots \pi (v_s,u).
$$
We call this epimorphism the \index{canonical homomorphism of coordinate groups of generalised equations}\emph{canonical homomorphism from $G_{R(\Omega _v^*)}$ to $G_{R(\Omega _{u}^*)}$}

There are two kinds of edges in $T(\Omega)$: \emph{principal} and \emph{auxiliary}. Every constructed edge is principal, if not stated otherwise.
Let $v \to v^\prime$ be an edge of $T(\Omega)$, we assume that active (non-active) sections in $\Omega_{v^\prime}$ are naturally inherited from $\Omega_v$. If $v \to v^\prime$ is a principal edge, then there exists a finite sequence of elementary  and derived transformations from $\Omega_v$ to $\Omega_{v^\prime}$ and the homomorphism $\pi(v,v')$ is a composition of the homomorphisms corresponding to these transformations.

Since both elementary and derived transformations uniquely associate a pair $(\Omega_i,H^{(i)})$ to $(\Omega,H)$:
$$
\ET:(\Omega,H)\to (\Omega_i,H^{(i)}) \hbox{ and } \D:(\Omega,H)\to (\Omega_i,H^{(i)}),
$$
a solution $H$ of $\Omega$ defines a path in the tree $T(\Omega)$. We call such a path \index{path!defined by a solution in the tree $T(\Omega)$}\emph{the path defined by a solution $H$ in $T$}.

Let $\Omega $ be a generalised equation. We construct an oriented rooted tree $T(\Omega)$. We start from the root $v_0$ and proceed by induction on the height of the tree.

Suppose, by induction, that the tree $T(\Omega)$ is constructed  up to  height $n$, and let $v$ be a vertex of height $n$. We now describe how to extend the tree from $v$.  The construction of the  outgoing edges from $v$ depends on which of the case described below takes place at the vertex $v$. We always assume that:
\begin{center}
\emph{If the generalised equation $\Omega_v$ satisfies the assumptions of Case $i$, then $\Omega_v$ does not satisfy the assumptions of all the Cases $j$, with $j < i$.}
\end{center}

The general guideline of the process is to use the entire transformation to transfer bases to the right of the interval. Before applying the entire transformation one has to make sure that the generalised equation is ``clean''. The stratification of the process into 15 cases, though seemingly unnecessary, is convenient in the proofs since each of the cases has a different behaviour with respect to ``complexity'' of the generalised equation, see Lemma \ref{3.1}. Furthermore, it allows one to narrow the particular type of the generalised equations that can occur in each of the infinite branches of the process, see Lemma \ref{3.2}.

\subsubsection*{Preprocessing}

In $\Omega_{v_0}$ we transport closed sections using $\D 2$ in such a way that all active sections are at the left end of the interval (the active part of the generalised equation), then come all variable non-active sections, behind them come $G_1$-sections, followed by $G_2$-sections, and then the free sections.

\subsubsection*{Termination conditions: Cases 1 and 2}

\subsubsection*{Case 1: The homomorphism $\pi (v_0,v)$ is not an isomorphism, or, equivalently, the canonical homomorphism $\pi(v_1,v)$, where $v_1\to v$ is an edge of $T(\Omega)$, is not an isomorphism.}

The vertex $v$, in this case, is a leaf of the tree $T(\Omega)$. There are no outgoing edges from $v$.

\subsubsection*{Case 2: The generalised equation $\Omega _v$ does not contain active sections.}

Then the vertex $v$ is a leaf of the tree $T(\Omega)$. There are no outgoing edges from $v$.

\subsubsection*{Moving $G_i$-constant and $G_i$-factor bases to the non-active part: Cases 3 and 4}

\subsubsection*{Case 3: The generalised equation $\Omega _v$ contains a $G_i$-constant base $\lambda$ {\rm(}a $G_i$-factor base $\lambda${\rm)} in an active section such that the section $\sigma(\lambda)$ is not closed.}

In this case, we make the section $\sigma(\lambda)$ closed using the derived transformation $\D 1$.

\subsubsection*{Case 4: The generalised equation $\Omega _v$ contains a $G_i$-constant base $\lambda$ {\rm(}a $G_i$-factor base $\lambda${\rm)} such that the section $\sigma(\lambda)$ is closed.}

In this case, we transport the section $\sigma(\lambda)$ to $G_i\Sigma$ using the derived transformation $\D 2$. Then (only in the case when $\lambda$ is a constant base labelled by $a\in \cA$) we identify  all closed sections of the type $[i,i+1]$, which contain a constant base with the label $a^{\pm 1}$, with the transported section $\sigma(\lambda)$, using the derived transformation $\D 6$. In the resulting generalised equation $\Omega_{v^\prime}$ the section $\sigma(\lambda)$ becomes a $G_i$-section, and the corresponding edge $(v,v^\prime)$ is auxiliary, see Figure \ref{constants} (note that $\D 1$ produces a finite set of generalised equations, Figure \ref{constants} shows only one of them).

\begin{figure}[h!tb]
  \centering
   \includegraphics[keepaspectratio,width=6in]{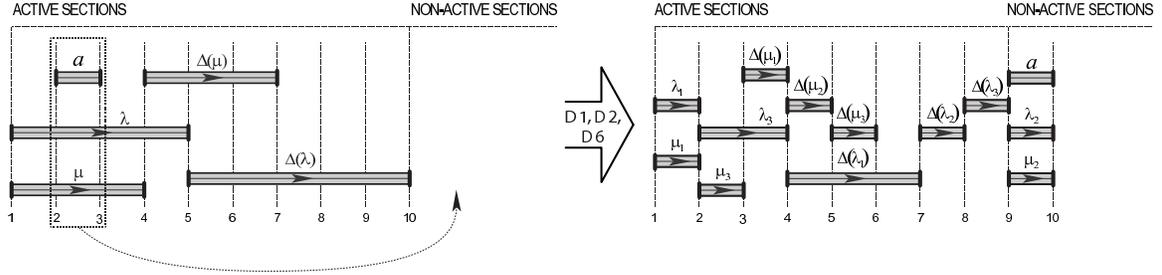}
\caption{Cases 3-4: Moving constant bases.} \label{constants}
\end{figure}

\subsubsection*{Moving free variables to the non-active part: Cases 5 and 6}

\subsubsection*{Case 5: The generalised equation $\Omega _v$ contains a free variable $h_q$ in an  active section.}

Using $\D 2$, we transport the section $[q,q+1]$ to the very end of the interval behind all the items of $\Omega_v$. In the resulting generalised equation $\Omega_{v^\prime}$ the transported section becomes a free section, and the corresponding edge $(v,v^\prime)$ is auxiliary.

\subsubsection*{Case 6: The generalised equation $\Omega _v$ contains a pair of matched bases in an active section.}

In this case, we perform $\ET 3$ and delete it, see Figure \ref{useless}.

\begin{figure}[!h]
  \centering
   \includegraphics[keepaspectratio,width=6in]{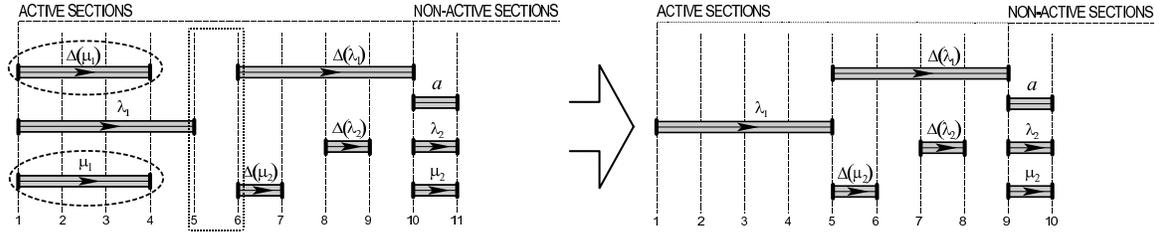}
 \caption{Cases 5-6: Removing a pair of matched bases and free variables.}
 \label{useless}
 \end{figure}

\begin{rem} \label{rem:stform}
If $v$ does not satisfy the conditions of any of the Cases 1-6, then the generalised equation $\Omega_v$ is in the standard form.
\end{rem}

\subsubsection*{Eliminating linear variables: Cases 7-10}

\subsubsection*{Case 7: There exists an item $h_i$ in an active section of $\Omega_v$ such that $\gamma _i=1$ and  such that both boundaries $i$ and $i+1$ are closed.}

Then, we remove the closed section $[i,i+1]$ together with the linear base using $\ET 4$.

\subsubsection*{Case 8: There exists an item $h_i$ in an active section of $\Omega_v$ such that $\gamma _i=1$ and such that  one of the boundaries $i$, $i+1$ is open, say $i+1$, and the other is closed.}

In this case, we first perform $\ET 5$ and $\mu$-tie $i+1$ by the only base $\mu$ it intersects; then using $\ET 1$ we cut $\mu$ in $i+1$; and then we delete the closed section $[i,i+1]$ using $\ET 4$,  see Figure \ref{linear'}.

\begin{figure}[!h]
  \centering
   \includegraphics[keepaspectratio,width=6in]{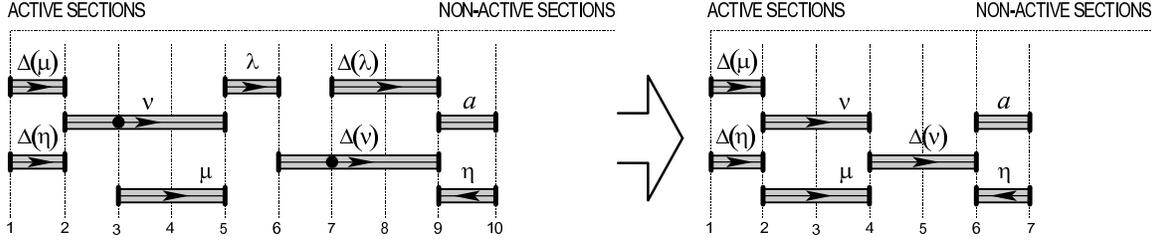}
 \caption{Cases 7-8: Linear variables.} \label{linear'}
\end{figure}

\subsubsection*{Case 9: There exists an item $h_i$ in an active section of $\Omega_v$ such that $\gamma _i=1$ and  such that both boundaries $i$ and $i+1$ are open. In addition, there is a closed section $\sigma$ such that $\sigma$ contains exactly two bases $\mu _1$ and $\mu _2$, $\sigma = \sigma(\mu_1) = \sigma(\mu_2)$ and $\mu_1, \mu_2$ is not a pair of matched bases, i.e. $\mu_1\ne \Delta(\mu_2)$; moreover, in the generalised equation $\widetilde {\Omega}_v= \D 3(\Omega)$ all the bases obtained from $\mu _1,\mu _2$  by $\ET 1$ when constructing $\widetilde {\Omega}_v$ from $\Omega_v$, do not belong to the kernel of $\widetilde{\Omega}_v$.}

In this case, using $\ET 5$ we  $\mu _1$-tie  all the boundaries that intersect $\mu_1$; using $\ET 2$, we transfer $\mu _2$ onto $\Delta (\mu_1)$; and  remove $\mu _1$ together with the closed section $\sigma$  using  $\ET 4$,  see Figure \ref{linear''}.

\subsubsection*{Case 10: There exists an item $h_i$ in an active section of $\Omega_v$ such that $\gamma _i=1$ and  such that both boundaries $i$ and $i+1$ are open.}

In this event we close the section $[i,i+1]$ using $\D 1$ and remove it using $\ET 4$, see Figure \ref{linear''}.

\begin{figure}[!h]
  \centering
   \includegraphics[keepaspectratio,width=6in]{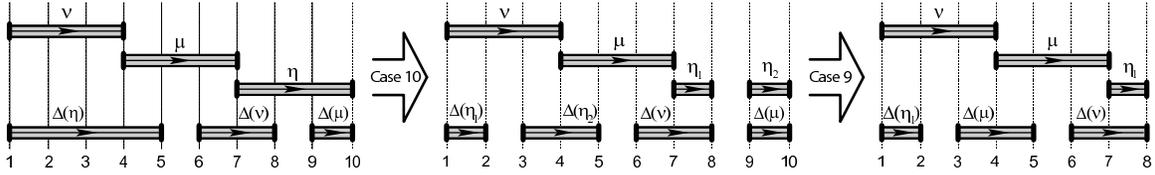}
  \caption{Cases 9-10: Linear case.} \label{linear''}
\end{figure}

\subsubsection*{Tying a free boundary: Case 11}

\subsubsection*{Case 11: Some boundary $i$ in the active part of $\Omega_v$ is free.}

Since the assumptions of Case 5 are not satisfied, the boundary $i$ intersects at least one base, say, $\mu$.

In this case, we $\mu$-tie $i$  using $\ET 5$.

\subsubsection*{Quadratic case: Case 12}

\subsubsection*{Case 12: For every item $h_i$ in the active part of $\Omega_v$ we have $\gamma _i = 2$.}

We apply the entire transformation $\D 5$, see Figure \ref{quadratic}.

\begin{figure}[!h]
  \centering
   \includegraphics[keepaspectratio,width=5in]{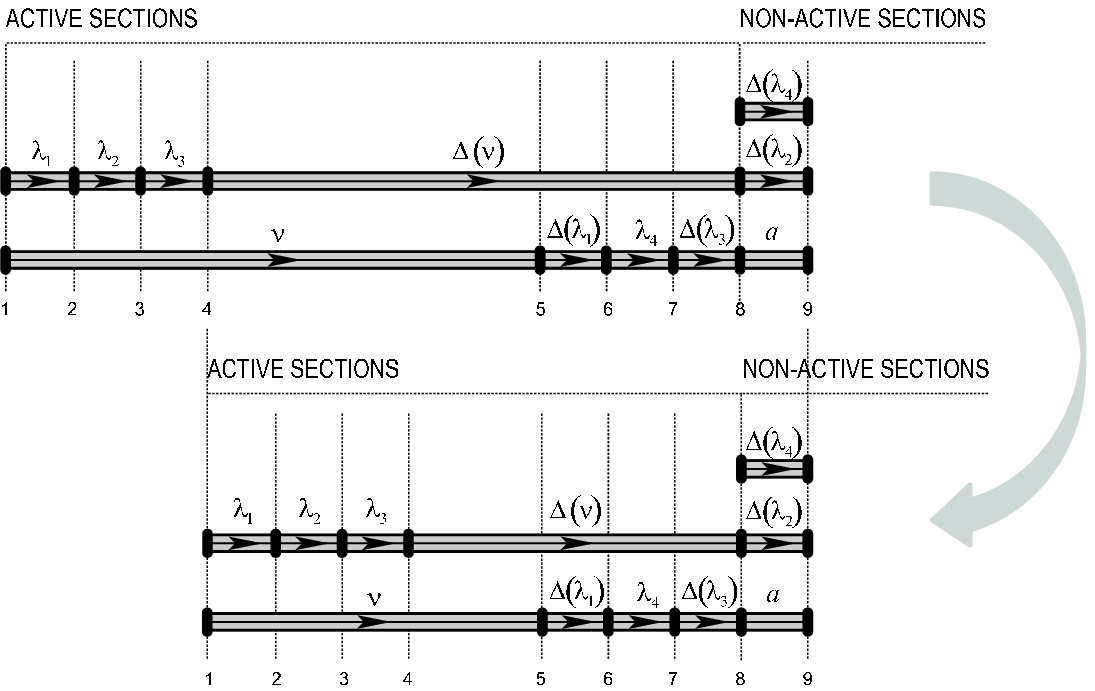}
  \caption{Case 12: Quadratic case, entire transformation.} \label{quadratic}
\end{figure}

\subsubsection*{Removing a closed section: Case 13}

\subsubsection*{Case 13: For every item $h_i$ in the active part of $\Omega_v$ we have $\gamma _i\geq 2$, and $\gamma_{i'}>2$ for at least one item $h_{i'}$ in the active part. Moreover, for some active base $\mu$ the section $\sigma(\mu)$ is closed.}

In this case, using $\D 2$, we transport the section $\sigma(\mu)$ to the beginning of the interval (this makes $\mu$ a carrier base). We then apply the entire transformation $\D 5$ (note that in this case the whole section $\sigma(\mu)$ is removed).

\subsubsection*{Tying a boundary: Case 14}

\subsubsection*{Case 14:  For every item $h_i$ in the active part of $\Omega_v$ we have $\gamma _i\geq 2$, and $\gamma _{i'}>   2$ for at least one item $h_{i'}$ in the active part. Moreover,  some boundary $j$ in the active part touches some base $\lambda$, intersects some base $\mu$,  and $j$ is not $\mu$-tied.}

In this case, using $\ET 5$, we $\mu$-tie $j$.

\subsubsection*{General case: Case 15}

\subsubsection*{Case 15: For every item $h_i$ in the active part of $\Omega_v$ we have $\gamma _i\geq 2$, and $\gamma _{i'}>   2$ for at least one item $h_{i'}$ in the active part. Moreover, every active section $\sigma(\mu)$ is not closed and every boundary $j$ is $\mu$-tied in every base it intersects.}

We first apply the entire transformation $\D 5$. Then, using $\ET 5$, we $\mu$-tie every boundary $j$ in the active part that intersects a base $\mu$ and touches at least one base. This results in finitely many new vertices connected to $v$ by principal edges.

If, in addition,  $\Omega_v$ satisfies the assumptions of Case 15.1 below, then besides the already constructed principal edges, we construct a few more auxiliary edges outgoing from the vertex $v$.

\subsubsection*{Case 15.1: The carrier base $\mu$ of the generalised equation $\Omega _v$ intersects with its dual $\Delta (\mu)$.}

We first construct an auxiliary generalised equation ${\widehat{\Omega}_{ v}}$ (which does \emph{not} appear in the tree $T(\Omega)$) as follows. Firstly, we add a new constant section $[\rho _{\Omega_v}+1,\rho_{\Omega_v}+2]$ to the right of all the sections in $\Omega_v$ (in particular, $h_{\rho _{\Omega_v}+1}$ is a new free variable). Secondly, we introduce a new pair of dual variable bases $\lambda ,\Delta (\lambda)$ so that
$$
\alpha(\lambda)=1, \ \beta(\lambda)=\beta(\Delta (\mu)),\  \alpha(\Delta(\lambda))=\rho_{\Omega_v}+1,\  \beta(\Delta(\lambda))=\rho_{\Omega_v}+2.
$$
Notice that $\Omega _v$ can be obtained from ${\widehat{\Omega}_{v}}$ if we apply $\ET 4$ to ${\widehat{\Omega}_{v}}$ and delete the base $\Delta(\lambda)$ together with the closed section $[\rho _v+1,\rho _v+2]$. Let
$$
{\hat \pi}_{v}: G_{R(\Omega _v^\ast)}\rightarrow G_{R({{{\widehat{\Omega}}_v}}^{\ast})}
$$
be the isomorphism induced by $\ET 4$. The assumptions of Case 15 still hold for $\widehat {\Omega} _{ v}$. Note that the carrier base of  $\widehat {\Omega} _{ v}$ is the base $\lambda$. Applying the transformations described in Case 15 to ${\widehat \Omega _{v}}$ (first $\D 5$ and then $\ET 5$), we obtain a set of new generalised equations $\{\Omega_{v^\prime_i}\mid i=1,\dots, \nn\}$ and the set of corresponding epimorphisms of the form:
$$
{\phi}_{v'_i}:G_{R({\widehat{\Omega}_{ v}}^\ast)} \rightarrow G_{R(\Omega_{v_i'}^\ast)}.
$$

Now for each generalised equation $\Omega_{v'_i}$ we add a vertex $v'_i$ and an auxiliary edge $v\to v'_i$ in the tree $T(\Omega)$. The edge $v\to v'_i$ is labelled by the homomorphism $\pi(v,v'_i)$ which is the composition of homomorphisms $\phi_{v'_i}$ and ${\hat \pi}_{v}$, $\pi(v,v'_i) = {\hat \pi}_{v}\phi_{v'_i}$. We associate the generalised equation $\Omega_{v'_i}$ to the vertex $v'_i$.

\bigskip

The tree $T(\Omega)$ is therefore constructed. Observe that, in general, $T(\Omega)$ is an infinite locally finite tree.

If the generalised equation $\Omega_v$ satisfies the assumptions of Case $i\ (1\leq i\leq 15)$, then we say that the vertex $v$ has \index{type!of a vertex}\index{type!of a generalised equation}\emph{type} $i$ and write \glossary{name={$\tp(v)$}, description={type of the vertex $v$ of the tree $T$, $T_0$, $T_{\dec}$, or $T_{\sol}$}, sort=T}$\tp(v)=i$.

\begin{lem}
If the universal Horn theory of $G$ is decidable, then for any vertex $v$ of the tree $T(\Omega)$ there exists an algorithm to decide whether or not $v$ is of type $k$, $k=1,\dots, 15$.
\end{lem}
\begin{proof}
The statement of the lemma is obvious for all cases, except for Case 1.

To decide whether or not the vertex $v$ is of type 1 it suffices to show that there exists an algorithm to check whether or not the canonical epimorphism $\pi(v,u)$ associated with an edge $v \rightarrow u$ in $T(\Omega)$ is a proper epimorphism. This can be done effectively by Lemma \ref{le:hom-check}.
\end{proof}

We now introduce certain characteristics of generalised equations and establish how these characteristics change depending on which of the above 15 cases holds for a given generalised equations.

\begin{defn}
Denote by \glossary{name={$n_A$, $n_A(\Omega)$}, description={the number of bases in the active sections of a generalised equation}, sort=N}$n_A=n_A(\Omega)$ the number of bases in the active sections of $\Omega$ and by \glossary{name={$\xi$}, description={the number of open boundaries in the active sections of a generalised equation}, sort=X}$\xi$ the number of open boundaries in the active sections.

For a closed section $\sigma \in \Sigma(\Omega)$ denote by $n(\sigma)$\glossary{name={$n(\sigma)$}, description={the number of bases in the section $\sigma$}, sort=N} the number of bases in $\sigma$. The \index{complexity of the generalised equation}{\em complexity} of a generalised equation $\Omega$ is defined as follows \glossary{name={$\comp$, $\comp (\Omega)$}, description={complexity of a generalised equation}, sort=C}
$$
\comp = \comp (\Omega) = \sum\limits_{\sigma \in A\Sigma(\Omega)} \max\{0, n(\sigma)-2\}.
$$
\end{defn}

\begin{rem}
We use the following convention. Let $\Omega$ be a generalised equation. By a function of a generalised equation $f(\Omega)$ we mean a function of the parameters $n_A(\Omega)$, $\xi(\Omega)$, $\rho_\Omega$ and $\comp(\Omega)$.
\end{rem}

\begin{lem}[cf. Lemma 3.1, \cite{Razborov3}] \label{3.1}
Let $u\rightarrow v$ be  a principal edge of the tree $T(\Omega)$. Then the following statements hold.
\begin{enumerate}
    \item If $\tp(u)\ne 3,10$, then $n_{A}(\Omega_v) \leq n_{A}(\Omega_u)$, moreover if $\tp(u)=6,7,9,13$, then this inequality is strict;
    \item If $\tp(u)=10$, then $n_{A}(\Omega_v) \leq n_{A}(\Omega_u) + 2$;
    \item If $\tp(u)\leq 13$ and $\tp (u)\ne 3,11$, then $\xi(\Omega_v) \leq \xi(\Omega_u)$;
    \item If $\tp (u) \ne 3$, then $\comp(\Omega_v)  \leq \comp(\Omega_u)$.
\end{enumerate}
\end{lem}
\begin{proof}
Straightforward verification.
\end{proof}

The following lemma gives a description of the infinite branches in the tree $T(\Omega)$. It is, basically, a consequence of Lemma \ref{3.1}; we refer to \cite{CK2} for a proof.

\begin{lem}[cf. Lemma 3.2, \cite{Razborov3}, Lemma 4.19, \cite{CK2}] \label{3.2} Let
\begin{equation}\label{eq:path}
v_0\rightarrow v_1\rightarrow\ldots \rightarrow v_r \rightarrow \ldots
\end{equation}
be an infinite path in the tree $T(\Omega )$. Then there exists a natural number $N$ such that all the edges
$v_n \rightarrow v_{n+1}$ of this path with $n \geq N$ are principal edges, and one of the following conditions holds:
$$
\begin{array}{crl}
  (A) & \hbox{linear case: } & 7\leq \tp(v_n)\leq 10 \hbox{ for all } n\ge N;\\
  (B) & \hbox{quadratic case: } & \tp(v_n)=12 \hbox{ for all } n\ge N; \\
  (C) & \hbox{general case: } & \tp(v_n)=15 \hbox{ for all } n\ge N.
\end{array}
$$
\end{lem}

\begin{rem}\label{rem:leng<}
Let
$$
(\Omega_{v_1},H^{(1)})\to (\Omega_{v_2},H^{(2)})\to \dots \to (\Omega_{v_l},H^{(l)})
$$
be the path defined by the solution $H^{(1)}$. If $ \tp(v_i)\in \{7,8,9,10,12,15\}$, then by {\rm Lemma \ref{lem:solleng}},  $|H^{(i+1)}|<|H^{(i)}|$.
\end{rem}

\section{Periodic structures}

Informally, the aim of this section is to prove the following strong version of the so-called Bulitko's Lemma, \cite{Bul}:
\begin{center}
\parbox{5.5in}{\textit{Applying automorphisms from a finitely generated subgroup $\AA(\Omega)$ of the group of automorphisms of the coordinate group $G_{R(\Omega^\ast)}$ to a periodic solution either one can bound the exponent of periodicity of the solution {\rm(}regular case, see {\rm Lemma \ref{lem:23-2})}, or one can get a solution of a proper equation  {\rm(} singular case, see {\rm Lemma \ref{lem:23-1})}.}}
\end{center}
Above, by a solution of a proper equation we mean a homomorphism from a proper quotient of the coordinate group of $\Omega$ to $G$.

This approach for free groups was introduced by A.~Razborov. In \cite{Razborov1}, he defines a combinatorial object, called a periodic structure on a generalised equation $\Omega$ and constructs a finite set of generators for the group $\AA(\Omega)$.

We fix till the end of this section a generalised equation $\Omega$ in the standard form. Suppose that some boundary $k$ (between $h_{k-1}$ and $h_k$) in the active part of $\Omega$ does not touch bases. Since the generalised equation $\Omega$ is in the standard form, the boundary $k$ intersects at least one base $\mu$. Using $\ET 5$ we $\mu$-tie the boundary $k$. Applying $\D 3$, if necessary, we may assume that the set of boundary connections in $\Omega$ is empty and that each boundary of $\Omega$ touches a base.

A reduced, cyclically reduced word $P$ in  $G$ is called a \index{period}\emph{period} if $P$ is not a proper power. Note that, in particular, one has that $P^2\doteq PP$. The latter condition is equivalent to saying that the length $|P^2|$ of $P^2$ equals $2|P|$ (note that, in particular $|P|\ge 2$). A reduced word $w \in G$ is called \index{periodic@($P$-)periodic word}$P$-\emph{periodic} if $|w| \ge |P|$ and, as written, it is a subword of $P^n$ for some $n$. Every $P$-periodic word $w$ can be presented in the form
\begin{equation}\label{2.50}
w = Q^rQ_1
\end{equation}
where $Q$ is a cyclic permutation of $P^{\pm 1}$,  $r \geq 1$, $Q \doteq Q_1 Q_2$ is a subdivision of $Q$, and $Q_2 \neq 1$. This presentation is unique if $r \ge 2$. The number $r$ is called the \index{exponent!of a word} \emph{exponent}  of $w$. A maximal exponent of the $P$-periodic subword in a word $w$ is called the \index{exponent!of ($P$-)periodicity}\emph{exponent of $P$-periodicity of $w$}. We denote it by \glossary{name={$\exp(w)$}, description={exponent of periodicity of the word $w$}, sort=E}$\exp(w)$.

\begin{defn}\label{11'}
A solution $H=(H_1,\dots, H_{\rho})$ of $\Omega$ is called \index{solution!of a generalised equation!periodic with respect to a period}\emph{periodic with respect to a period $P$}, if for every closed variable section $\sigma$ of $\Omega$ one of the following conditions holds:
\begin{enumerate}
 \item  \label{it:per1} $H(\sigma)$ is $P$-periodic with exponent $r \ge 2$;
 \item  \label{it:per2} $|H(\sigma)| \le |P|$;
 \item  \label{it:per3} $H(\sigma)$ is $A$-periodic and $|A| \le |P|$;
\end{enumerate}
Moreover, condition (\ref{it:per1}) holds for at least one closed variable section $\sigma$ of $\Omega$.
\end{defn}

Let $H$ be a $P$-periodic solution of $\Omega$. Then a section $\sigma$ satisfying condition (\ref{it:per1}) of the above definition is called \index{section!($P$-)periodic}\emph{$P$-periodic} (with respect to $H$).

The following lemma gives an intuition about the kind of generalised equations that have periodic solutions.

\begin{lem} \label{lem:case2}
Let $\Omega$ be a generalised equation such that every closed section $\sigma_i$ of $\Omega$ is either constant or there exists a pair of dual bases $\mu_i$, $\Delta(\mu_i)$ such that $\mu_i$ and $\Delta(\mu_i)$ intersect but do not form a pair of matched bases, and $\sigma_i=[\alpha(\mu_i),\beta(\Delta(\mu_i))]$. Let $H$ be a solution of $\Omega$. Then there exists a period $P$ such that $H$ is $P$-periodic.
\end{lem}
\begin{proof}
Consider a section $\sigma=[\alpha(\mu),\beta(\Delta(\mu))]$. The boundary $i_1=\alpha(\Delta(\mu))$ intersects the base $\mu$, since the bases $\mu$ and $\Delta(\mu)$ overlap. We $\mu$-tie the boundary $i_1$, i.e. we introduce a boundary connection $(i_1,\mu,i_2)$ in such a way that $H$ is a solution of the obtained generalised equation. It follows that $i_1<i_2$.

Repeating this argument, we obtain a finite set of boundaries $i_1<\dots<i_{k+1}$, $k\ge 1$ such that
\begin{itemize}
\item $i_1,\dots,i_{k}$  intersect $\mu$,
\item $i_{k+1}$ does not intersect $\mu$,
\item there is a boundary connection $(i_j, \mu, i_{j+1})$ for all $j=1,\dots, k$ and
\item $H$ induces a solution of the generalised equation obtained.
\end{itemize}
This set of boundaries is finite, since the length of the solution $H$ is finite.

Let $H[\alpha(\mu),i_1]=w$, $w=A^l$, where $l\ge 1$ and $A$ is a period. Then the section $\sigma$ is $A$-periodic. Indeed,
$$
\sigma=[\alpha(\mu),i_1]\cup[i_1,i_2]\cup\dots\cup[i_k,i_{k+1}]\cup[i_{k+1},\beta(\Delta(\mu))].
$$
Using the boundary equations, we get that
$$
h[\alpha(\mu),i_1]=h[i_1,i_2]=\dots=h[i_k,i_{k+1}] \hbox{ and } h[i_{k+1},\beta(\Delta(\mu))]=h[i_k,\beta(\mu)],
$$
thus $H(\sigma)\doteq A^{l\cdot(k+1)}\cdot A_1$, where $A\doteq A_1A_2$.

Set $P=A_j$, where $|A_j|=\max\limits_i\{|A_i|\mid\sigma_i \hbox{ is $A_i$-periodic}\}$. By definition, $H$ is $P$-periodic.
\end{proof}

\subsection{Periodic  structures}

Below we introduce the notion of a periodic structure. The idea of considering periodic structures on $\Omega$ is to subdivide the set of periodic solutions into subsets so that any two solutions from the same subset have the same set of ``long items'', i.e. $P$-periodic solutions that factor through the generalised equation $\Omega$ and a periodic structure $\langle \P, R\rangle$ on $\Omega$ satisfy the following property:
$$
h_i\in \P \hbox{ if and only if } |H_i|\ge 2|P|.
$$
One can regard Lemma \ref{le:PP} below as a motivation for the definition of a periodic structure.

\begin{defn} \label{above}
Let  $\Omega$ be a generalised equation in the standard form without boundary connections. A \index{periodic structure}\emph{periodic structure} on $\Omega$ is a pair \glossary{name={$\langle {\P}, R \rangle$}, description={periodic structure}, sort=P}$\langle {\P}, R \rangle$, where
\begin{enumerate}
\item \label{it:ps1} \glossary{name={$\P$},description={non-empty set of variables, variable bases, and closed sections that belong to the periodic structure $\langle {\P}, R \rangle$}, sort=P}${\P}$ is a non-empty set consisting of some variables $h_i$, some variable bases $\mu$, and some closed sections $\sigma$ from $V\Sigma$ such that the following conditions hold:
\begin{itemize}
    \item[(a)] if $h_i \in {\P}$, $h_i \in \mu$, $\mu\in V\Sigma$ and $\Delta(\mu) \in V\Sigma$, then $\mu \in {\P}$;

    \item[(b)] if $\mu \in {\P}$, then $\Delta(\mu) \in {\P}$;

    \item[(c)] if $\mu \in {\P}$ and $\mu \in \sigma$, then $\sigma \in {\P}$;

    \item[(d)] there exists a function ${\mathcal X}$ mapping the set of closed sections from ${\P}$ into $\{-1,1\}$ such that for every $\mu, \sigma_1, \sigma_2 \in {\P}$, the condition that $\mu \in \sigma_1$ and $\Delta(\mu) \in \sigma_2$ implies $\varepsilon(\mu) \cdot \varepsilon(\Delta(\mu)) = {\mathcal X}(\sigma_1) \cdot {\mathcal X}(\sigma_2)$;
\end{itemize}
\item \label{it:ps2}  $R$ is an equivalence relation on a certain set ${\mathcal B}$ (defined in (e)) such that condition (f) is satisfied.
\begin{itemize}
 \item[(e)] Notice, that for every boundary $l$ belonging to a closed section in $\P$ either there exists a unique closed section $\sigma(l)$ in ${\P}$ containing $l$, or there exist precisely two closed sections $\sigma_{\lef}(l) = [i,l], \sigma_{\rig} =  [l,j]$ in ${\P}$ containing $l$. The set of boundaries of the first type we denote  by ${\B}_1$, and of the second type by  ${\B}_2$. Put
$$
{\B} = {\B}_1  \cup \{l_{\lef}, l_{\rig}  \mid l \in {\B}_2\}
$$
here $l_{\lef}, l_{\rig}$ are two ``formal copies'' of $l$. We  will use the following agreement: for any base $\mu$ if $\alpha(\mu) \in {\B}_2$ then by $\alpha(\mu)$ we mean $\alpha(\mu)_{\rig}$ and, similarly, if $\beta(\mu) \in {\B}_2$ then by $\beta(\mu)$ we mean $\beta(\mu)_{\lef}$.

\item[(f)]  If $\mu \in {\P}$ then
$$
\begin{array}{lll}
\alpha(\mu) \sim_R \alpha(\Delta(\mu)), &  \beta(\mu) \sim_R \beta(\Delta(\mu)), & \hbox{ if } \varepsilon(\mu) =
\varepsilon(\Delta(\mu)); \\
\alpha(\mu) \sim_R \beta(\Delta(\mu)), & \beta(\mu)\sim_R  \alpha(\Delta(\mu)), & \hbox{ if }\varepsilon(\mu) = -
 \varepsilon(\Delta(\mu)).
\end{array}
$$
\end{itemize}
\end{enumerate}
\end{defn}

\begin{rem} \label{rem:defperstr}
For a given generalised equation $\Omega$, there exists only finitely many periodic structures on $\Omega$, and all of them can be constructed effectively. Indeed, every periodic structure $\langle\P,R\rangle$ is uniquely defined by the subset of items of $\Omega$ that belong to $\P$ and the relation $\sim_R$. Therefore, to describe all periodic structures on $\Omega$ it suffices to consider all subsets of the set of items of $\Omega$ and different relations $\sim_R$ on them.
\end{rem}

Now we will show how to a $P$-periodic solution $H$ of  $\Omega$ one can associate a periodic structure \glossary{name={${\P}(H, P)$}, description={periodic structure associated to a $P$-periodic solution $H$}, sort=P}${\P}(H, P) = \langle {\P}, R \rangle$ on $\Omega$. We define ${\P}$ as follows. A closed section $\sigma$ is in ${\P}$ if and only if $\sigma$ is $P$-periodic. A variable $h_i$ is in ${\P}$ if and only if $h_i \in \sigma$ for some $\sigma \in {\P}$ and $|H_i| \geq 2 |P|$. A variable base $\mu$ is in ${\P}$ if and only if either $\mu$ or $\Delta(\mu)$
contains an item $h_i$ from ${\P}$. Notice that for any factor base $\nu$, $|H(\nu)|=1$ thus $\nu$ never belongs to ${\P}(H, P)$.

Put ${\mathcal X}([i,j]) = \pm 1$ depending on whether in (\ref{2.50}) the word $Q$ is conjugate to $P$ or to $P^{-1}$.

Now let $[i,j]\in {\P}$ and $ i \leq l \leq j$. Then one can write $P \doteq P_1P_2$ in such a way that if ${\mathcal X} ([i,j]) =1$, then the word $H[i,l]$ is the terminal subword of the word $(P^\infty)P_1$, where $P^\infty$ is the infinite word obtained by concatenating the powers of $P$, and $H[l,j]$ is the initial subword of the word $P_2(P^\infty)$; and if ${\mathcal X}([i,j])= -1$, then the word $H[i,l]$ is the terminal subword of the word $(P^{-1})^\infty P_2^{-1}$ and $H[l,j]$ is the initial subword of $P_1^{-1}(P^{-1})^\infty$. By Lemma 1.2.9 \cite{1}, the subdivision $P\doteq P_1P_2$ with these properties is unique;  denote this decomposition by \glossary{name={$\delta(l)$}, description={decomposition of the period defined by the boundary $l$}, sort=D}$\delta(l)$. We define the relation $R$ as follows:
$$
l_1 \sim_R l_2 \hbox{ if and only if } \delta(l_1) =  \delta(l_2).
$$

\begin{lem}\label{le:PP}
Let $H$ be a periodic solution of $\Omega$. Then ${\P}(H, P)$ is a periodic structure on $\Omega$.
\end{lem}
\begin{proof}
Let ${\P}(H, P) = \langle {\P}, R \rangle$.  Obviously,  ${\P}$ satisfies conditions (a) and (b) from Definition \ref{above}.

We now prove that $\P$ satisfies condition (c) from Definition \ref{above}. Let $\mu \in {\P}$ and $\mu \in [i,j]$. There exists an item $h_k \in {\P}$ such that $h_k \in \mu$ or $h_k \in \Delta({\mu})$. If $h_k \in \mu$, then, by construction, $[i,j] \in {\P}$. If $h_k \in \Delta(\mu)$ and $\Delta(\mu) \in [i',j']$, then $[i', j'] \in {\P}$, and hence, the word $H(\Delta(\mu))$ can be written in the form $Q^{r'} Q_1$, where $Q\doteq Q_1 Q_2$ is a cyclic permutation of the word $P^{\pm 1}$ and $r' \geq 2$. Since $|H[i,j]| \ge |H(\mu)| = |H(\Delta(\mu))|\ge 2 |P|$ and from Definition \ref{11'}, it follows that $[i,j]$ is an $A$-periodic section, where $|A|\le |P|$. Then $H(\mu) = B^s B_1$, where $B$ is a cyclic permutation of the word $A^{\pm 1}$, $|B| \leq |P|$, $B \doteq B_1 B_2$, and $s \ge 0$. From the equality ${H(\mu)}^{\varepsilon(\mu)} = {H(\Delta(\mu))}^{\varepsilon(\Delta(\mu))}$ and Lemma 1.2.9 \cite{1} it follows that $B$ is a cyclic permutation of the word $Q^{\pm 1}$. Consequently, $A$ is a cyclic permutation of the word $P^{\pm 1}$. Therefore, $[i,j]$ is a $P$-periodic section of $\Omega$ with respect to $H$, in other words, the length of $H[i,j]$ is greater than or equal to $2|P|$ and so $[i,j]\in \P$.

If $\mu \in [i_1, j_1]$, $\Delta(\mu) \in [i_2, j_2]$ and $ \mu \in {\P}$, then the equality $\varepsilon(\mu) \cdot \varepsilon(\Delta(\mu))$ = ${\mathcal X}([i_1, j_1]) \cdot {\mathcal X}([i_2, j_2])$ follows from the fact that given $A^r A_1 = B^s B_1$ and $r,s \geq 2$, the word $A$ cannot be a cyclic permutation of the word $B^{-1}$, hence condition (d) of Definition \ref{11'} holds.

Since  ${H(\mu)}^{\varepsilon(\mu)}\doteq {H(\Delta(\mu))}^{\varepsilon(\Delta(\mu))}$, from Lemma 1.2.9 in \cite{1} it follows that condition (f) also holds for the relation $R$.
\end{proof}

\begin{rem} \label{rem:subword}
Now let us fix a nonempty periodic structure $\langle {\P}, R \rangle$ on a generalised equation $\Omega$. Item (d) of Definition \ref{above} allows us to assume (after replacing the variables $h_i, \ldots, h_{j-1}$ by $h_{j-1}^{-1}, \ldots, h_i^{-1}$ on those closed sections $[i,j] \in {\P}$ for which ${\mathcal X}([i,j])=-1$) that $\varepsilon(\mu)=1$ for all $\mu \in {\P}$. Therefore, we may assume that for every item $h_i$, the word $H_i$ is a subword of the word $P^\infty$.
\end{rem}

The rest of this section is devoted to defining the group of automorphisms $\AA(\Omega)$. The idea is as follows. We change the set of generators $h$ of the coordinate group $G_{R(\Omega^*)}$ by $\bar x$ and we use the new set of generators to give an explicit description of the generating set of the group of automorphisms $\AA(\Omega)$.

To construct the set of generators $\bar x$ of $G_{R(\Omega^*)}$, the following definitions are in order.

\begin{defn}
We construct the \index{graph!of a periodic structure}\emph{graph \glossary{name=$\Gamma$, description={graph of a periodic structure}, sort=G}$\Gamma=\Gamma(\langle \P,R\rangle)$ of a periodic structure $\langle \P,R\rangle$}. The set of vertices $V(\Gamma)$ of the graph $\Gamma$ is the set of $R$-equivalence classes. For a boundary $k$, denote by $(k)$ the equivalence class of the relation $R$ to which it belongs. For each variable $h_k$ that belongs to a certain closed section from ${\P}$, we introduce an oriented edge $e\in E(\Gamma)$ from $(k)$ to $(k+1)$, $e:(k)\to(k+1)$ and an inverse edge $e^{-1}:(k+1)\to (k)$. We label the edge $e$ by $h(e) = h_k$ (correspondingly, $h(e^{-1}) = h_k^{-1}$). For every path $\p=e_1^{\epsilon_1} \ldots e_j^{\epsilon_j}$ in the graph $\Gamma$, we denote its label by $h(\p)$,  $h(\p)=h(e_1^{\epsilon_1}) \ldots h(e_j^{\epsilon_j})$, $\epsilon_1,\dots,\epsilon_j\in \{1,-1\}$.
\end{defn}

The periodic structure $\langle {\P}, R \rangle$ is called \index{periodic structure!connected}\emph{connected}, if its graph $\Gamma$ is connected.

Let $\langle {\P}, R \rangle$ be an arbitrary periodic structure of a generalised equation $\Omega$. Let $\Gamma_1, \ldots, \Gamma_r$ be the connected components of the graph $\Gamma$. The labels of edges of the component $\Gamma_i$ form (in the generalised equation $\Omega$) a union of labels of closed sections from ${\P}$. Moreover, if a base $\mu \in {\P}$ belongs to a section from $\P$, then its dual $\Delta(\mu)$, by condition (f) of Definition \ref{above}, also belongs to a section from $\P$. Therefore, by taking for ${\P}_i$: the set of labels of edges from $\Gamma_i$ that belong to ${\P}$, closed sections to which these labels belong, and bases $\mu \in {\P}$ that belong to these sections, and restricting the relation $R$ accordingly, we obtain a connected periodic structure $\langle {\P}_i, R_i \rangle$ whose graph is  $\Gamma_i$.

We write $\langle {\P}', R' \rangle\subseteq \langle {\P}, R \rangle$ meaning that ${\P}'
\subseteq {\P}$ and the relation $R'$ is a restriction of the relation $R$. In particular, in the above notation $\langle {\P}_i, R_i \rangle \subseteq \langle {\P}, R \rangle$.

We further assume that the periodic structure $\langle {\P}, R \rangle$ is connected.

Let $\Gamma$ be the graph of a periodic structure $\langle {\P}, R \rangle = {\P}(H, P)$, let $\p=e_1^{\epsilon_1} \ldots e_j^{\epsilon_j}$ be a path in $\Gamma$ and let $h(\p)$ be its label, $h(\p)=h(e_1^{\epsilon_1}) \ldots h(e_j^{\epsilon_j})$, $\epsilon_1,\dots,\epsilon_j\in \{1,-1\}$. To simplify the notation we write $H(\p)$ instead of $H(h(\p))$.

\begin{lem} \label{2.9}
Let $H$ be a $P$-periodic solution of a generalised equation $\Omega$, let $\langle {\P}, R \rangle = {\P}(H, P)$ be a periodic structure on $\Omega$ and let $\cc$ be a cycle in the graph $\Gamma=\Gamma(\langle \P,R\rangle)$ at the vertex $(l)$, $\delta(l)=P_1P_2$. Then there exists $n \in \Z$ such that $H(\cc) = (P_2P_1)^n$.
\end{lem}
\begin{proof}
If $e$ is an edge $v\to v'$ in the graph $\Gamma$, and $P = P_1P_2$,  $P = P_1' P_2'$ are two decompositions corresponding to the boundaries from $(v)$ and $(v')$ respectively. Then, obviously, $H(e) = P_2 P^{n_k}P_1'$, $n_k \in \Z$. The statement follows if we multiply the values $H(e)$ for all the edges $e$ in the cycle $\cc$.
\end{proof}

\begin{defn} \label{2.51}
A generalised equation $\Omega$ is called \index{generalised equation!periodised}\emph{periodised} with respect to a given periodic structure $\langle {\P}, R \rangle$, if for every two cycles $\cc_1$ and $\cc_2$ based at the same vertex in the graph $\Gamma(\langle \P,R\rangle)$, there is a relation $[h(\cc_1), h(\cc_2)]=1$ in $G_{R(\Omega^\ast)}$.
\end{defn}

Let \glossary{name={$\Sh$, $\Sh(\Gamma)$}, description={the set of ``short'' edges of the graph $\Gamma$ of a periodic structure}, sort=S}\glossary{name={$h(\Sh)$}, description={labels of the ``short'' edges of $\Gamma$, $h(\Sh)=\{h(e) \mid e\in \Sh \}$}, sort=H}$\Sh=\Sh(\Gamma)=\{e \in E(\Gamma) \mid h(e)\notin \P\}$ and $h(\Sh)=\{h(e) \mid e\in \Sh \}$.

Let \glossary{name={$\Gamma_0$}, description={the subgraph of $\Gamma$ all of whose edges are ``short'', $\Gamma_0=(V(\Gamma),\Sh(\Gamma))$}, sort=G}$\Gamma_0=(V(\Gamma),\Sh(\Gamma))$ be the subgraph of the graph $\Gamma=\Gamma(\langle \P,R\rangle)$ having the same set of vertices as $\Gamma$ and the set of edges  $E(\Gamma_0)=\Sh$. Choose a maximal subforest $T_0=T_0(\langle \P,R\rangle)$ in the graph $\Gamma_0$ and extend it to a maximal subforest $T=T(\langle \P,R\rangle)$ of the graph $\Gamma$. Since the periodic structure $\langle {\P}, R \rangle$ is connected by assumption, it follows that $T$ is a tree. Fix an arbitrary vertex $v_\Gamma$ of the graph $\Gamma$ and denote by $\p(v_\Gamma, v)$ the (unique) path in $T$ from $v_\Gamma$ to $v$. For every edge $e: v \to v'$ not lying in $T$, we introduce a cycle $\cc_e = \p(v_\Gamma, v) e (\p(v_\Gamma, v'))^{-1}$. Then the fundamental group $\pi_1(\Gamma, v_\Gamma)$ is generated by the cycles $\cc_e$ (see, for example, the proof of Proposition III.2.1, \cite{LS}).

If the universal horn theory of $G$ is decidable, then the property of a generalised equation ``to be periodised with respect to a given periodic structure'' is algorithmically decidable. Indeed, it suffices to check if the following universal formula (quasi-identity) holds in $G$ (for every pair of cycles $\cc_{e_1}$, $\cc_{e_2}$):
$$
\forall H_1, \dots, H_\rho \left(\left(\bigwedge (\Omega^*(H)=1)\right)\to \left([H(\cc_{e_1}),H(\cc_{e_2})]=1\right)\right).
$$

Furthermore, the set of elements
\begin{equation} \label{2.52}
\{h(e) \mid e \in T \} \cup \{h(\cc_e) \mid e \not \in T \}
\end{equation}
forms a basis of the free group generated by
$$
\{h_k \mid h_k\in \sigma, \sigma\in {\P} \}.
$$
If $\mu \in {\P}$, then $(\beta(\mu)) = (\beta(\Delta(\mu)))$, $(\alpha(\mu)) = (\alpha(\Delta(\mu)))$ by property (f) from Definition \ref{above} and, consequently, the word $h(\mu) {h(\Delta(\mu))}^{-1}$ is the label of a cycle $\cc'_\mu$ from $\pi_1 (\Gamma, (\alpha(\mu)))$. Let $\cc_\mu = \p(v_\Gamma, (\alpha(\mu)))\cc'_\mu \p(v_\Gamma, (\alpha(\mu)))^{-1}$. Then
\begin{equation} \label{2.53}
h(\cc_\mu) = uh(\mu) {h(\Delta(\mu))}^{-1} u^{-1},
\end{equation}
where $u$ is the label of the path $\p(v_\Gamma, (\alpha(\mu)))$. Since $\cc_\mu \in \pi_1(\Gamma, v_\Gamma)$, it follows that $\cc_\mu = b_\mu (\{\cc_e \mid e \not \in T \})$, where $b_\mu$ is a certain word in the indicated generators that can be  constructed effectively (see Proposition III.2.1, \cite{LS}).

Let $\tilde{b}_\mu$ denote the image of the word $b_\mu$ in the abelianisation of $\pi_1(\Gamma ,v_\Gamma)$. Denote by $\widetilde{Z}$ the free abelian group consisting of formal linear combinations $\sum\limits_{e \not \in T} n_e \tilde{\cc}_e$, $n_e \in {\Z}$, and by $\widetilde{B}$ its subgroup generated by the elements $\tilde{b}_\mu$, $\mu \in {\P}$ and the elements $\tilde{\cc}_e$, $e \not \in T$, $e \in \Sh$.

By the classification theorem of finitely generated abelian groups, one can effectively construct a basis $\{\widetilde{C}^{(1)}, \widetilde{C}^{(2)}\}$ of $\widetilde{Z}$ such that
\begin{equation} \label{2.54}
\widetilde{Z} = \widetilde{Z}_1 \oplus \widetilde{Z}_2, \ \widetilde{B} \subseteq \widetilde{Z}_1, \ [\widetilde{Z}_1 : \widetilde{B}] < \infty,
\end{equation}
where $\widetilde{C}^{(1)}$ is a basis of $\widetilde{Z}_1$ and $\widetilde{C}^{(2)}$ is a basis of $\widetilde{Z}_2$.

By Proposition I.4.4 in \cite{LS}, one can effectively construct a basis \glossary{name={$C^{(1)}$}, description={``short'' cycles in $\Gamma$}, sort=C}$C^{(1)}=C^{(1)}(\langle \P, R\rangle)$, \glossary{name={$C^{(2)}$}, description={``free'' cycles in $\Gamma$}, sort=C}$C^{(2)}=C^{(2)}(\langle \P, R\rangle)$ of the free (non-abelian) group $\pi_1(\Gamma, v_\Gamma)$ such that $\widetilde{C}^{(1)}$, $\widetilde{C}^{(2)}$ are the natural images of the elements $C^{(1)}$, $C^{(2)}$ in $\widetilde{Z}$.

\begin{rem} \label{rem:hab}
Notice that any equation in $\widetilde{Z}$ of the form
$$
\tilde{ \cc}= \sum\limits_{\tilde \cc_i\in \widetilde{Z}} n_i \tilde {\cc}_i
$$
lifts to an equation in $\pi_1(\Gamma,v_\Gamma)$ of the form
$$
\cc =\prod\limits_{\cc_i\in \pi_1(\Gamma,v_\Gamma)} {\cc}_i^{n_i} V,
$$
where $V$ is an element of the derived subgroup of $\pi_1(\Gamma,v_\Gamma)$. If the generalised equation is periodised, for any two cycles $\cc'_1, \cc'_2 \in \pi_1(\Gamma,v_\Gamma)$, we have that $[h(\cc'_1),h(\cc'_2)]=1$ in $G_{R(\Omega^*)}$. Hence, $h(\cc)=\prod\limits_{\cc_i\in \pi_1(\Gamma,v_\Gamma)} {h(\cc_i)}^{n_i}$ in $G_{R(\Omega^*)}$.
\end{rem}

\begin{lem} \label{cl:propc1}
Let $\Omega$ be a periodised generalised equation. Then the basis $\widetilde{C}^{(1)}$ can be chosen in such a way that for every $\cc\in C^{(1)}$ either $\cc=\cc_e$, where $e\notin T$, $e\in \Sh$, or for any solution $H$ we have $H(\cc)=1$.
\end{lem}
\begin{proof}
The set $\{\tilde{\cc}_e \mid e\notin T, e\in \Sh\}$ is a subset of the set of generators of $\widetilde{Z}$ contained in $\widetilde{Z}_1$. Thus, this set can be extended to a basis of $\widetilde{Z}_1$. Since, by (\ref{2.54}),  $[\widetilde{Z}_1 : \widetilde{B}] < \infty$, for every $\tilde \cc\in \widetilde{Z}_1$ there exists $n_\cc\in \N$ such that
$$
n_\cc \tilde \cc=\sum\limits_{e \not \in T,  e\in \Sh} n_e\tilde{\cc}_e +\sum\limits_{\mu \in {\P}} n_\mu \tilde {b}_\mu.
$$
It follows that the set $\{\tilde{\cc}_1,\dots, \tilde{\cc}_k\}$ which completes the set $\{\tilde{\cc}_e \mid e\notin T,  e\in \Sh\}$ to a basis of $\widetilde{Z}_1$, can be chosen so that the following equality holds:
$$
n_{\cc_i} \tilde{ \cc}_i= \sum\limits_{\mu \in {\P}} n_\mu \tilde {b}_\mu.
$$

Hence, by Remark \ref{rem:hab}, for any solution $H$ we have ${H(\cc_i)}^{n_{\cc_i}}=\prod\limits_{\mu \in {\P}} {H(b_\mu)}^{n_\mu}=1$.
Since $P$ is cyclically reduced and $P\notin G_1\cup G_2$ and since $n_{\cc_i}\ne 0$, we have that $H(\cc_i)=1$.
\end{proof}

Let $\{e_1, \ldots, e_m\}$ be the set of edges of $T \setminus T_0$. Since $T_0$ is the spanning forest of the graph $\Gamma_0$, it follows that $h(e_1), \ldots, h(e_m) \notin h({\Sh})$; in particular, $h(e_1), \ldots, h(e_m) \in \P$.

Let $F(\Omega)$ be the free group
generated by the variables of $\Omega$. Consider in the group $F(\Omega )$ a new set of generators $\bar x$ defined below (we prove that the set $\bar x$ is in fact a set of generators of $F(\Omega )$ in part (\ref{spl3}) of Lemma \ref{2.10''}).

Let $v_i$ be the origin of the edge $e_i$. We introduce new variables
\begin{equation} \label{eq:uie}
\bar u^{(i)}=\{u_{ie}\mid e\not\in T,\ e\in \Sh\}, \ \bar z^{(i)}=\{z_{ie}\mid e\not\in T, e\in \Sh\}, \hbox{ for } 1\leq i\leq m,
\end{equation}
as follows
\begin{equation}\label{2.59}
u_{ie}={h(\p(v_\Gamma,v_i))}^{-1}h(\cc_e)h(\p(v_\Gamma,v_i)),\ z_{ie}=h(e_i)^{-1}u_{ie}h(e_i).
\end{equation}
We denote by \glossary{name={$\bar u$}, description={a subset of the set $\bar x$ of generators of the coordinate group of a periodised generalised equation}, sort=U}$\bar u$ the union $\bigcup\limits_{i=1}^m \bar u^{(i)}$ and by \glossary{name={$\bar z$}, description={a subset of the set $\bar x$ of generators for the coordinate group of a periodised generalised equation}, sort=Z}$\bar z$ the union $\bigcup\limits_{i=1}^m \bar z^{(i)}$. Denote by $\bar t$ the family of variables that do not belong to closed sections from $\P$. Let
\glossary{name={$\bar x$}, description={a set of generators of the coordinate group of a periodised generalised equation}, sort=X}
$$
\bar x = \bar t \cup \{h(e) \mid e\in T_0\} \cup \bar u \cup \bar z\cup \{h(e_1),\dots,h(e_m)\}\cup h(C^{(1)})\cup h(C^{(2)})
$$

\begin{rem}
Note that without loss of generality we may assume that $v_\Gamma$ corresponds to the beginning of the period $P$. Indeed, it follows from the definition of the periodic structure that for any cyclic permutation $P'=P_2P_1$ of $P$, we have $\P(H,P)=\P(H,P')$.
\end{rem}

\begin{lem} \label{lem:spl1}
Let $\Omega $ be a generalised equation periodised with respect to a periodic structure $\langle {\P},R\rangle $. Then for any cycle $\cc_{e_0}$ such that $h(e_0)\notin \P$ and for any solution $H$ of $\Omega$ periodic with respect to a period $P$ such that ${\P}(H,P)=\langle {\P},R\rangle $, one has $H(\cc_{e_0})=P^n$, where $|n|\le 2\rho$. In particular, one can choose a basis $C^{(1)}$ in such a way that for any $\cc\in C^{(1)}$ one has $H(\cc)=P^n$, where $|n|\le 2\rho$.
\end{lem}
\begin{proof}
Let $\cc_{e_0}$ be a cycle such that the edges of $\cc_{e_0}$ are labelled by variables $h_k$, $h_k\notin \P$. Observe that $e_0 = \p_1 \cc_{e_0} \p_2$, where $\p_1$ and $\p_2$ are paths in the tree $T$. Since $e_0 \in {\Gamma_0}$, it follows that the origin and the terminus of the edge $e_0$ lie in the same connected component of the graph $\Gamma_0$ and, consequently, are connected by a path $\s$ in the forest $T_0$. Furthermore, $\p_1$ and $\s\p_2^{-1}$ are paths in the tree $T$ connecting the same vertices; therefore, $\p_1 = \s \p_2^{-1}$. Hence, $\cc_{e_0} = \p_2 \cc'_{e_0} \p_2^{-1}$, where $\cc'_{e_0}$ is a certain cycle based at the vertex $v_\Gamma'$ in the graph $\Gamma_0$.

From the equality $H(\cc_{e_0}) = H(\p_2) H(\cc'_{e_0}) {H(\p_2)}^{-1}$ and by Lemma \ref{2.9}, we get that $P^{n_{e_0}}=P^{n_2}P_1(P_2P_1)^{n_{e_0}'}P_1^{-1}P^{-n_2}$, where $\delta(v_\Gamma')=P_1P_2$, $n_{e_0}=\exp(H(\cc_{e_0}))$, $n_{e_0}'=\exp(H(\cc'_{e_0}))$, $n_2=\exp(H(\p_2))$. Hence $n_{e_0}=n_{e_0}'$ and thus $|H(\cc_{e_0})| = |H(\cc'_{e_0})|$.  From the construction of ${\P}(H, P)$, it follows that the inequality $|H_k| \le 2 |P|$ holds for every item $h_k \not \in {\P}$. Since the cycle $\cc'_{e_0}$ is simple, we have that $|H(\cc_{e_0})| = |H(\cc'_{e_0})|\le 2 \rho |P|$.

In particular, one has that $|\exp(H(\cc_e))| \leq 2 \rho$ for every $e \notin T$, $e \in \Sh$. By Lemma \ref{cl:propc1} one can choose a basis $C^{(1)}$ such that for every $\cc\in C^{(1)}$ either $\cc=\cc_e$ and $|\exp(H(\cc_e))| \leq 2 \rho$, or $H(\cc)=1$.
\end{proof}

\begin{lem} \label{2.10''}
Let $\Omega $ be a  generalised equation periodised with respect to a periodic structure $\langle {\P},R\rangle $. Then the following statements hold.
\begin{enumerate}
    \item \label{spl2} Let $\widetilde{G}$ be a fully residually $G$ quotient of $G_{R(\Omega^\ast)}$, such that $\widetilde{G}$ is discriminated by solutions that satisfy the following condition: for any $\cc\in C^{(1)}$, $H(\cc)=P^n$ and $|n|\le 2\rho$. Then the image of $\langle h(C^{(1)})\rangle$ in $\widetilde{G}$ is either trivial or a cyclic subgroup.

    \item \label{spl3} The system $\Omega ^\ast$ is equivalent to the union of the following two systems of equations:
            $$
            \mathcal{O}=
                \left\{
                \begin{array}{ll}
                u_{ie}^{h(e_i)}=z_{ie},&\hbox{where } e\in T,\, e\in \Sh; 1\leq i\leq m \\
                \left[u_{ie_1},u_{ie_2}\right]=1,&\hbox{where } e_j\in T, \,e_j\in \Sh, \, j=1,2; 1\leq i\leq m \\
                \left[h(\cc_1),h(\cc_2)\right]=1, & \hbox{where } \cc_1,\cc_2\in C^{(1)}\cup C^{(2)}
                \end{array}
                \right.
            $$
        and a system:
            $$
            \Psi \left( \{h(e)\,\mid\, e\in T, e\in \Sh\}, h(C^{(1)}), \bar t,\bar u,\bar z,\cA\right)=1,
            $$
        such that neither $h(e_i)$, $1\le i\le m$, nor $h(C^{(2)})$ occurs in $\Psi$.

    \item \label{spl4} If $\cc$ is a cycle based at the origin of $e_i$, then the transformation $h(e_i)\rightarrow h(\cc)h(e_i)$ which is identical on all the other elements from $\cA\cup\bar x$, extends to a $G$-automorphism of $G_{R(\Omega^\ast)}$.

    \item \label{spl5} If $\cc\in C^{(2)}$ and $\cc'\in C^{(1)} \cup C^{(2)}$, $\cc'\ne \cc$, then the transformation defined by $h(\cc)\rightarrow h(\cc')h(\cc)$, which is identical on all the other elements from $\cA\cup\bar x$, extends to a $G$-automorphism of $G_{R(\Omega^\ast)}$.
\end{enumerate}
\end{lem}
\begin{proof}
We first prove (\ref{spl2}). Since the image of the group $\langle h(C^{(1)})\rangle $ in $G$ is cyclic, and since for all $\cc\in C^{(1)}$ the exponent of periodicity $\exp(H(\cc))$ is bounded by Lemma  \ref{lem:spl1}, for any pair of elements $\cc,\cc'\in C^{(1)}$ and any solution $H$ at least one of the elements of the finite set:
$$
\{H(\cc)^nH(\cc')^{-n'} \, \mid\, 0 \le n,n'\le 2\rho\}
$$
is trivial.

Since $\widetilde{G}$ is $G$-discriminated (by $G$) by $P$-periodic solutions, for any pair of elements $\cc,\cc'\in C^{(1)}$ one of the elements of the finite set
$$
\{h(\cc)^n h(\cc')^{-n'} \, \mid\, 0 \le n,n'\le 2\rho\}
$$
is trivial in $\widetilde G$. Therefore, the subgroup generated by $h(C^{(1)})$ is cyclic.

To prove (\ref{spl3}) we will rewrite the system of equations $\Omega^*$ in the new variables.  It is easy to check that (in the above notation):
\begin{gather}\notag
\begin{split}
\{ h_1,\dots, h_\rho\}=& \{h(e)\, \mid \, h(e)\in \sigma, \sigma \notin \P\}\, \cup\\
&\{h(e) \,\mid\, e\in T_0\} \cup  \{h(e) \,\mid\, e\notin T_0, e\in \Sh\}\, \cup \\
&\{h(e) \,\mid\, e\in T\setminus T_0\} \cup \{h(e) \,\mid\, e\notin T, e\notin \Sh\}.
\end{split}
\end{gather}
From Equation (\ref{2.52}) and the discussion above it follows that
\begin{equation} \label{eq:genset}
\langle h_1,\dots, h_\rho\rangle= \left< \bar t \cup \{h(e) \mid e\in T_0\} \cup \bar u\cup \bar z\cup \{h(e_1),\dots,h(e_m)\}\cup h(C^{(1)})\cup h(C^{(2)})\right>=\langle \bar x\rangle,
\end{equation}
i.e. the sets $h$ and $\bar x$ generate the same free group $F(\Omega)$.

We now rewrite the equations from $\Omega$ in terms of the new set of generators $\bar x$.

We first consider the relations induced by bases which do not belong to $\P$, i.e. basic equations of the form $h(\mu)=h(\Delta(\mu))$, where $\mu$ is a variable base, $\mu\notin \P$ and factor equations $\nu_{j,i,1},\nu_{j,i,2},\nu_{j,i,3}=1$. By construction of the generating set $\bar x$, all the items $h_k$ which appear in these relations belong to the set
$$
\bar t \cup \{h(e) \,\mid\, e\in T_0\} \cup  \{h(e) \,\mid\, e\notin T_0, e\in \Sh\}= \bar t \cup M_1\cup M_2.
$$
The elements of the sets $\bar t$ and $M_1$ are generators in both basis. We now study how elements of $M_2$ rewrite in the new basis.
Let $h(e) = h_k$, $e\notin T_0$, $e\in \Sh$. We have $e=\s\p_2^{-1}\cc_e\p_2$ and
\begin{equation} \label{eq:obt}
h_k = h(\s) h(\p_2)^{-1} h(\cc_e)h(\p_2),
\end{equation}
where $\s$ is a path in $T_0$ and $\p_2$ is a path in $T$ (see proof of Lemma \ref{lem:spl1}). The variables $h(e_i)$, $1 \leq i \leq m$ can occur in the right-hand side of Equation (\ref{eq:obt}) (written in the basis $\bar{x}$) only in $h(\p_2)^{\pm 1}$ and at most once. Moreover, the sign of this occurrence (if it exists) depends only on the orientation of the edge $e_i$ with respect to the root $v_\Gamma$ of the tree $T$. If $\p_2 = \p_2' e_i^{\pm 1}\p_2 ''$, then all the occurrences of the variable $h(e_i)$ in the words $h_k$ written in the basis $\bar{x}$, with $h_k \not\in {\P}$, are contained in the subwords of the form $h(e_i)^{\mp 1} h((\p_2')^{-1}\cc_e\p_2')h(e_i)^{\pm 1}$, i.e. in the subwords of the form $h(e_i)^{\mp 1} h(\cc) h(e_i)^{\pm 1}$, where $\cc$ is a certain cycle in the graph $\Gamma$ based at the origin of the edge $e_i^{\pm 1}$. So the variable $h_k$ rewrites as a word in the generators $\bar u$, $\bar z$, $\{h(e) \,\mid\, e\in T_0 \}$.

Summarising, the basic equations corresponding to variable bases which do not belong to $\P$ and the factor equations rewrite in the new basis as words in $\bar t$, $\{ h(e) \,\mid\, e\in T_0\}$,  $\bar u$ and $\bar z$.

In the new basis, the relations of the form $h(\mu)=h(\Delta(\mu))$, where $\mu\in \P$, are, modulo commutators, words in $h(C^{(1)})$, see Remark \ref{rem:hab}.

Since $\Omega$ is periodised with respect to $\langle \P, R\rangle$, we have
\begin{equation}\label{2.61}
[u_{ie_1},u_{ie_2}]=1 \hbox{ and } [h(\cc_1),h(\cc_2)]=1, \ \cc_1,\cc_2\in C^{(1)}\cup C^{(2)}.
\end{equation}

It follows that the system $\Omega ^\ast$ is equivalent to the union of the following two systems of equations in the new variables, a system
$$
\mathcal{O}=\left\{
\begin{array}{ll}
u_{ie}^{h(e_i)}=z_{ie},&\hbox{where } e\in T,\, e\in \Sh; 1\leq i\leq m \\
\left[u_{ie_1},u_{ie_2}\right]=1,&\hbox{where } e_j\in T, \,e_j\in \Sh, \, j=1,2; 1\leq i\leq m  \\
\left[h(\cc_1),h(\cc_2)\right]=1, & \hbox{where } \cc_1,\cc_2\in C^{(1)}\cup C^{(2)}
\end{array}
\right.
$$
and a system (defined by the equations from $\Omega$):
$$
\Psi \left( \{h(e)\,\mid\, e\in T, h(e)\notin {\P}\}, h(C^{(1)}), \bar t,\bar u,\bar z,\cA\right)=1,
$$
such that neither $h(e_i)$, $1\le i\le m$, nor variables from  $h(C^{(2)})$ occur in $\Psi$.

The transformations $h(e_i)\rightarrow h(\cc)h(e_i)$ (and $h(\cc)\rightarrow h(\cc')h(\cc)$) from statement (\ref{spl4}) (statement (\ref{spl5})) of the lemma extend to an automorphism $\varphi$ of $G_{R(\Omega^*)}$. Indeed, by the universal property of the quotient, the following diagram commutes
$$
\xymatrix{
 G[\bar x] \ar[d] \ar[r]^{\varphi} & G[\bar x] \ar[d]  \\
 G_{R(\Psi\cup \mathcal{O})}  \ar@{-->>}[r]^{\tilde\varphi}  & G_{R(\Psi\cup \mathcal{O}\cup \varphi(\mathcal{O}))}
 }
$$
It is easy to check that $\varphi(\Psi)=\Psi$ and that $\varphi(\mathcal{O}) \subseteq R(\Psi\cup \mathcal{O})$. Therefore, since by statement (\ref{spl3}) of the lemma the system $\Psi\cup \mathcal{O}$ is equivalent to $\Omega^*$, we get that $\tilde\varphi$ is an automorphism of $G_{R(\Omega^*)}$.
\end{proof}

We now introduce a reflexive, transitive relation on the set of solutions of a generalised equation. We use this relation to introduce the notion of a minimal solution with respect to a group of automorphisms of the coordinate group of the generalised equation

\begin{defn}\label{defn:sol<}
Let $G[B] = G\ast F(B)$, where $F(B)$ is a free group with basis $B$, let $\Omega$  be a generalised equation with coefficients from $\cA$. Let $\BB(\Omega)$ be an arbitrary group of $G$-automorphisms of $G_{R(\Omega^*)}$. For solutions  $H^{(1)}$ and $H^{(2)}$ of the generalised equation $\Omega$ in $G[B]$ we write \glossary{name={`$<_{\BB(\Omega)}$'}, description={reflexive, transitive relation on the set of solutions of a generalised equation}, sort=Z}$H^{(1)}< _{\BB(\Omega)} H^{(2)}$ if there exists a $G$-endomorphism $\pi$ of the group $G[B]$ and an automorphism $\sigma\in \BB(\Omega)$  such that the following conditions hold:
\begin{enumerate}
\item \label{it:minsol1} $\pi_{ H^{(2)}}=\sigma\pi _{ H^{(1)}}\pi$;
\item\label{it:minsol3} For any $k$ we have $\itype(H_k^{(1)})=\itype(H_k^{(2)})$ and $\ttype(H_k^{(1)})=\ttype(H_k^{(2)})$.
\end{enumerate}
\end{defn}
Obviously, the relation `$<_{\BB(\Omega )}$' is transitive. We would like to draw the reader's attention to the fact that the relation $H<_{\BB(\Omega)}H'$ does not imply relations on the lengths of the solutions $H$ and $H'$.

\begin{lem}
\label{lem:minsol}
Let $G[B] = G\ast F(B)$ and let $\Omega$  be a generalised equation with coefficients from $\cA\cup B$. Let $\BB(\Omega)$ be an arbitrary  group of $G$-automorphisms of $\factor{G[B][h]}{R(\Omega^*)}$ and let  $H^{(1)}$ and $H^{(2)}$ be two solutions of the generalised equation $\Omega$ such that $H^{(1)}<_{\BB(\Omega)} H^{(2)}$. Then for any word $W(x_1,\dots,x_\rho)\in F(x_1,\dots,x_\rho)$ such that the $W(H^{(2)}_1,\dots,H^{(2)}_\rho)$ is reduced as written, the word $W(H^{(1)}_1,\dots,H^{(1)}_\rho)$ is reduced as written.
\end{lem}
\begin{proof}
Follows from condition (\ref{it:minsol3}) in Definition \ref{defn:sol<}.
\end{proof}

\begin{defn}\label{defn:nfmatrix}
Let $\Omega$ be a generalised equation in $\rho$ variables and let $H$ be a solution of $\Omega$.  Consider a $\rho$-vector $(m_{i})$, with entries in the set $\{1,2\}\times \{1,2\}$ constructed by the solution $H$ in the following way. We set $m_{i}=(j,k)$ if and only if $\itype(H_{i})=j$  and $\ttype(H_{i})=k$.

We call the vector $(m_{i})$ the \index{type vector}($\rho$-)\emph{type vector} of the solution $H$.
\end{defn}

\begin{defn}\label{def:minsol}
A solution $H$ of $\Omega$ is called \index{solution!of a generalised equation!minimal with respect to the group of automorphisms}\emph{minimal with respect to the group of automorphisms $\BB(\Omega)$} if there exist no solutions $H'$ of the generalised equation $\Omega$ so that $H' <_{\BB(\Omega)} H$ and $|H_k'|\leq |H_k|$ for all $k$, $k=1,\dots,\rho$ and $|H_j'|< |H_j'|$ for at least one $j$, $1\le j\le \rho$.

Since the length of a solution $H$ is a positive integer, every strictly descending chain of solutions
$$
H>_{\BB(\Omega)} H^{(1)} >_{\BB(\Omega)}\dots >_{\BB(\Omega)} H^{(k)} >_{\BB(\Omega)}\ldots
$$
is finite. It follows that for every solution $H$ of $\Omega$ there exists a minimal solution $H^+$ such that $H^+ < _{\BB(\Omega)} H$.
\end{defn}

\begin{rem} \label{rem:ms}
Note that given a solution $H$, there exists a minimal solution $H^+$ (perhaps more than one) so that $H^+<_{\BB(\Omega)} H$. Furthermore, there may exist minimal solutions $H^+$ and $H'$ so that $H'\not <_{\BB(\Omega)} H^+$ and $H^+\not <_{\BB(\Omega)} H'$.

Therefore, among all minimal solutions $H^+$ such that $H^+ < _{\BB(\Omega)} H$ one can consider a solution of minimal total length.
\end{rem}

\begin{rem} \label{rem:minsol}
Note that every generalised equation $\Omega$ with coefficients from $\cA$ can be considered as a generalised equation $\Omega'$ with coefficients from $\cA \cup B$ for \emph{any} finite set $B$. Therefore, any solution $H$ of $\Omega$ induces a solution $H'$ of $\Omega'$ such that the following diagram commutes:
$$
\xymatrix@C3em{
 G_{R(\Omega^*)}  \, \ar@{^{(}->}[r] \ar[d]_{\pi_H}  & \,\factor{G[B][h]}{R(\Omega^\ast)} \ar[d]^{\pi_{H'}}
                                                                             \\
                   G \,   \ar@{^{(}->}[r] & \,G[B]
}
$$
A solution $H$ of $\Omega$ is minimal if \emph{any} induced solution $H'$ is.

There is a subtlety with the type constraints: the type constraints are to be satisfied only when they make sense. In the case that a component $H_i$ of the solution $H$ ends (begins) with a letter from $B$ (rather than from $G$), we do not require that $H_i$ satisfy the terminal (initial) type constraint. Alternatively, one can assume that the letters from $B$ have both type $1$ and $2$.
\end{rem}
The reason for extending the generating set from $\cA$ to $\cA\cup B$ in the definition above, becomes clear in the proof of Lemma \ref{lem:23-1.5}
\bigskip

\begin{lem} \label{lem:2.1}
Let the generalised equation $\Omega_1$ be obtained from the generalised equation $\Omega$ by one of the elementary transformation $\ET 1-\ET 5$, i.e. $\Omega_1\in \ET i(\Omega)$ for some $i=1,\dots, 5$ and assume that $G_{R(\Omega_1^*)}$ is isomorphic to $G_{R(\Omega^*)}$. Let $H$ be a solution of $\Omega$ and $H^{(1)}$ be a solution of $\Omega_1$ so that the following diagram commutes
$$
\xymatrix@C3em{
 G_{R(\Omega^\ast)}  \ar[rd]_{\pi_H} \ar[rr]^{\theta}  &   &G_{R(\Omega_1^\ast )} \ar[ld]^{\pi_{H^{(1)}}} \\
                               &  G &}
$$
If $H$ is a minimal solution of $\Omega$ with respect to a group of automorphisms $\BB$ of $G_{R(\Omega^*)}$, then $H^{(1)}$ is a minimal solution of $\Omega_1$ with respect to the group of automorphisms $\theta^{-1}\BB\theta$.
\end{lem}
\begin{proof}
Assume the converse, i.e. $H^{(1)}$ is not minimal with respect to $\theta^{-1}\BB\theta$. Then there exists a solution ${H^{(1)}}^+$ of $\Omega_1$ so that ${H^{(1)}}^+<_{\theta^{-1}\BB\theta} H^{(1)}$,  $|{H^{(1)}}^+_k|\le |H^{(1)}_k|$ for all $k$ and $|{H^{(1)}}^+_j|< |H^{(1)}_j|$ for some $j$.

Let $H^+$ be a solution of the system of equations $\Omega^*$ so that $\pi_{H^+}=\theta\pi_{{H^{(1)}}^+}$. As described in the definition of elementary transformations, $\theta(h_k)=h_{j_1}^{(1)}\cdots h_{j_s}^{(1)}$ and
\begin{equation} \label{2.16}
H_k\doteq H_{j_1}^{(1)}\cdots H_{j_s}^{(1)},
\end{equation}
where $h_{j_i}^{(1)}\in h^{(1)}$ for all $i$, is an item of $\Omega_1$.

Since ${H^{(1)}}^+<_{\theta^{-1}\BB\theta} H^{(1)}$, by condition (\ref{it:minsol3}) from Definition \ref{defn:sol<}, we have
\begin{equation} \label{2.17}
H_k^+\doteq {H_{j_1}^{(1)}}^+\cdots {H_{j_s}^{(1)}}^+
\end{equation}
and thus $H^+$ is a solution of $\Omega$ and condition (\ref{it:minsol3}) from Definition \ref{defn:sol<} holds for the pair $H^+$ and $H$.

Furthermore, by condition  (\ref{it:minsol1}) from Definition \ref{defn:sol<}, we have $\pi_{H^{(1)}}=\theta^{-1}\psi\theta\pi_{{H^{(1)}}^+}$, where $\psi \in \BB$, hence $\pi_H=\psi\pi_{H^+}$, and thus condition (\ref{it:minsol1}) from Definition \ref{defn:sol<} holds for the pair $H^+$ and $H$. We thereby have proven that $H^+<_{\BB}H$.

Finally, from Equations (\ref{2.16}), (\ref{2.17}) and the fact that ${H^{(1)}}^+$ is minimal, we get that $|H_k^+|\le |H_k|$ for all $k$ and $|H_j^+|<|H_j|$ for some $j$, contradicting the minimality of $H$.
\end{proof}

\begin{defn} \label{defn:singreg}
Let $\Omega$ be a generalised equation and let $\langle \P,R\rangle$ be a connected periodic structure on $\Omega$. We say that the generalised equation $\Omega$ is \index{generalised equation!singular with respect to a periodic structure}\emph{singular (of type (a), (b) or (c)) with respect to the periodic structure $\langle \P,R\rangle$} if one of the following conditions holds
\begin{itemize}
\item[(a)] The generalised equation $\Omega$ is not periodised with respect to the periodic structure $\langle \P,R\rangle$.
\item[(b)] The set $C^{(2)}$ has more than one element.
\item[(c)] The set $C^{(2)}$ has exactly one element, and (in the above notation) there exists a cycle $\cc_{e_0}\in \langle C^{(1)}\rangle$, $h(e_0)\notin \P$ such that $h(\cc_{e_0})\ne 1$ in $G_{R(\Omega^\ast)}$.
\end{itemize}

Otherwise, we say that $\Omega$ is \index{generalised equation!regular with respect to a periodic structure}\emph{regular with respect to the periodic structure $\langle \P,R\rangle$}. In particular if $\Omega$ regular with respect to the periodic structure  $\langle \P,R\rangle$ then $\Omega$ is periodised.

When no confusion arises, instead of saying that $\Omega$ is singular (or regular) with respect to the periodic structure $\langle \P,R\rangle$ we say that the periodic structure $\langle \P,R\rangle$ is \index{periodic structure!singular}\emph{singular} (of type (a), (b) or (c)) (or \index{periodic structure!regular}\emph{regular}).
\end{defn}

\begin{defn} \label{defn:AA}
Let $\Omega$ be a generalised equation and let  $\langle \P, R\rangle$ be a periodic structure on $\Omega$. If $\Omega$ is  singular of type (a) with respect to $\langle \P, R\rangle$, then we define the group \glossary{name={$\AA(\Omega)$}, description={finitely generated group of automorphisms of $G_{R(\Omega^\ast)}$ associated with a periodic structure on $\Omega$}, sort=A}$\AA(\Omega)$ of automorphisms of $G_{R(\Omega^\ast)}$ to be trivial.

Otherwise, i.e. if $\langle \P, R\rangle $ is singular of types (b) or (c) or regular, we set  $\AA(\Omega)$ to be the group of automorphisms of $G_{R(\Omega^\ast)}$ generated by the automorphisms described in statements (\ref{spl4}) and (\ref{spl5}) of Lemma \ref{2.10''}. Note that by definition, the group $\AA(\Omega)$ is finitely generated.
\end{defn}

\subsection{Singular Case}

The next lemma states that  if a generalised equation $\Omega$ is singular with respect to a periodic structure, then one can construct finitely many proper quotients of the coordinate group $G_{R(\Omega^*)}$, such that for every periodic solution there exists an $\AA(\Omega)$-automorphic image such that this image is in fact a solution of one of the quotients constructed. In other words, every solution of the generalised equation can be obtained as a composition of an automorphism from $\AA(\Omega)$ and a solution of a proper generalised equation.

\begin{lem} \label{lem:23-1}
Let $\Omega$ be a formally consistent generalised equation without boundary connections, singular with respect to the periodic structure $\langle \P, R\rangle $. Then there exists a finite family of cycles $\cc_1,\ldots,\cc_r$ in the graph $\Gamma$ such that:
\begin{enumerate}
\item \label{it:23-11} $h(\cc_i)\ne 1$, $i=1,\ldots, r$ in $G_{R(\Omega^\ast)}$;
\item \label{it:23-12}for any solution $H$ of $\Omega$ such that $H$ is periodic with respect to a period $P$ and ${\P}(H,P)=\langle{\P},R\rangle$, there exists an $\AA(\Omega)$-automorphic image $H^+$ of $H$ such that $H^+(\cc_i)=1$ for some $1\le i\le r$;
\item\label{it:23-13} for any $h_k \notin \P$, $H_k\doteq H^+_k$; and for any $h_k \in \P$ if $H_k \doteq P_1P^{n_k}P_2$, then $H^+_k\doteq P_1P^{n^+_k}P_2$, where $\delta(k)=P_1P_2$, $n_k, n^+_k \in \Z$.
\end{enumerate}
\end{lem}
\begin{proof}
Suppose that $\Omega$ is not periodised with respect to the periodic structure $\langle \P, R\rangle$, i.e. there exist two cycles $\cc,\cc'$ so that $[h(\cc),h(\cc')]\ne 1$ in $G_{R(\Omega^*)}$. Since, by Lemma \ref{2.9}, for every $P$-periodic solution $H$, one has that $H([\cc,\cc'])=1$, it suffices to take as the set of cycles $\{\cc_1,\ldots,\cc_r\}$ the set consisting of the commutator $[\cc,\cc']$.

We further assume that $\Omega$ is periodised with respect to the periodic structure  $\langle \P, R\rangle$.

Suppose that $|C^{(2)}|\ge 2$. Let $H$ be a $P$-periodic solution of a generalised equation $\Omega $, such that ${\P} (H, P)=\langle {\P}, R\rangle$. Let $\cc_1,\cc_2\in C^{(2)}$. By Lemma \ref{2.9}, $H(\cc_1)=P^{n_1}$, $H(\cc_2)=P^{n_2}$.

Without loss of generality we may assume that $n_1\le n_2$. Write $n_2=tn_1+r$, where either $r=0$ or $|r|<|n_1|$. Applying the canonical automorphism $\varphi:h(\cc_2)\mapsto {h({\cc_1})}^{-t}h(\cc_2)$, we get $H(\varphi(h(\cc_2)))=P^r$, hence the exponent of periodicity $\exp(H(\varphi(h(\cc_2))))$ is reduced.

Applying the Euclidean algorithm, we get that there exists an automorphism $\psi$ from $\AA(\Omega)$ such that $H(\psi(h(\cc_2)))=1$ or $H(\psi(h(\cc_1)))=1$. Set $H^+=\psi(H)$. Then the set $\{\cc_1,\cc_2\}$ satisfies conditions (\ref{it:23-11}) and (\ref{it:23-12}) of the lemma.

In the notation of Lemma \ref{2.10''}, the automorphism $\psi$ is identical on all the elements of $\bar x$ except for $h(C^{(2)})$, hence, in particular, it is identical on $h(e)$ such that $e \in T$ or $e\notin\P$, and on $h(C^{(1)})$, i.e. $H^+(e)=H(e)$, for any $e\in T$ or $e\notin \P$ and $H^+(\cc^{(1)})=H(\cc^{(1)})$ for any $\cc^{(1)}\in C^{(1)}$.

If $h_k=h(e)$, $e \notin T$, $h(e) \in \P$, then $h_k=h(\p_1)h(\cc_e)h(\p_2)$, where $\p_1,\p_2$ are paths in $T$. Therefore, $H_k^+=H^+(e)=H^+(\p_1)H^+(\cc_e)H^+(\p_2)=H(\p_1)H^+(\cc_e)H(\p_2)$. Since $h(\cc_e)$ lies in the subgroup generated by $h(C^{(1)})$ and $h(C^{(2)})$, then $H^+(\cc_e)$ and $H(\cc_e)$ lie in the cyclic group generated by $P$. This proves statement (\ref{it:23-13}) of the lemma.

Suppose that $C^{(2)}=\{\cc^{(2)}\}$, since the periodic structure is singular, there exists a cycle $\cc\in \langle C^{(1)}\rangle$ such that $h(\cc)\ne 1$ in $G_{R(\Omega^\ast)}$  and such that the edges of $\cc$ are labelled by variables $h_k$, $h_k\notin \P$. Let $n$ be the number of edges in $\cc$.

We define the set of cycles $\{\cc_1,\ldots, \cc_r\}$  to be
$$
\left\{(\cc)^i(\cc^{(2)})^j\,\mid\,  \hbox{$i$ and $j$ are not simultaneously zero}, \, |i|,|j|\leq 2n\right\}.
$$

We now show that if ${h(\cc)}^i {h(\cc^{(2)})}^j=1$ in $G_{R(\Omega ^\ast)}$, then $i=j=0$. Suppose ${h(\cc)}^i {h(\cc^{(2)})}^j=1$. Let $\sigma _0$ be a generator of the group of automorphisms $\AA(\Omega)$ such that $\sigma _0(h(\cc))=h(\cc)$ and $\sigma_0(h(\cc^{(2)}))=h(\cc)h(\cc^{(2)})$. Hence ${h(\cc)}^{i+j}{h(\cc^{(2)})}^j=1$ in $G_{R(\Omega ^\ast)}$ and ${h(\cc)}^{j}=1$. This implies that either $h(\cc)=1$ (see proof of part (\ref{spl3}) of Lemma \ref{2.10''}) or $j=0$. Since $h(\cc)\ne 1$, we have $j=0$. Similarly, we get $i=0$.

Let $H$ be a $P$-periodic solution of the generalised equation $\Omega$, such that ${\P}(H, P)=\langle {\P}, R \rangle$. Since all the edges in the cycle $\cc$ are labelled by items $h_k$, $h_k\notin \P$, by Lemma \ref{lem:spl1} we have $H(\cc)=P^{n_0}$, $|n_0| \leq 2n$. Let $H(\cc^{(2)})= P^m$.

If $n_0 = 0$, we can take $\sigma = 1$, $H^+ = H$ and  the set of cycles $\{\cc\}$.

Let $n_0 \neq 0$, $m = tn_0 + m'$, $t\in \Z$ and $|m'| \leq 2 n$. Let $\sigma= \sigma _0^{-t}$ and define $H^+$ to be the image of $H$ under $\sigma$. Set $\cc=(\cc)^{-m'} (\cc^{(2)})^{n_0}$, then $H^+ (\cc) = P^{n_0}$, $H^+(\cc^{(2)}) = P^{m'}$ and $H^+(\cc)=1$.

An analogous argument to the one given in the case $|C^{(2)}| \geq 2$ shows that $H^+_k=H_k$ if $h_k \notin \P$, and $H^+_k=P_1P^{n^+_k}P_2$ if $H_k=P_1P^{n_k}P_2$ and $h_k \in \P$.

\end{proof}

\subsection{Regular Case}

The aim of this section is to prove that if a generalised equation $\Omega$ is regular with respect to a periodic structure, then periodic solutions of $\Omega$ minimal with respect to the group of automorphisms $\AA(\Omega)$ have bounded exponent of periodicity. In other words, the length of such a solution is bounded above by a function of $\Omega$ and the length $|P|$ of its period $P$.

\begin{lem}\label{lem:23-1.5}
Let $\Omega $ be a generalised equation with no boundary connections, periodised with respect to a connected periodic structure $\langle \P, R\rangle$. Let $H$ be a periodic solution of $\Omega$ such that $\P(H,P)=\langle \P, R\rangle$ and minimal with respect to the trivial group of automorphisms. Then, either for all $k$, $1\le k\le \rho$ we have
\begin{equation} \label{2.62}
|H_k| \le 2 \rho |P|,
\end{equation}
or there exists a cycle $\cc\in \pi_1(\Gamma,v_\Gamma)$  so that $H(\cc)=P^n$, where $1\le n\le 2\rho$.
\end{lem}
\begin{proof}
Let $H$ be a $P$-periodic solution of the generalised equation $\Omega$ minimal with respect to the trivial group of automorphisms.

Suppose that there exists a variable $h_k \in {\P}$ such that $|H_k| > 2 \rho |P|$. We then prove that there exists a cycle $\cc\in \pi_1(\Gamma,v_\Gamma)$  so that $H(\cc)=P^n$, where $1\le n\le 2\rho$.

Construct a chain
\begin{equation} \label{2.63}
(\Omega, {H})=(\Omega_{v_0}, {H}^{(0)}) \to (\Omega_{v_1}, {H}^{(1)}) \to \cdots \to (\Omega_{v_t}, {H}^{(t)}),
\end{equation}
in which for all $i$, $\Omega_{v_{i+1}}$ is obtained from $\Omega_{v_{i}}$ using $\ET 5$: by $\mu$-tying a free boundary that intersects a certain base $\mu \in {\P}$. Chain (\ref{2.63}) is constructed once all the boundaries intersecting bases $\mu$ from ${\P}$ are $\mu$-tied. This chain is finite, since, by the definition, boundaries that are introduced when applying $\ET 5$ are not free.

By construction, the generalised equations $\Omega_{v_{i}}$'s in (\ref{2.63}) have boundary connections. Our definition of periodic structure is given for generalised equations without boundary connections, see Definition \ref{above}. For this reason we define $\Omega_{v_{i}'}$ to be the generalised equation obtained from $\Omega_{v_i}$ by omitting all boundary connections (here we do not apply $\D 3$). It is obvious that the solution ${H}^{(i)}$ of  $\Omega_{v_i}$ is also a solution of the generalised equation $\Omega_{v_{i}'}$ and is periodic with respect to the period $P$. Denote by $\langle {\P}_i, R_i \rangle$ the periodic structure ${\P}({H}^{(i)}, P)$ on the generalised equation $\Omega_{v_i'}$ restricted to the closed sections from ${\P}$, and by $\Gamma^{(i)}$ the corresponding graph.

If $(p, \mu, q)$, $\mu \in {\P}$, is a boundary connection of the generalised equation $\Omega_{v_i}$, $i=1,\dots,t$, then $\delta(p) = \delta(q)$. Therefore, all the graphs $\Gamma^{(0)}$, $\Gamma^{(1)}, \ldots, \Gamma^{(t)}$ have the same set of vertices, whose cardinality does not exceed $\rho$. By Lemma \ref{lem:2.1}, the solution ${H}^{(t)}$ of the generalised equation $\Omega_{v_t}$ is also minimal with respect to the trivial group of automorphisms.

Suppose that for some variable $h_l$ belonging to a section from ${\P}$ the inequality $|H_l^{(t)}| > 2 |P|$ holds.

Let $\mathcal{H}=\left\{h_i\in \sigma \mid \sigma \in \P\hbox{ and } H_i^{(t)}\doteq {{H}_l^{(t)}}^{\pm 1}\right\}$.  Consider the group $G[u]=G\ast\langle u\rangle$, where $u$ is a new letter. In the solution ${H}^{(t)}$, replace all the components ${H_i}^{(t)}$ such that $H_i^{(t)}\doteq {{H}_l^{(t)}}$ or $H_i^{(t)}\doteq {{H}_l^{(t)}}^{- 1}$ by the letter $u$ or $u^{-1}$, correspondingly (see Definition \ref{def:minsol}). Denote the resulting $\rho_{\Omega_{v_t}}$-tuple of words by ${H^{(t)}}'$.

We show that ${H^{(t)}}'$ is in fact a solution of $\Omega_{v_t}$ (considered as a generalised equation with coefficients from $\cA \cup \{u\}$, see Remark \ref{rem:minsol}). Obviously, every component of ${H^{(t)}}'$ is non-empty and written in the normal form. Since in the generalised equation $\Omega_{v_t}$ all the boundaries from ${\P}$ are $\mu$-tied, $\mu\in \P$, there is a one-to-one correspondence between the items that belong to $\mu$ and the items that belong to $\Delta(\mu)$. Therefore, ${H^{(t)}}'$ satisfies all basic equations $h(\mu)=h(\Delta(\mu))$, $\mu \in \P$ and all the boundary equations of the generalised equation $\Omega_{v_t}$.

If, on the other hand, $\mu \not \in {\P}$, then for every item $h_k \in \mu$ of the generalised equation $\Omega$ lying on a closed section from ${\P}$, we have $h_k \not \in {\P}$ and, consequently, $|H_k| \leq 2 |P|$. In particular, this inequality holds for every item $h_k \in \mu$, $\mu \notin \P$ of the generalised equation $\Omega_{v_t}$.  Therefore, such items have not been replaced in the solution ${H}^{(t)}$, thus ${H^{(t)}}'$ satisfies all basic equations $h(\mu)=h(\Delta(\mu))$, $\mu \notin \P$.

Let $\pi:G[u]\to G[u]$ be a map that sends $u$ to ${H}_l^{(t)}$ and fixes $G$. It is easy to see that $\pi$ is a $G$-endomorphism, which contradicts the minimality of the solution ${H^{(t)}}$. Indeed, one has $\pi_{H^{(t)}}=\pi_{{H^{(t)}}'}\pi$. By construction, $u$ does not occur in the solution ${H}^{(t)}$ and condition (\ref{it:minsol3}) of Definition \ref{defn:sol<} holds. Therefore, ${H^{(t)}}'<_{\{1\}}{H^{(t)}}$. Obviously, $|{H_i^{(t)}}'|=|{H_i^{(t)}}|$ for all $i\ne l$ and $1=|u|=|{H_l^{(t)}}'|<|{H_l^{(t)}}|$. We thereby have shown that $|H_l^{(t)}| \leq 2 |P|$, if $h_l$ belongs to a closed section from ${\P}$.

In the construction of chain (\ref{2.63}) we introduced new boundaries, so every item from $\Omega_{v_0}$ can be expressed as a product of items from $\Omega_{v_t'}$. Consequently, since $|H_l^{(t)}| \leq 2 |P|$ for every $l$, if the component $H_k$ of the solution $H$ of the generalised equation $\Omega$ does not satisfy inequality (\ref{2.62}), then $h_k$ is a product of at least $\rho+1$ distinct items of $\Omega_{v_t'}$, $h_k=h^{(t)}_{s}\cdots h^{(t)}_{s+\varrho}$, where $\varrho\ge \rho+1$.

Since the graph $\Gamma^{(t)}$ contains at most $\rho$ vertices, there exist boundaries $l,l'\in \{s,\dots, s+\rho+1\}$, $l<l'$ so that  $\delta({l})=\delta({l'})$. The word $h[l,l']$ is a label of a cycle $\cc_t$ of the graph $\Gamma^{(t)}$ for which $0<|H(\cc_t)|\le 2 \rho |P|$.

Recall that by $\pi(v_i,v_j)$ we denote the homomorphism $G_{R({\Omega_{v_i}}^\ast)} \to G_{R({\Omega_{v_j}}^\ast)}$ induced by the elementary transformations. It remains to prove the existence of a cycle $\cc_0$ of the graph $\Gamma$ for which $\pi(v_\Gamma,v_t) (h(\cc_0))=h(\cc_t)$. To do this, it suffices to show that for every path $\p_{i+1} : v \rightarrow v'$ in the graph $\Gamma^{(i+1)}$ there exists a path $\p_i : v \rightarrow v'$ in the graph $\Gamma^{(i)}$ such that $\pi(v_i, v_{i+1}) (h(\p_i)) = h(\p_{i+1})$. In turn, it suffices to verify the latter statement in the case when $\p_{i+1}$ is an edge $e$.

The generalised equation $\Omega_{v_{i+1}}$ is obtained from $\Omega_{v_i}$ by $\mu$-tying a boundary $p$. Below we use the notation from the definition of the elementary transformation $\ET 5$.
\begin{enumerate}
    \item Either we introduce a boundary connection $(p,\mu,q)$. In this case every variable $h(e)$ of the generalised equation $\Omega_{v_{i+1}}$ is also a variable of the generalised equation $\Omega_{v_i}$ and the statement is obvious;
    \item Or we introduce a new boundary $q'$ between the boundaries $q$ and $q+1$, and a boundary connection $(p,\mu, q')$. Using the boundary equations we get
\begin{gather}\label{eq:theseform}
\begin{split}
&\pi(v_i,v_{i+1})^{-1}(h_{q})= {h[\alpha (\Delta (\mu )),q]}^{-1}h[\alpha (\Delta(\mu) ),q'] ={h[\alpha (\Delta (\mu )),q]}^{-1}h[\alpha (\mu ),p],\\
&\pi (v_i,v_{i+1})^{-1}(h_{q'})= {h[\alpha(\Delta(\mu) ),q']}^{-1} h[\alpha (\Delta (\mu )),q+1]={h[\alpha(\mu ),p]}^{-1} h[\alpha (\Delta (\mu )),q+1].
\end{split}
\end{gather}
Moreover, $\pi(v_i,v_{i+1})^{-1}(h(e))=h(e)$ for any other variable. Notice that since $(\alpha(\mu)) = (\alpha(\Delta(\mu)))$, the right-hand sides of Equations (\ref{eq:theseform}) are labels of paths in $\Gamma^{(i)}$.
\end{enumerate}
Thus, we have deduced that there exists a cycle $\cc\in \Gamma$ so that $1\le\exp(H(\cc))\le 2\rho$.
\end{proof}

\begin{lem} \label{lem:23-2}
Let $\Omega $ be a generalised equation with no boundary connections. Suppose that $\Omega$ is regular with respect to a periodic structure $\langle \P, R\rangle$. Then there exists a computable function $f_0 (\Omega, {\P},R)$ such that, for every $P$-periodic solution $H$ of $\Omega$ such that $\P(H,P)=\langle \P,R\rangle$ and such that $H$ is minimal with respect to $\AA(\Omega)$, the following inequality holds
$$
|H_k| \leq f_{0} (\Omega, {\P},R) \cdot |P| \ \hbox{ for every $k$}.
$$
\end{lem}
\begin{proof}
Let $H$ be a $P$-periodic minimal solution with respect to the group of automorphisms $\AA(\Omega)$. Notice that $H$ is also minimal with respect to the trivial group of automorphisms.

We first prove that for any regular periodic structure and any $P$-periodic solution $H$ of $\Omega$ minimal with respect to the group of automorphisms $\AA(\Omega)$, the exponent of periodicity of every simple cycle $\cc_e$, $e\notin T$ is bounded by a certain computable function $g_1 (\Omega, {\P}, R)$.

If for every $i$ we have $|H_i|\le 2\rho |P|$, the statement follows. We further assume that there exists $i$ such that  $|H_i|> 2\rho |P|$.

Since the periodic structure is regular, either $|C^{(2)}|=1$ and for all $e \not \in T$ so that $e \in \Sh$ we have $H(\cc_e)=1$, or $|C^{(2)}|=0$.

Assume first that, $C^{(2)}=\{\cc^{(2)}\}$. Since the periodic structure $\langle \P,R\rangle$ is regular, it follows that $H(\cc_e)=1$ for all $e\notin T$, $h(e)\notin \P$. Since ${H}$ is a solution of $\Omega$, it follows that $\exp(H(b_\mu))=0$, $\mu \in {\P}$. Therefore, $\exp(H(\cc))=0$ for all $\cc$ such that $\tilde{\cc}\in\widetilde{B}$. By (\ref{2.54}), we have that $[\widetilde{Z}_1:\widetilde{B}]<\infty$, hence we get that $\exp(H(\cc))=0$ for all $\cc$ such that $\tilde{\cc}\in \widetilde{Z}_1$. Consequently, since the only non-trivial cycle at $v_\Gamma$ is $\cc^{(2)}$, by Lemma \ref{lem:23-1.5}, we get that $|\exp(H({\cc}^{(2)}))| \le 2 \rho$. Using factorisation (\ref{2.54}), one can effectively express $\tilde{\cc}_e$, $e \not \in T$ in terms of the elements of the basis:
$$
\tilde{\cc}_e = n_e \tilde{\cc}^{(2)} + \tilde{z}_e^{(1)},\  \tilde{z}_e^{(1)} \in \widetilde{Z}_1.
$$
Hence $|\exp(H(\cc_e))| = |n_e \exp (H(\cc^{(2)}))| \leq 2 \rho n_e$, and we finally obtain
\begin{equation} \label{2.64}
|\exp(H(\cc_e))| \le g_1 (\Omega, {\P}, R),
\end{equation}
where $g_1$ is a certain computable function.

Suppose next that $|C^{(2)}| = 0$, i.e. $\widetilde{Z} = \widetilde{Z}_1$. As we have already seen in the proof of Lemma~\ref{lem:spl1}, $|H(\cc_e)| \leq 2 \rho |P|$, where  $e \not \in T$, $h(e) \not \in {\P}$. Hence, $|\exp(H(\cc_e))| \leq 2 \rho$ for all $e$ such that $h(e) \not \in {\P}$. Since $(\widetilde{Z}_{1} : \widetilde{B}) < \infty$, for every $e_0 \not \in T$ one can effectively construct the following equality
$$
n_{e_0} {\tilde{\cc}}_{e_0} = \sum\limits_{h(e) \notin \P} n_e \tilde{\cc}_e + \sum\limits_{\mu \in \P} n_\mu \tilde{b}_\mu.
$$
Hence,
$$
|\exp(H({\cc}_{e_0}))| \leq |n_{e_0}\cdot\exp( H({\cc}_{e_0}))|\leq \sum\limits_{h(e)\notin {\P}}
|n_e \exp(H({\cc}_e))| \leq 2 \rho \cdot \sum\limits_{h(e) \not \in {\P}} |n_e|.
$$
Thus, $|\exp(H({\cc}_{e_0}))|\le g_2(\Omega,\P,R)$.

We now address the statement of the lemma. The way we proceed is as follows. For a $P$-periodic solution $H$ minimal with respect to the group of automorphisms $\AA(\Omega)$, we show that the vector that consists of exponents of periodicity of each of the components $H_k$ of $H$ is bounded by a minimal solution of a linear system of equations whose  coefficients depend only on the generalised equation. Since by Lemma 1.1 from \cite{Makanin}, the components of a minimal solution of a linear system of equations are bounded above by a recursive function of the coefficients of the system, we then get a recursive bound on the exponents of periodicity of the components of the solution $H$.

Let $\delta(k) = P_1^{(k)} P_2^{(k)}$. Denote by $t(\cc, h_k)$ the algebraic sum of occurrences of the edge with the label $h_k$ in the cycle $\cc$, (i.e. edges with different orientation contribute with different signs). For every item $h_k$ that belongs to a closed section from ${\P}$ one can write $$
H_k = P_2^{(k)} P^{n_k} P_1^{(k+1)}.
$$
Note that in the case that $h_k\in \P$, the above equality is graphical. However, in the case that $h_k\notin \P$ and $H_k$ is a subword of $P^{\pm 1}$ there is cancellation. Direct calculations show that
\begin{equation} \label{2.66}
H(\cc) = P^{\left(\sum\limits_k t(\cc, h_k)(n_k+1)\right) -1}.
\end{equation}
By Lemma \ref{lem:23-1.5}, $e_0 \not \in T$ can be chosen in such a way that $\exp(H({\cc}_{e_0})) \neq 0$. Let $n_k = | \exp(\tilde{\cc}_{e_0})| m_k + r_k$, where $0 < r_k \leq |
\exp (\tilde{\cc}_{e_0})|$. Equation (\ref{2.66}) implies that the vector $\{ m_k \mid h_k \in {\P} \}$ is a solution of the following system of Diophantine equations in variables $\{ z_k \mid h_k \in {\P} \}$:
\begin{equation} \label{2.67}
\sum\limits_{h_k \in {\P}} t(\cc_e, h_k)(|\exp(H({\cc}_{e_0}))| z_k + r_k +1) + \sum\limits_{h_k \not \in {\P}} t(\cc_e, h_k)(n_k +1) -1 = \exp(H({\cc}_e)), \ e \not \in T.
\end{equation}
Note that the number of variables of the system (\ref{2.67}) is bounded. Furthermore, as we have proven above, free terms $\exp(H({\cc}_e))$ of this system are also bounded above, and so are the coefficients $|n_k|\leq 2$ for $h_k\not\in {\P}$.

A solution $\{m_k\}$ of a system of linear Diophantine equations is called {\em minimal}, see \cite{Makanin}, if $m_k \geq 0$ and there is no other solution $\{m_k^+\}$ such that $0 \leq m_k^+ \leq m_k$ for all $k$, and at least one of the inequalities $m_k^+ \leq m_k$ is strict. Let us verify that the solution $\{m_k \mid h_k \in {\P} \}$ of system (\ref{2.67}) is minimal.

Indeed, let $\{ m_k^+ \}$ be another solution of system (\ref{2.67}) such that $0 \leq m_k^+ \leq m_k$ for all $k$, and at least for one $k$ the inequality is strict. Let $n_k^+ = |\exp(H({\cc}_{e_0}))| m_k^+ + r_k$. Define a vector $H^+$ as follows: $H_k^+ = H_k$ if $h_k \not \in {\P}$, and $H_k^+ = P_2^{(k)} P^{n_k^+}P_1^{(k+1)}$ if $h_k \in {\P}$.

We now show that $H^+$ is a solution of the generalised equation which can be obtained from $H$ by an automorphism from $\AA(\Omega)$.

The vector $H^+$ satisfies all the coefficient, factor and basic equations of $\Omega$ and type constraints. Indeed, since $\{m_k^+\}$ is a solution of system (\ref{2.67}), $H^+(\cc_e) = P^{\exp(H({\cc}_e))} = H(\cc_e)$. Therefore, for every cycle $\cc$ we have $H^+(\cc) = H(\cc)$ and, in particular, $H^+(b_\mu) = H(b_\mu) = 1$. Thus the vector $H^+$ is a solution of the system $\Omega^\ast$.

By construction every component of the solution $H^+$ is non-empty, has the same type as the respective component of $H$, and the words $H^+(\mu)$, $H^+(\Delta(\mu))$ are reduced as written. On the other hand, for every $\mu$ we have
$$
H^+(\mu){H^+(\Delta(\mu))}^{-1} = 1.
$$
It follows that
$$
H^+ (\mu) = H^+ (\Delta(\mu)).
$$
Thus, $H^+$ is a solution of the generalised equation $\Omega$.

Denote by $\delta_{ie_0}$ an element from the group of automorphisms $\AA(\Omega)$ defined in the following way. For every $e_i\in T\setminus T_0$, $i=1,\dots,m$ set
$$
\delta_{ie_0}: h(e_j) \mapsto
\left\{
  \begin{array}{ll}
    h(\p(v_\Gamma, v_i)^{-1} \cc_{e_0} \p(v_\Gamma, v_i))h(e_i), & \hbox{for } j=i; \\
    h(e_j), & \hbox{for } j\ne i.
  \end{array}
\right.
$$
Therefore, if $\pi_{{H}'} = \delta_{ie_0}\pi_{H}$ and $h(e_i) = h_k \in {\P}$, then $H_k' =P_2^{(k)} P^{n_k + \exp(H(\cc_{{e}_0}))} P_1^{(k+1)}$, and all the other components of $H'$ are the same as in $H$, $H'_l=H_l$, $l\ne k$. Denote by $\delta = \prod\limits_{i=1}^{m} \delta_{ie_0}^{\Delta_i}$, where $h(e_i)=h_{k_i}$, $\Delta_i = (m^+_{k_i} - m_{k_i}) \cdot {\rm \sign}(\exp(H(\cc_{{e}_0})))$. Let us
verify the equality
\begin{equation} \label{2.69}
\pi_{{H}^+} = \pi_{{H}} \delta.
\end{equation}

Let $\pi_{{H}^{(1)}} =\delta\pi_{{H}}$. Then, by construction, $H_k^{(1)} = P_2^{(k)} P^{m_k^+} P_1^{(k+1)} = H_k^+$ for all $k$ such that $h_k$ is the label of an edge from $T \setminus T_0$. If the edge labelled by $h_k$ belongs to $T_0$, or $h_k$ does not belong to a closed section from ${\P}$, then $h_k \not \in {\P}$ and $H_k^{(1)}= H_k = H_k^+$. Finally, for every $e \not \in T$ we have $H^{(1)}(\cc_e) = H(\cc_e) = H^+(\cc_e)$. As $\cc_e = \p_1 e \p_2$, where $\p_1$, $\p_2$ are paths in the tree $T$, and for every item $h_k$ which labels an edge from $T$, the equality $H_k^{(1)} = H_k^+$ has already been established, so it follows that $H^{(1)}(e) = H^+(e)$. This proves (\ref{2.69}).

Since $H^+$ is a solution of $\Omega$, since by construction $H^+$ satisfies condition (\ref{it:minsol3}) from Definition \ref{defn:sol<} and by (\ref{2.69}), it follows that ${H}^+ <_{\AA(\Omega)} {H}$. We arrive to a contradiction with the minimality of the solution ${H}$. Consequently, the solution $\{m_k \mid h_k \in {\P} \}$ of system (\ref{2.67}) is minimal.

Lemma 1.1 from \cite{Makanin} states that the components of the  minimal solution $\{m_k \mid h_k \in {\P} \}$ are bounded by a recursive function of the coefficients, the number of variables and the number of equations of the system. Since, as shown above, all of these parameters of system (\ref{2.67}) are bounded above by a computable function, the lemma follows.
\end{proof}

\section{The finite tree $T_0(\Omega )$ and minimal solutions} \label{se:5.3}

In Section \ref{se:5} we constructed an infinite tree $T(\Omega)$. Though for every solution $H$ the path in $T(\Omega)$
$$
(\Omega,H)=(\Omega_{v_0},H^{(0)})\to(\Omega_{v_1},H^{(1)})\to\dots \to (\Omega_{v_t},H^{(t)})
$$
defined by the solution $H$ is finite, in general, there is no global bound for the length of these paths.

Informally, the aim of this section is to prove that the set of solutions of the generalised equation $\Omega$ can be parametrised by a finite subtree $T_0$ (to be constructed below) of the tree $T$,  automorphisms from a finitely generated group of automorphisms of the coordinate group $G_{R(\Omega^*)}$, and solutions of the generalised equations associated to leaves of $T_0$. In other words, we prove that there exists a global bound  $M$, such that for any solution $H$ of $\Omega$ one can effectively construct a path \index{path!$\p(H)$}$\p(H)=(\Omega_{v_0},H^{[0]})\to(\Omega_{v_1},H^{[1]})\to\dots$ in $T(\Omega)$ of length bounded above by $M$ and such that $H^{[i]}<_{\Aut(\Omega)} H$ for all $i$, where ${\Aut(\Omega)}$  is the group of automorphisms of $G_{R(\Omega^*)}$ defined in Section \ref{5.5.3}, see  Definition \ref{defn:Aut} (note the abuse of notation: $H^{[i]}$ and $H$ are solutions of different generalised equations; in fact this relation is between two solutions of $\Omega$: a solution induced by $H^{[i]}$ and $H$; see below for a formal definition).

We summarise the results of this section in the proposition below.

\begin{prop}\label{prop:sum5}
For a generalised equation $\Omega=\Omega_{v_0}$, one can effectively construct a \emph{finite} oriented rooted at $v_0$ tree $T_0$, $T_0=T_0(\Omega_{v_0})$, such that:
\begin{enumerate}
\item The tree $T_0$ is a subtree of the tree $T(\Omega)$.
\item To the root $v_0$ of $T_0$ we assign a finitely generated group of automorphisms $\Aut(\Omega)$ {\rm (}see {\rm Definition \ref{defn:Aut})}.
\item For any solution $H$ of a generalised equation $\Omega $ there exists a leaf $w$ of the tree $T_0(\Omega )$, $\tp(w)=1,2$, and a solution $H^{[w]}$ of the generalised equation $\Omega _w$ such that
\begin{itemize}
    \item  ${H}^{[w]}<_{\Aut(\Omega)} H$;
    \item  if $\tp(w)=2$ and the generalised equation $\Omega_{w}$ contains non-constant non-active sections, then there exists a period $P$ such that $H^{[w]}$ is periodic with respect to the period $P$ and the generalised equation $\Omega _w$ is singular with respect to the periodic structure ${\P}(H^{[w]},P)$.
\end{itemize}
\end{enumerate}
\end{prop}

This section is organised in three subsections.  The aim of Section \ref{5.5.2} is to define the finitely generated group of automorphisms $\VV(\Omega_v)$ of the coordinate group of the generalised equation $\Omega_v$ associated to $v$, $v\in T(\Omega)$.

In Section \ref{5.5.3}, we define a finite subtree $T_0(\Omega)$ of the tree $T(\Omega)$. In order to define $T_0(\Omega)$ we introduce the notions of prohibited paths of type 7-10, 12 and 15.  Using Lemma \ref{3.2}, we prove that any infinite branch of the tree $T(\Omega)$ contains a prohibited path. The tree $T_0(\Omega)$ is defined to be a subtree of $T(\Omega)$ that does not contain prohibited paths and, by construction, is finite. We then define a finitely generated group of  automorphisms $\Aut(\Omega)$ that we assign to the root vertex of the tree $T_0(\Omega)$. The group $\Aut(\Omega)$ is generated by conjugates of the groups $\VV(\Omega_v)$, $v\in T_0(\Omega)$, $\tp(v)\ne 1$.

Finally, in Section \ref{5.5.4} for any solution $H$ of $\Omega$ we construct the path $\p(H)$ from the root $v_0$ to $w$, prove that $w$ is a leaf of the tree $T_0(\Omega)$ and show that the leaf $w$ satisfies the properties required in Proposition \ref{prop:sum5}.

\subsection{Automorphisms}\label{5.5.2}

To every vertex $v$ of the tree $T(\Omega)$, we assign a finitely generated group $\VV(\Omega_v)$ of automorphisms of $G_{R(\Omega_v^\ast)}$. \glossary{name={$\VV(\Omega_v)$}, description={a finitely generated group of automorphisms of $G_{R(\Omega_v^\ast)}$ that we associate to a vertex $v$ of the tree $T(\Omega)$}, sort=V}

If $\tp (v)=2$, set $\VV(\Omega_v)$ to be the group generated by all the groups of automorphisms $\AA(\Omega_v)$ corresponding to \emph{regular} periodic structures on $\Omega_v$, see Definition \ref{defn:AA}.

Let $7\leq \tp (v)\leq 10$. Let $\widehat{\Omega}_v$ be obtained from $\D 3(\Omega_v)$ by removing all equations corresponding to constant bases and all bases from the kernel of $\Omega_v$. Let $\pi$ be the natural homomorphism from $G_{R(\widehat{\Omega}_v^*)}$ to $G_{R(\Omega_v^*)}$.

An automorphism $\varphi$ of the coordinate group $G_{R(\Omega_v^*)}$ is called \index{automorphism!invariant with respect to the kernel}\emph{invariant with respect to the kernel} if it is induced by an automorphism $\varphi'$ of the coordinate group $G_{R(\widehat{\Omega}_v^*)}$ identical on the kernel of $\Omega_v$, i.e. there exists a $G$-automorphism $\varphi'$ of the coordinate group $G_{R(\widehat{\Omega}_v^*)}$ so that $\varphi'(h_i)=h_i$ for every variable $h_i\in \Ker(\Omega_v)$ and the following diagram commutes
$$
\CD
G_{R(\widehat{\Omega}_v^*)}   @>{\pi}>> G_{R(\Omega_v^*)} \\
  @V{\varphi'}VV        @V \varphi VV  \\
  G_{R(\widehat{\Omega}_v^*)}  @>{\pi}>> G_{R(\Omega_v^*)}
\endCD
$$
In this case, we define $\VV(\Omega_v)$ to be the group of automorphisms of $G_{R(\Omega_v^\ast)}$ invariant with respect to the kernel.

\begin{lem}\label{lem:razker}
There exists a finitely presented subgroup $K$ of the group of automorphisms of the automorphism group of $G[h]$ such that every automorphism from $K$ induces an automorphism of $\factor{G[h]}{\ncl\langle\Omega_v\rangle}$, which, in turn, induces an automorphism of the coordinate group $G_{R({\Omega_v}^*)}$. The finitely generated group of all automorphisms $K'$ of $G_{R({\Omega_v}^*)}$ induced by automorphisms from $K$ contains all the automorphisms invariant respect to the kernel of $\Omega_v$.
\end{lem}
\begin{proof}
By Lemma \ref{7-10}, for any generalised equation $\Omega$ we have
$$
G_{R(\Omega^\ast)} \simeq G_{R({\overline {\Ker(\Omega)}}^\ast)} \ast F(Z).
$$
It follows that, in the above notation, $G_{R(\widehat{\Omega}^*)}\simeq G\ast F(Y\cup Z)$. Let $K$ be the group of $G$-automorphisms of $G_{R(\widehat{\Omega}^*)}$ that fix every element from $F(Y)$, i.e. $K$ is the automorphism group of the free group $F(Z)$. It follows that $K$ is finitely presented, see \cite{McCool}. Furthermore, every automorphism from $K$ induces an automorphism of $G_{R(\Omega_v^\ast)}$ and the following diagram commutes:
$$
\CD
G\ast F(Y\cup Z)  @>{\pi}>> G_{R({\overline {\Ker(\Omega_v)}}^\ast)} \ast F(Z) \\
  @V{\varphi'}VV        @V \varphi VV  \\
G\ast F(Y\cup Z) @>{\pi}>> G_{R({\overline {\Ker(\Omega_v)}}^\ast)} \ast F(Z)
\endCD
$$
It follows that the group of induced automorphisms $K'$ by $K$ of $G_{R(\Omega_v^\ast)}$ is finitely generated and contains all automorphisms of $G_{R(\Omega_v^\ast)}$ invariant with respect to the kernel of $\Omega_v$.
\end{proof}

Let $\tp (v)=15$. Apply derived transformation $\D 3$ and consider the generalised equation $\D 3(\Omega_v)=\widetilde{\Omega} _v$. Notice that, since every boundary in the active part of $\widetilde{\Omega} _v$ that touches a base and intersects another base $\mu$ is $\mu$-tied (i.e. assumptions of Case 14 do not hold), the function $\gamma$ is constant on closed sections of $\widetilde{\Omega}_v$, i.e. $\gamma(h_i)=\gamma(h_j)$ whenever $h_i$ and $h_j$ belong to the same closed section of $\widetilde{\Omega}_v$. Applying $\D 2$, we can assume that every item in the section $[1,j+1]$ is covered  exactly twice (i.e. $\gamma(h_i)=2$ for every $i=1,\dots,j$) and for all $k\ge j+1$ we have $\gamma(h_k)>2$. In this case we call the section $[1,j+1]$ \index{part of a generalised equation!quadratic}\emph{the quadratic part of $\Omega_v$}. We sometimes refer to the set of non-quadratic sections of the generalised equation $\Omega_v$ as to the \index{part of a generalised equation!non-quadratic}\emph{non-quadratic part} of $\Omega_v$.

Let $\tp (v)=12$. Then the \emph{quadratic part of $\Omega_v$} is the whole active part of $\Omega_v$.

A variable base $\mu$ of the generalised equation $\Omega$ is called a \index{base!quadratic}{\em quadratic base} if $\mu $ and its dual $\Delta(\mu)$ both belong to the quadratic part of $\Omega$. A base $\mu$ of the generalised equation $\Omega$ is called a \index{base!quadratic-coefficient}\emph{quadratic-coefficient base} if $\mu$ belongs to the quadratic part of $\Omega$ and its dual $\Delta(\mu)$ does not belong to the quadratic part of $\Omega$.

Let $\tp (v)=12$ or $\tp(v)=15$ and let $[1,j+1]$ be the quadratic part of $\Omega_v$. Denote by $\widehat{\Omega}_v$ the generalised equation obtained from $\Omega_v$ by removing all non-quadratic bases, all quadratic-coefficient bases and all equations corresponding to constant bases.  Let $\phi$ be the  natural homomorphism from $\factor{G[h]}{\ncl\langle\widehat{\Omega}^*_v\rangle}$ to $G_{R(\Omega_v^*)}$.

By Lemma 2.6 in \cite{Razborov3}, $\widehat{\Omega}_v^*$ is equivalent to a system of quadratic equations $Q_1(y^{(1)}), \dots, Q_k(y^{(k)})$, where the sets of variables $y^{(i)}$ and $y^{(j)}$ are disjoint if $i\ne j$. It follows that the group  $\factor{G[h]}{\ncl\langle\widehat{\Omega}_v\rangle}$ is isomorphic to $G\ast S_1\ast\dots\ast S_k$, where $S_i$, $i=1,\dots, k$ is the fundamental group of a surface.

An automorphism $\varphi$ of the coordinate group $G_{R(\Omega_v^*)}$ is called \index{automorphism!invariant with respect to the non-quadratic part} \emph{invariant with respect to the non-quadratic part}, if there exists an automorphism $\hat{\varphi}$ of the group $\factor{G[h]}{\ncl(\widehat{\Omega}^*_v)}\simeq G\ast S_1\ast\dots\ast S_k$ so that: $\hat{\varphi}=\id\ast \varphi_1\ast \dots \ast \varphi_k$, where $\varphi_i$ is an automorphism of $S_i$; for every quadratic-coefficient base $\nu$ of $\Omega_v$ one has $\hat{\varphi}(h(\nu))=h(\nu)$; for every $h_i$, $i=j+1,\dots, \rho_{\Omega_v}$ (recall that $[1,j+1]$ is the quadratic part of $\Omega_v$) one has $\hat{\varphi}(h_i)=h_i$, and the following diagram commutes:
$$
\CD
  \factor{G[h]}{\ncl\langle\widehat{\Omega}^*_v\rangle} @>\phi>> G_{R(\Omega_v^*)} \\
  @V \hat{\varphi} VV                               @VV\varphi V  \\
  \factor{G[h]}{\ncl\langle\widehat{\Omega}^*_v\rangle} @>>\phi> G_{R(\Omega_v^*)}
\endCD
$$
In this case, we define $\VV(\Omega_v)$ to be the group of automorphisms of $G_{R(\Omega_v^\ast)}$ invariant with respect to the non-quadratic part.

\begin{lem}\label{lem:raznqp}
There exists a finitely presented subgroup $K$ of the group of automorphisms of $G[h]$ such that every automorphism from $K$ induces an automorphism of $\factor{G[h]}{\ncl\langle\Omega_v\rangle}$, which, in turn, induces an automorphism of the coordinate group $G_{R(\Omega_v^*)}$. The finitely generated group $K'$ of all automorphisms of $G_{R(\Omega_v^*)}$ induced by automorphisms from $K$ is the group of all automorphisms invariant with respect to the non-quadratic part of $\Omega_v$.
\end{lem}
\begin{proof}
Since the group of automorphisms of the fundamental group of a surface is isomorphic to the stabiliser of a cyclic quadratic word in the free group, by a result of J.~McCool, see \cite{McCool}, it follows that the group of automorphisms $K$ of $G[h]$ so that every automorphism from $K$ induces an automorphism from $\factor{G[h]}{\ncl(\widehat{\Omega}^*_v)}$, so that $\hat{\varphi}(h(\nu))=h(\nu)$ for every quadratic-coefficient base $\nu$ of $\Omega_v$ and $\hat{\varphi}(h_i)=h_i$ for every $h_i$, $i=j+1,\dots, \rho_{\Omega_v}$ is finitely presented. Since, by the definition $\hat{\varphi}(h(\nu))=h(\nu)$ for every quadratic-coefficient base $\nu$ of $\Omega_v$ and $\hat{\varphi}(h_i)=h_i$ for every $h_i$, $i=j+1,\dots, \rho_{\Omega_v}$, every automorphism from $K$ induces an automorphism of $G_{R({\Omega_v}^*)}$ and the diagram above is commutative. Hence, the group of all automorphisms invariant with respect to the non-quadratic part of $\Omega_v$ is finitely generated.
\end{proof}

In all other cases set $\VV(\Omega_v)=1$.

\subsection{The finite subtree $T_0(\Omega )$}\label{5.5.3}

The aim of this section is to construct the finite subtree \glossary{name={$T_0(\Omega)$}, description={the finite subtree of $T(\Omega)$ that does not contain prohibited paths}, sort=T}$T_0(\Omega)$ of $T(\Omega)$ as the subtree that does not contain prohibited paths.

The definition of a prohibited path is designed in such a way that the paths $\p(H)$ in the tree $T$ associated to the solution $H$ do not contain them. Therefore, the nature of the definition of a prohibited path will become clearer in Section \ref{5.5.4}.

\bigskip

By Lemma \ref{3.2}, infinite branches of the tree $T(\Omega)$ correspond to the following cases: $7\leq \tp(v_k)\leq 10$ for all $k$, or $\tp(v_k)=12$ for all $k$, or $\tp(v_k)=15$ for all $k$. We now define prohibited paths of types 7-10, 12 and 15.

\begin{defn}
We call a path $v_1\to v_2 \to\ldots\to v_k$ in $T(\Omega )$ \index{path!prohibited of type 7-10}{\em prohibited of type 7-10} if $7\leq \tp(v_i)\leq 10$  for all $i=1,\dots,k$ and some generalised equation with $\rho$ variables occurs among $\{\Omega_{v_i}\mid 1\le i\le l\}$ at least $4^{\rho}+1$ times.

Similarly, a path $v_1\to v_2 \to\ldots\to v_k$ in $T(\Omega )$ is called \index{path!prohibited of type 12}{\em prohibited of type 12} if $\tp(v_i)=12$ for all $i=1,\dots,k$  and some generalised equation with $\rho$ variables occurs among $\{\Omega_{v_i}\mid 1\le i\le l\}$ at least $4^{\rho}+1$ times.
\end{defn}

\bigskip

We now prove that an infinite branch of $T(\Omega)$  of type 7-10 or 12 contains a prohibited path of type 7-10 or 12, correspondingly.

\begin{lem} \label{3.3}
Let $v_0\to v_1 \to \ldots\to v_n\to\ldots $ be an infinite path in the tree $T(\Omega)$, where $7\leq \tp(v_i) \leq 10$ for all $i$, and let $\Omega_{v_0}, \Omega_{v_1}, \dots, \Omega_{v_n}, \dots$ be the sequence of corresponding generalised equations. Then among $\{\Omega_{v_i}\}$ some generalised equation occurs infinitely many times. Furthermore, if $\Omega_{v_k}=\Omega _{v_l}$, then $\pi(v_k,v_l)$ is a $G$-automorphism of $G_{R(\Omega_{v_k}^*)}$ invariant with respect to the kernel of $\Omega_{v_k}$.
\end{lem}
\begin{proof}
By Lemma \ref{3.1}, we have that $\comp(\Omega_{v_k})\leq \comp (\Omega_{v_0})$ and $\xi(\Omega_{v_k})\leq \xi(\Omega_{v_0})$ for all $k$. We, therefore, may assume that $\comp=\comp(\Omega_{v_k})=\comp(\Omega_{v_0})$ and $\xi(\Omega_{v_k})=\xi(\Omega_{v_0})$ for all $k$. It follows that all the transformations $\ET 5$ introduce a new boundary.

For all $k$, the generalised equations $\Ker (\widetilde{\Omega}_{v_k})$ have the same set of bases, recall that $\widetilde{\Omega}_{v_k}=\D 3(\widetilde{\Omega}_{v_k})$. Indeed, consider the generalised equations $\widetilde{\Omega} _{v_k}$ and $\widetilde{\Omega}_{v_{k+1}}$. Since $\tp(v_k)\ne 3,4$, the active part of $\widetilde{\Omega}_{v_k}$ does not contain constant bases.

If $\tp(v_k)=7,8,10$, then $\widetilde{\Omega}_{v_{k+1}}$ is obtained from $\widetilde{\Omega} _{v_k}$ by cutting some  base $\mu$ which is eliminable in $\widetilde{\Omega} _{v_k}$  and then deleting one of the new bases, which is also eliminable, since it falls under the assumption a) of the definition of an eliminable base. Since every transformation $\ET 5$ introduces a new boundary, the remaining part of the base $\mu$ falls under the assumption b) of the definition of an eliminable base. Therefore, in this case, the set of bases that belong to the kernel does not change.

Let $\tp(v_k)=9$. It suffices to show that, in the notation of Case 9, all the bases of $\widetilde{\Omega}_{v_{k+1}}$ obtained by cutting the base $\mu_2$ do not belong to the kernel. Without loss of generality we may assume that $\sigma(\mu_2)$ is a closed section of $\widetilde{\Omega}_{v_k}$. Indeed, if $\sigma(\mu_2)$ is not closed, instead, we can consider one of its closed subsections $\sigma'$ in the generalised equation $\widetilde{\Omega}_{v_k}$.

Notice that, since $\sigma(\mu_2)$ is closed, every boundary that intersects $\mu_1$ and $\mu_2$ in $\Omega_{v_k}$, touches exactly two bases in $\widetilde{\Omega}_{v_k}$. Thus, for every boundary connection $(p,\mu_2,q)$ in $\Omega_{v_{k+1}}$ either the boundary $p$ or the boundary $q$ touches exactly two bases in $\widetilde{\Omega}_{v_{k+1}}$.

Construct an elimination process (see description of the derived transformation $\D 4$) for the generalised equation $\widetilde{\Omega}_{v_k}$ and take the first generalised equation $\Omega_i$ in this elimination process, such that the base $\nu$ eliminated in this equation was obtained from either $\mu_1$ or $\mu_2$ or $\Delta(\mu _1)$ or $\Delta(\mu _2)$ by applying $\D 3$ to ${\Omega}_{v_k}$. The base $\nu$ could not be obtained from $\mu _1$ or $\mu _2$, since every item in the section $\sigma(\mu_2)$ is covered twice and every boundary in this section touches two bases.

If $\nu$ falls under the assumption of case b) of the definition of an eliminable base, then
$$
\hbox{either }\alpha (\nu)\in\{\alpha(\Delta (\mu _1)), \alpha (\Delta (\mu _2))\}\hbox{ or }\beta (\nu)\in\{\beta (\Delta (\mu _1)), \beta(\Delta (\mu _2))\}.
$$
We now construct an elimination process for the generalised equation $\widetilde{\Omega}_{v_{k+1}}$. The first $i$ steps of the elimination process for $\widetilde{\Omega}_{v_{k+1}}$ coincide with the first $i$ steps of the elimination process constructed for the generalised equation $\widetilde{\Omega}_{v_k}$. Then the eliminable base $\nu$ of $\Omega_{i}$ corresponds to a base $\nu '$ obtained from either $\mu_2$ or $\Delta(\mu _2)$ by applying $\D 3$ to ${\Omega}_{v_{k+1}}$. The base $\nu'$ is eliminable.

Notice that one of the boundaries $\alpha(\nu)$, $\beta(\nu)$, $\alpha(\Delta(\nu))$ or $\alpha(\Delta(\nu))$ touches just two bases. Therefore, after eliminating $\nu'$, this boundary touches just one base $\eta$ that was obtained from $\mu_2$ or $\Delta(\mu_2)$. The base $\eta$ falls under the assumptions b) of the definition of an eliminable base. Repeating this argument, one can subsequently eliminate all the other bases obtained from $\mu _2$ or $\Delta(\mu_2)$. It follows that all the bases of the generalised equation $\widetilde{\Omega} _{v_{k+1}}$, obtained from $\mu _2$ or $\Delta(\mu_2)$ do not belong to the kernel.

We thereby have shown that the set of bases is the same for all the generalised equations $\Ker(\widetilde{\Omega} _{v_k})$ and thus the set of bases is the same for all the generalised equations $\Ker(\widetilde{\Omega} _{v_k})$. We denote the cardinality of this set by $n'$.

We now prove that the number of bases in the active sections of $\Omega_{v_{k}}$, for all $k$, is bounded above by a function of $\Omega_{v_0}$:
\begin{equation}\label{3.2'}
n_A(\Omega_{v_{k}})\leq 3\comp+6n'+1.
\end{equation}
Indeed, assume the contrary and let $k$ be minimal for which inequality (\ref{3.2'}) fails. Then
\begin{equation}\label{3.3'}
n_A(\Omega_{v_{k-1}})\leq 3\comp +6n'+1, n_A(\Omega_{v_{k}})>3\comp+6n'+1.
\end{equation}
By Lemma \ref{3.1}, $\tp(v_{k-1})=10$. In particular the generalised equation $\Omega_{v_{k-1}}$ fails the assumptions of Cases $5,6,7,8,9$. Therefore, every active section of $\Omega _{v_{k-1}}$ either contains at least three bases or contains some base of the generalised equation $\Ker (\widetilde{\Omega} _{v_{k-1}})$. Let $u_{k-1}$ and $w_{k-1}$ be the number of active sections of $\Omega_{v_{k-1}}$ that contain one base and more than one base, respectively. Hence $u_{k-1}+w_{k-1}\leq\frac{1}{3}n_A(\Omega_{v_{k-1}})+n'$. It is easy to see that
$$
\comp = \sum_{\sigma \in A\Sigma(\Omega)} \max\{0, n(\sigma)-2\}= n_A(\Omega_{v_{k-1}})-2w_{k-1}-u_{k-1}.
$$
Then, $\comp\ge n_A(\Omega_{v_{k-1}})-2(w_{k-1}+u_{k-1})\geq\frac{1}{3}n_A(\Omega_{v_{k-1}})-2n'$, which contradicts  (\ref{3.3'}).

Furthermore, the number $\rho_A(\Omega_{v_k})$ of items in the active part of  $\Omega_{v_k}$ is bounded above:
$$
\rho_A(\Omega_{v_k})\leq \xi(\Omega_{v_k})+n_A(\Omega_{v_{k-1}})+1\leq 3\comp +6n'+\xi(\Omega_{v_0}) +2.
$$
Since the number of bases and number of items is bounded above, the set $\{\Omega _{v_k}\mid k\in\mathbb{N}\}$ is finite and thus some generalised equation occurs in this set infinitely many times.

Let $\Omega _{v_k}=\Omega _{v_l}$. Since, by assumption, the edges $v_j\to v_{j+1}$, $j=k,\dots, l-1$ are labelled by isomorphisms, the homomorphism $\pi (v_k,v_l)$ is an automorphism of the coordinate group $G_{R(\Omega_{v_k}^*)}$.

We are left to show that $\pi (v_k,v_l)$ is invariant with respect to the kernel. As shown above,
$$
\Ker (\widetilde{\Omega}_{v_{k}})=\Ker (\widetilde{\Omega}_{v_{k+1}})=\dots=\Ker (\widetilde{\Omega} _{v_{l}}).
$$
From the above, it follows that $\Ker(\widetilde{\Omega}_{v_{i+1}})$ is obtained from  $\Ker(\widetilde{\Omega}_{v_{i}})$ by introducing new boundaries and removing some of the items that do not belong to the kernel of $\widetilde{\Omega}_{v_{i+1}}$. Therefore, the number of items that belong to the kernel $\Ker(\widetilde{\Omega}_{v_{i+1}})$ can only increase. As $\Omega_{v_k}=\Omega_{v_l}$, so this number is the same for all $i$, $i=k,\dots, l-1$. It follows that $\pi (v_k,v_l)'(h_i)=h_i$ for all $h_i$ that belong to the kernel of $\Omega_{v_k}$.

Since the transformations that take the generalised equation $\widetilde{\Omega}_{v_{k}}$ to $\widetilde{\Omega}_{v_{k+1}}$ do not involve bases that belong to the kernel of $\widetilde{\Omega}_{v_{k}}$, the same sequence of transformations can be applied to the generalised equation $\widehat{\widetilde{\Omega}}_{v_{k}}$, where $\widehat{\widetilde{\Omega}}_{v_{k}}$ is obtained from $\widetilde{\Omega}_{v_{k}}$ by removing all equations corresponding to constant bases and all bases that belong to the kernel of $\widetilde{\Omega}_{v_k}$.

Since, by assumption, every time we $\mu$-tie a boundary a new boundary is introduced, we get that the epimorphism from $G_{R(\widehat{\widetilde{\Omega}}_{v_{k}}^*)}$ to $G_{R(\widehat{\widetilde{\Omega}}_{v_{k+1}}^*)}$ is, in fact, an isomorphism. We therefore get the following commutative diagram (see the definition of an automorphism invariant with respect to the kernel, Section \ref{5.5.2}):
$$
\CD
G_{R(\widehat{\Omega}_{v_k}^*)}   @>{\pi}>> G_{R({\Omega_{v_k}}^*)} \\
  @VVV        @VV\pi (v_k,v_l)V  \\
  G_{R({\widehat{\Omega}_{v_{l}}}^*)}  @>{\pi}>> G_{R({\Omega_{v_{l}}}^*)}
\endCD
$$
It follows that the automorphism $\pi (v_k,v_l)$ is invariant with respect to the kernel of ${\Omega}_{v_{k}}$.
\end{proof}

\begin{cor} \label{cor:proh710}
Let $\p=v_1 \to \ldots\to v_n\to\ldots $ be an infinite path in the tree $T(\Omega)$, and $7\le \tp(v_i)\le 10$ for all $i$. Then $\p$ contains a prohibited subpath of type 7-10.
\end{cor}

\begin{lem}  \label{lem:c12}
Let $v_0\to v_1 \to \ldots\to v_n\to\ldots $ be an infinite path in the tree $T(\Omega)$, where $\tp(v_i) = 12$ for all $i$, and $\Omega_{v_0}, \Omega_{v_1}, \dots, \Omega_{v_n}, \dots$ be the sequence of corresponding generalised equations. Then among $\{\Omega_{v_i}\}$ some generalised equation occurs infinitely many times. Furthermore, if $\Omega_{v_k}=\Omega _{v_l}$, then $\pi(v_k,v_l)$ is a $G$-automorphism of the coordinate group $G_{R(\Omega_{v_k}^*)}$ invariant with respect to the non-quadratic part.
\end{lem}
\begin{proof}
Notice that since $\Omega _{v_i}$ is a quadratic generalised equation, quadratic-coefficient bases of $\Omega _{v_i}$ are bases whose duals belong to the non-active part.

Let $\mu_{i}$ be the carrier base of the generalised equation $\Omega _{v_i}$. Consider the sequence $\mu_0,\dots,\mu_i,\dots$ of carrier bases.
By Lemma \ref{3.1}, if $\tp({v_{i}})=12$, then $n_A(\Omega_{v_{i+1}}) \le n_A(\Omega_{v_{i}})$. Furthermore, if the carrier base $\mu_i$ is quadratic-coefficient, then this inequality is strict. Hence, it suffices to consider the case when all carrier bases are quadratic.

The number of consecutive quadratic bases in the sequence $\mu_1,\dots,\mu_i,\ldots$ is bounded above. Indeed, by Lemma \ref{3.1}, when the entire transformation is applied, the complexity of the generalised equation does not increase. Furthermore, since the generalised equation is quadratic and does not contain free boundaries, the number of items does not increase. The number of generalised equations with a bounded number of items and bounded complexity is finite and thus some generalised equation occurs in the sequence $\{\Omega_{v_i}\}$ infinitely many times.

Obviously, if $\Omega _{v_k}=\Omega _{v_l}$, then $\pi(v_k,v_l)$ is a $G$-automorphism of $G_{R(\Omega_{v_k}^*)}$. We are left to show that $\pi (v_k,v_l)$ is invariant with respect to the non-quadratic part.

From the definition of the entire transformation $\D 5$, it follows that  the number of items that belong to a given quadratic-coefficient base can only increase. As $\Omega_{v_k}=\Omega_{v_l}$, so this number is the same for all $i$, $i=k,\dots, l$. It follows that $\pi (v_k,v_l)'(h_i)=h_i$ for all $h_i$ that belong to a quadratic-coefficient base.

Since the transformations that take the generalised equation $\widetilde{\Omega}_{v_{k}}$ to $\widetilde{\Omega}_{v_{k+1}}$ involve only quadratic bases of $\widetilde{\Omega}_{v_{k}}$, the same sequence of transformations can be applied to the generalised equation $\widehat{\widetilde{\Omega}}_{v_{k}}$, where $\widehat{\widetilde{\Omega}}_{v_{k}}$ is the generalised equation obtained from $\widetilde{\Omega}_{v_k}$ by removing all non-quadratic bases, all quadratic-coefficient bases and all coefficient equations.

Since, by assumption, every time we $\mu$-tie a boundary a new boundary is introduced, we get that the epimorphism from $G_{R(\widehat{\widetilde{\Omega}}_{v_{k}}^*)}$ to $G_{R(\widehat{\widetilde{\Omega}}_{v_{k+1}}^*)}$ is, in fact, an isomorphism. We therefore get the following commutative diagram (see the definition of an automorphism invariant with respect to the non-quadratic part, Section \ref{5.5.2}):
$$
\CD
\factor{G[h]}{\ncl\langle\widehat{\Omega}_{v_k}\rangle}   @>{\phi}>> G_{R({\Omega_{v_k}}^*)} \\
  @VVV        @VV\pi (v_k,v_l)'V  \\
  \factor{G[h]}{\ncl\langle\widehat{\Omega}_{v_l}\rangle}  @>{\phi}>> G_{R({\Omega_{v_{l}}}^*)}
\endCD
$$
It follows that the automorphism $\pi (v_k,v_l)$ is invariant with respect to the non-quadratic part of $\Omega_{v_k}$.
\end{proof}

\begin{cor}\label{cor:proh12}
Let $\p=v_1 \to \ldots\to v_n\to\ldots $ be an infinite path in the tree $T(\Omega)$, and $\tp(v_i)=12$ for all $i$. Then $\p$ contains a prohibited subpath of type 12.
\end{cor}

\bigskip

Below we shall define  prohibited paths of type 15 in $T(\Omega )$. In this case, the definition of a prohibited path is much more involved.

We need some auxiliary definitions. Recall that the complexity of a generalised equation $\Omega $ is defined as follows:
$$
\comp = \comp (\Omega) = \sum_{\sigma \in A\Sigma_\Omega} \max\{0, n(\sigma)-2\},
$$
where $n(\sigma)$ is the number of bases in $\sigma$. Let \glossary{name={$\tau_v$, $\tau (\Omega_v)$}, description={function of the generalised equation}, sort=T}$\tau_v=\tau (\Omega_v)=\comp(\Omega_v)+\rho -\rho_{v}'$, where $\rho=\rho_\Omega$ is the number of variables in the initial generalised equation $\Omega $ and \glossary{name={$\rho_v'$}, description={the number of free variables that belong to the non-active sections of  $\Omega _v$}, sort=R}$\rho_{v}'$ is the number of free variables belonging to the non-active sections of the generalised equation $\Omega _v$. We have $\rho _{v}'\le \rho$ (see the proof of Lemma \ref{3.2}), hence $\tau_v\geq 0$. If, in addition,  $v_1\to v_2$ is an auxiliary edge, then $\tau_{v_2}< \tau_{v_1}$.

We use induction on $\tau _v$ to construct a finite subtree $T_0(\Omega _v)$ of $T(\Omega _v)$, and the function $\ss(\Omega_v)$.

The tree $T_0(\Omega _v)$ is a rooted tree at $v$ and consists of some of the vertices and edges of $T(\Omega)$ that lie above $v$.

Suppose that $\tau _v=0$. It follows that $\comp(\Omega_v)=0$. Then in $T(\Omega)$ there are no auxiliary edges. Furthermore, since $\comp(\Omega_v)=0$, it follows that every closed active section contains at most two bases and so no vertex of type 15 lies above $v$. We define the subtree $T_0(\Omega _v)$ as follows. The set of vertices of $T_0(\Omega _v)$ consists of all vertices $v_1$ of $T(\Omega)$ that lie above $v$, and so that the path from $v$ to $v_1$ does not contain prohibited subpaths of types 7-10 and 12. By Corollary \ref{cor:proh710} and Corollary \ref{cor:proh12},  $T_0(\Omega _v)$ is finite.

Let
\glossary{name={$\ss(\Omega_v)$}, description={a function that bounds the length of a minimal solution}, sort=S}
\begin{equation}\label{so}
\ss(\Omega_v)=\max\limits_w \max\limits_{\langle {\P},R\rangle}\rho_{\Omega_w}\cdot\{ f_{0} (\Omega_w,{\P},R)\},
\end{equation}
where $\max\limits_w$ is taken over all the vertices of $T_0(\Omega_v)$ for which $\tp (w)=2$ and $\Omega _w$ contains
non-active sections; $\max\limits_{\langle {\P},R\rangle}$ is taken over all regular periodic structures such that the generalised equation  $\widetilde{\Omega}_w$ is regular with respect to $\langle \P, R\rangle$; and $f_{0}$ is the function defined in Lemma \ref{lem:23-2}.

Suppose now that $\tau_v> 0$. By induction, we assume that for all $v_1$ such that $\tau _{v_1}< \tau _v$ the finite tree $T_0(\Omega _{v_1})$ and $\ss(\Omega _{v_1})$ are already defined. Furthermore, we assume that the full subtree of $T(\Omega)$ whose set of vertices consists of all vertices that lie above $v$ does not contain prohibited paths of type 7-10 and of type 12. Consider a path $\p$ in $T(\Omega)$:
\begin{equation}\label{3.6}
v_1\rightarrow v_2\rightarrow \ldots\rightarrow v_m,
\end{equation}
where $\tp(v_i)=15$, $1\leq i\leq m$ and all the edges are principal. We have $\tau _{v_i}=\tau _v$.

Denote by $\mu _i$ the carrier base of the generalised equation $\Omega_{v_i}$. Path (\ref{3.6}) is called \index{path!$\mu$-reducing}\emph{$\mu$-reducing} if $\mu_1=\mu$ and either there are no auxiliary edges from the vertex $v_2$ and $\mu$ occurs in the sequence $\mu _1,\ldots ,\mu_{m-1}$ at least twice, or there are auxiliary edges $v_2\to w_1$, $v_2\to w_2, \ldots ,v_2\to w_\nn$ and $\mu$ occurs in the sequence $\mu _1,\ldots ,\mu _{m-1}$ at least $\max \limits_{1\leq i\leq \nn} \ss(\Omega _{w_i})$ times. We will show later, see Equation (\ref{3.26}), that, informally, in any $\mu$-reducing path the length of the solution $H$ is reduced by at least $\frac{1}{10}$ of the length of $H(\mu)$, hence the terminology.

\begin{defn}\label{defn:proh15}
Path (\ref{3.6}) is called \index{path!prohibited of type 15}{\em prohibited of type 15}, if it can be represented in the form
\begin{equation}\label{3.7}
\p_1\s_1\ldots \p_l\s_l\p',
\end{equation}
where for some sequence of bases $\eta _1,\ldots ,\eta _l$ the following three conditions are satisfied:
\begin{enumerate}
    \item\label{it:prp2} the path $\p_i$ is $\eta _i$-reducing;
    \item\label{it:prp1} every base $\mu_i$ that occurs at least once in the sequence $\mu_1,\dots, \mu_{m-1}$, occurs at least $40n^2f_{1}(\Omega_{v_2})+20n+1$ times in the sequence $\eta _1,\ldots ,\eta _l$, where $n=|\BS(\Omega_{v_i})|$ is the number of all bases in the generalised equation $\Omega_{v_i}$, and $f_{1}$ is the function from Lemma \ref{2.8}; in other words, in a prohibited path of type 15, for every carrier base $\mu_i$ there exists at least $40n^2f_{1}(\Omega_{v_2})+20n+1$ many $\mu_i$-reducing paths.
    \item\label{it:prp3} every transfer base of some generalised equation of the path $\p$ is a transfer base of some generalised equation of the path $\p'$.
\end{enumerate}
\end{defn}
Note that for any path of the form (\ref{3.6}) in $T(\Omega)$, there is an algorithm to decide whether this path  is prohibited of type 15 or not.

\bigskip

We now prove that any infinite branch of the tree $T(\Omega)$ of type 15 contains a prohibited subpath of type 15.
\begin{lem}
Let $\p=v_1 \to \ldots\to v_n\to\ldots $ be an infinite path in the tree $T(\Omega)$, and $\tp(v_i) =15$ for all $i$. Then
$\p$ contains a prohibited subpath of type 15.
\end{lem}
\begin{proof}
Let $\omega$ be the set of all bases occurring in the sequence of carrier bases $\mu _1,\mu_2,\ldots$ infinitely many times, and $\tilde\omega$ be the set of all bases that are transfer bases of infinitely many equations $\Omega _{v_i}$. Considering, if necessary, a subpath $\tilde \p$ of $\p$ of the form $v_j \to v_{j+1}\to \ldots$, one can assume that all the bases in the sequence $\mu _1,\mu_2 \ldots $ belong to $\omega$ and every base which is a transfer base of at least one generalised equation belongs to $\tilde\omega$. Then for any $\mu\in\omega$ the path $\tilde \p$ contains infinitely many non-intersecting $\mu$-reducing finite subpaths. Hence  there exists a subpath of the form (\ref{3.7}) of $\tilde \p$  which satisfies conditions (\ref{it:prp2}) and (\ref{it:prp1}) of the definition of a prohibited path of type 15, see Definition \ref{defn:proh15}. Taking a long enough subpath $\p'$ of $\p$, we obtain a prohibited subpath of $\p$.
\end{proof}

We now construct the tree \glossary{name={$T_0(\Omega)$}, description={the finite subtree of $T(\Omega)$ that does not contain prohibited paths}, sort=T}$T_0(\Omega)$. Let $T'(\Omega_v)$ be the subtree of $T(\Omega _v)$ consisting of the vertices $v_1$ such that the path from $v$ to $v_1$ in $T(\Omega)$ does not contain prohibited subpaths and does not contain vertices $v_2\ne v_1$ such that $\tau_{v_2}< \tau _v$. Thus, the leaves of $T'(\Omega_v)$ are either vertices $v_1$ such that $\tau_{v_1}< \tau_v$ or leaves of $T(\Omega _v)$.

The subtree $T'(\Omega_v)$ can be effectively constructed. The tree $T_0(\Omega _v)$ is obtained from $T'(\Omega_v)$ by attaching (gluing) $T_0(\Omega _{v_1})$ (which is already constructed by the induction hypothesis) to those leaves $v_1$ of $T'(\Omega _v)$ for which $\tau _{v_1}< \tau _v$. The function $\ss(\Omega _v)$ is defined by (\ref{so}). Set $T_0(\Omega)=T_0(\Omega _{v_0})$, which is finite by construction.

\bigskip

\begin{defn}\label{defn:Aut}
Denote by \glossary{name={$\Aut (\Omega)$}, description={finitely generated group of automorphisms associated to the root of $T_0(\Omega)$}, sort=A}$\Aut (\Omega)$, $\Omega=\Omega_{v_0}$, the group of  automorphisms of $G_{R(\Omega^\ast)}$, generated by all the groups $\pi(v_0,v)\VV(\Omega _v)\pi (v_0,v)^{-1}$, $v\in T_0(\Omega)$, $\tp(v)\ne 1$ (thus $\pi (v_0,v)$ is an isomorphism). Note that by construction the group $\Aut (\Omega )$ is finitely generated.
\end{defn}

We adopt the following convention. Given two solutions $H^{(i)}$ of $\Omega_{v_i}$ and ${H^{(i')}}$ of $\Omega_{v_i'}$, by \glossary{name={`$H^{(i)}<_{\Aut(\Omega)}{H^{(i')}}$'}, description={if $H^{(i)}$ is a solution of $\Omega_{v_i}$ and ${H^{(i')}}$ is a solution of $\Omega_{v_{i'}}$, $v_i, v_{i'}\in T_0(\Omega)$, then $H^{(i)}<_{\Aut(\Omega)}{H^{(i')}}$ if and only if $H^{(i)}<_{\pi (v_0,v_{i'})^{-1}\Aut(\Omega)\pi(v_0,v_{i})} {H^{(i')}}$}, sort=Z}
$H^{(i)}<_{\Aut(\Omega)}{H^{(i')}}$ we mean $H^{(i)}<_{\pi (v_0,v_{i'})^{-1}\Aut(\Omega)\pi(v_0,v_{i})} {H^{(i')}}$.

\begin{lem} \label{cor:2.1}
Let
$$
(\Omega,H)=(\Omega _{v_0},H^{(0)})\to (\Omega _{v_1}, H^{(1)})\to\ldots\to (\Omega _{v_l}, H^{(l)})
$$
be the path defined by the solution $H$. If $H$ is a minimal solution with respect to the group of automorphisms $\Aut(\Omega)$, then $H^{(i)}$ is a minimal solution of $\Omega _{v_i}$ with respect to the group $\VV(\Omega_{v_i})$ for all $i$.
\end{lem}
\begin{proof}
Follows from Lemma \ref{lem:2.1}
\end{proof}

\subsection{Paths $\p(H)$ are in $T_0(\Omega )$} \label{5.5.4}

The goal of this section is to give a proof of the proposition below.
\begin{prop} \label{3.4}
For any solution $H$ of a generalised equation $\Omega $ there exists a leaf $w$ of the tree $T_0(\Omega )$, $\tp(w)=1,2$, and a solution $H^{[w]}$ of the generalised equation $\Omega _w$ such that
\begin{enumerate}
    \item \label{it:prop1} ${H}^{[w]}<_{\Aut(\Omega)} H$;
    \item \label{it:prop2} if $\tp(w)=2$ and the generalised equation $\Omega_{w}$ contains non-constant non-active sections, then there exists a period $P$ such that $H^{[w]}$ is periodic with respect to the period $P$, and the generalised equation $\Omega _w$ is singular with respect to the periodic structure ${\P}(H^{[w]},P)$.
\end{enumerate}
\end{prop}

The proof of this proposition is rather long and technical. We now outline the organisation of the proof.

In part (A), for any solution $H$ of $\Omega$ we describe a path $\p(H):(\Omega_{v_0},H^{[0]})\to \dots \to (\Omega_{v_l}, H^{[l]})$, where $H^{[i]}<_{\Aut(\Omega)} H$.

In part (B) we prove that all vertices of the paths $\p(H)$ belong to the tree $T_0$. In order to do so we show that they do not contain prohibited subpaths. In steps (I) and (II) we prove that the paths $\p(H)$ do not contain prohibited paths of type 7-10 and 12, correspondingly. The proof is by contradiction: if $\p(H)$ is not in $T_0$, the fact that a generalised equation repeats enough times, allows us to construct an automorphism that makes the solution shorter, contradicting its minimality.

To prove that the paths $\p(H)$  do not contain prohibited paths of type 15 (step (III)) we show, by contradiction, that on one hand, the length of minimal solutions is bounded above by a function of the excess (see Definition \ref{defn:excess} for definition of excess), see Equation (\ref{3.18}) and, on the other hand the length of a minimal solution $H$ such that $\p(H)$ contains a prohibited path of type 15 fails inequality (\ref{3.18}).

Finally, in part (C) we prove that the pair $(\Omega_w,H^{[w]})$, where $w$ is a leaf of type 2, satisfies the properties required in Proposition \ref{3.4}.

\subsubsection*{\textbf{{\rm(A):} Constructing the paths $\p(H)$}}

To define the path \index{path!$\p(H)$}$\p(H)$ we shall make use of two functions \glossary{name={$\edge$}, description={a function that assigns a pair $(\Omega_{v'},H^{(v')})$ to the pair $(\Omega_v,H^{(v)})$, $\tp(v)\ne 1,2$}, sort=E}$\edge$ and \glossary{name={$\edge'$}, description={a function that assigns a pair $(\Omega_{v'},H^{(v')})$ to the pair $(\Omega_v,H^{(v)})$, $\tp(v)=15$}, sort=E}$\edge'$ that assign to the pair $(\Omega_v,H^{(v)})$ a pair $(\Omega_{v'},H^{(v')})$, where either $v'=v$ or there is an edge $v\to v'$ in $T(\Omega)$. The function $\edge$ can be applied to any pair $(\Omega_v,H^{(v)})$, where $\tp(v)\ne 1,2$ and the function $\edge'$ can only be applied to a pair $(\Omega_v,H^{(v)})$, where $\tp(v)=15$ and there are auxiliary edges outgoing from the vertex $v$.

\medskip

We now define the functions $\edge$ and $\edge'$.

Let $\tp(v)=3$ or $\tp (v)\geq 6$ and let $v\to w_1,\ldots,v\to w_m$ be the list of all \emph{principal} outgoing edges from $v$, then the generalised equations $\Omega _{w_1},\dots ,\Omega_{w_m}$ are obtained from $\Omega _v$ by a sequence of elementary transformations. For every solution $H$ the path defined by $H$ is unique, i.e. for the pair $(\Omega,H)$ there exists a unique pair $(\Omega_{w_i},H^{(i)})$ such that the following diagram commutes:
$$
\xymatrix@C3em{
 G_{R(\Omega^\ast)}  \ar[rd]_{\pi_H} \ar[rr]^{\theta_i}  &  &G_{R(\Omega_{w_i}^\ast )} \ar[ld]^{\pi_{H^{(i)}}}
                                                                             \\
                               &  G &
}
$$
Define a function $\edge$ that assigns the pair $(\Omega _{w_i},H ^{(i)})$ to the pair $(\Omega _v,H)$, $\edge(\Omega_v,H)=(\Omega _{w_i},H ^{(i)})$.

Let $\tp(v)=4$ or $\tp(v)=5$. In these cases there is a single edge $v\to w_1$ outgoing  from $v$ and this edge is auxiliary. We set $\edge(\Omega_v,H)=(\Omega_{w_1},H^{(1)})$.

If $\tp (v)=15$ and there are auxiliary outgoing edges from the vertex $v$, then the carrier base $\mu$ of the generalised equation $\Omega _v$ intersects with $\Delta (\mu)$. Below we use the notation from the description of Case 15.1. For any solution $H$ of the generalised equation $\Omega_v$ one can construct a solution $H'$ of the generalised equation $\Omega_{v'}$ as follows: $H'_{\rho_v+1}=H[1,\beta(\Delta(\mu))]$. We define the function $\edge'$ as follows  $\edge'(\Omega_v,H)=(\Omega_{v'},H')$.

\medskip

To construct the path $\p(H)$
\begin{equation}\label{3.8}
(\Omega,H)\to (\Omega _{v_0},H^{[0]})\to (\Omega _{v_1}, H^{[1]})\to\ldots
\end{equation}
we use induction on its length $i$.

Let $i=0$, we define $H^{[0]}$ to be a solution of the generalised equation $\Omega$ minimal with respect to the group of automorphisms $\Aut(\Omega)$, such that $H^{[0]} <_{\Aut(\Omega)} H$. Let $i \ge 1$ and suppose that the term $(\Omega_{v_i},H^{[i]})$ of the sequence (\ref{3.8}) is already constructed. We construct $(\Omega_{v_{i+1}},H^{[i+1]})$

If $3\le \tp(v_i)\le 6$, $\tp(v_i)=11,13,14$, we set $(\Omega _{v_{i+1}},H^{[i+1]})=\edge(\Omega _{v_i},{H^{[i]}})$.

If $7\leq \tp(v_i)\leq 10$ or $\tp(v_i)=12$ and there exists a minimal solution $H^+$ of $\Omega _{v_i}$ such that $H^+<_{\Aut(\Omega)} H^{[i]}$ and $|H^+|<|H^{[i]}|$, then we set $(\Omega_{v_{i+1}},H^{[i+1]})=(\Omega_{v_i}, H^+)$. Note that, since  $H^{[0]}$ is a minimal solution of $\Omega$ with respect to the group of automorphisms $\Aut(\Omega)$, by construction and by Lemma \ref{cor:2.1}, we have that the solution $H^{[i]}$ is minimal with respect to the group of automorphism $\VV(\Omega_{v_i})$ for all $i$. Although $H^{[i]}$ is a minimal solution, in this step we take a  minimal solution of minimal total length, see Remark \ref{rem:ms}.

Let $\tp(v_i)=15$, $v_i\neq v_{i-1}$ and $v_i\to w_1,\ldots ,v_i\to w_{\nn}$ be the auxiliary edges outgoing from $v_i$ (the carrier base $\mu$ intersects with its dual $\Delta(\mu)$). If there exists a period $P$ such that
\begin{equation}\label{3.9}
{H^{[i]}}[1,\beta (\Delta (\mu))]\doteq P^rP_1,\ P\doteq P_1P_2, \ r\ge \max \limits_{1\leq i\leq {\nn}} \ss(\Omega _{w_i}),
\end{equation}
then we set $(\Omega_{v_{i+1}},H^{[i+1]}) = \edge'(\Omega _{v_i},H^{[i]})$ and declare the section $[1,\beta (\Delta (\mu))]$ non-active.

In all the other cases (when $\tp(v_i)=15$) we set $(\Omega _{v_{i+1}},H^{[i+1]})=\edge(\Omega _{v_i},{H^{[i]}})$.

The path (\ref{3.8}) ends if $\tp (v_i)\leq 2$.

\medskip

A leaf $w$ of the tree $T(\Omega)$ is called \index{leaf!final of the tree $T$}\emph{final} if there exists a solution $H$ of $\Omega_{v_0}$ and a path $\p(H)$  such that $\p(H)$ ends in $w$.

\subsubsection*{\textbf{{\rm (B):} Paths $\p(H)$ belong to $T_0$}}

We use induction on $\tau$ to show that every vertex $v_i$ of the path $\p(H)$ (see Equation (\ref{3.8})) belongs to $T_0(\Omega)$, i.e.  $v_i\in T_0(\Omega)$.  Suppose that $v_i\not\in T_0(\Omega )$ and let $i_0$ be the least among such numbers. It follows from the construction of $T_0(\Omega )$ that there exists $i_1< i_0$ such that the path from $v_{i_1}$ to $v_{i_0}$ contains a prohibited subpath $\s$. From the minimality of $i_0$ it follows that the prohibited path $\s$ goes from $v_{i_2}$, $i_1\le i_2\le i_0$ to $v_{i_0}$.

\subsubsection*{{\rm (I):} Paths $\p(H)$ do not contain prohibited subpaths of type 7-10}

Suppose first that the prohibited path $\s$ is of type 7-10, i.e. $7\leq \tp(v_i)\leq 10$. By definition, there exists a generalised equation $\Omega_{v_{k_1}}$ that repeats $r=4^{\rho_{\Omega_{v_{k_1}}}}+1$ times, i.e.
$$
\Omega_{v_{k_1}}=\dots=\Omega_{v_{k_r}}.
$$
Since the path $\s$ is prohibited, we may assume that $v_{k_i}\ne v_{k_{i+1}}$ for all $i$ and $v_{k_i}\ne v_{k_i+1}$ for all $i$, i.e.
$(\Omega_{v_{k_i+1}},H^{[{k_i+1}]})=\edge(\Omega_{v_{k_i}},H^{[{k_i}]})$.

We now prove that there exist $k_j$ and $k_{j'}$, $k_{j}<k_{j'}$ such that $H^{[k_{j'}]}<_{\Aut(\Omega)}H^{[k_j]}$.

Since, the number of different $\rho_{\Omega_{v_{k_1}}}$-type vectors (see Definition \ref{defn:nfmatrix}) is bounded above by $r$, if the generalised equation $\Omega_{v_{k_1}}$ repeats $r$ times, there exist $k_j$ and $k_{j'}$ such that $H^{[k_j]}$ and $H^{[k_{j'}]}$ have the type vectors, i.e. satisfy condition (\ref{it:minsol3}) from Definition \ref{defn:sol<}. Moreover, by Lemma \ref{3.3}, $\pi(v_{k_j},v_{k_{j'}})$ is an automorphism of $G_{R(\Omega_{v_{k_j}}^\ast)}$ invariant with respect to the kernel of ${\Omega}_{v_{k_j}}$.

By Remark \ref{rem:leng<}, we have $|H^{[k_j]}| >|{H^{[k_{j'}]}}|$. This derives a contradiction, since, by construction of the sequence (\ref{3.8}) one has $v_{k_j+1}=v_{k_j}$.

\subsubsection*{{\rm (II):} Paths $\p(H)$ do not contain prohibited subpaths of type 12.}
Suppose next that the path $\s$ is prohibited  of type 12, i.e. $\tp(v_i)=12$. An analogous argument to the one for prohibited paths of type 7-10, but using Lemma \ref{lem:c12} instead of Lemma \ref{3.3}, leads to a contradiction. Hence, we conclude that $v_i\in T_0(\Omega)$, where $\tp(v_i)=12$.

\subsubsection*{{\rm (III):} Paths $\p(H)$ do not contain prohibited subpaths of type 15.}
Finally, suppose that the path $\s$ is prohibited of type 15, i.e. $\tp(v_i)=15$. Abusing the notation, we consider a subpath of (\ref{3.8})
$$
(\Omega _{v_1},H^{[1]})\to (\Omega_{v_2},H^{[2]})\to\ldots \to(\Omega_{v_m},H^{[m]})\to\ldots,
$$
where $v_1,v_2,\ldots$ are vertices of the tree $T_0(\Omega)$, $\tp(v_i)=15$ and the edges $v_i\to v_{i+1}$ are principal for all $i$. Notice, that by construction the above path is the path defined by the solution $H^{(1)}=H^{[1]}$:
\begin{equation}\label{3.11}
(\Omega _{v_1},H^{(1)})\to (\Omega_{v_2},H^{(2)})\to\ldots \to(\Omega_{v_m},H^{(m)})\to\ldots,
\end{equation}

To simplify the notation, below we write $\rho_i$ for $\rho_{\Omega_{v_i}}$.

Let $\omega =\{\mu _1,\ldots ,\mu_{m}, \ldots\}$ be the set of carrier bases $\mu _i$ of the generalised equations $\Omega_{v_i}$'s and let $\tilde\omega$ denote the set of bases which are transfer bases for at least one generalised equation in (\ref{3.11}). By $\omega_2$ we denote the set of all bases $\nu$ of $\Omega_{v_i}$, $i=1,\dots,m,\dots$ so that $\nu, \Delta(\nu)\notin \omega\cup\tilde \omega$.  Let
$$
\alpha (\omega)=\min \left\{\min\limits_{\mu\in\omega _2}\{\alpha(\mu)\},\rho_A\right\},
$$
where $\rho_A$ is the boundary between the active part and the non-active part.

For every element $(\Omega _{v_i},H^{(i)})$ of the sequence (\ref{3.11}), using $\D 3$, if necessary, we make the section $[1,\alpha (\omega)]$ of the generalised equation $\Omega _{v_i}$ closed and set $[\alpha(\omega),\rho_i]$ to be the non-active part of the generalised equation $\Omega_{v_i}$ for all $i$.

Recall that by $\omega _1$ we denote the set of all variable bases $\nu $ for which either $\nu$ or $\Delta (\nu)$ belongs to the active part $[1,\alpha(\omega)]$ of the  generalised equation $\Omega_{v_1}$, see Definition \ref{defn:excess}.

\subsubsection*{{\rm(III.1):} Lengths of minimal solutions are bounded by a function of the excess}

Let $H$ be a solution of the generalised equation $\Omega$ and let $[1,j+1]$ be the quadratic part of $\Omega$. Set
\glossary{name={$d_1(H)$}, description={length of the ``quadratic part of the solution'' $H$}, sort=D}\glossary{name={$d_2(H)$}, description={length of the ``quadratic-coefficient part of the solution'' $H$}, sort=D}
$$
d _1(H)=\sum \limits_{i=1}^{j}|H_i|, \ d_2(H)=\sum \limits_{\nu}|H(\nu)|,
$$
where $\nu$ is a quadratic-coefficient base.

\begin{lem} \label{2.8}
Let $v$ be a vertex of $T(\Omega)$,  $\tp(v)=15$. There exists a recursive function $f_{1}(\Omega _v)$ such that for any solution $H$ minimal with respect to $\VV(\Omega_v)$ one has
$$
d_1(H)\leq f_{1}(\Omega_v)\max \left\{d_2(H),1\right\}.
$$
\end{lem}
\begin{proof}
Since $\tp(v)=15$, every boundary that touches a base is $\eta$-tied in every base $\eta$ which it intersects.  Instead of $\Omega _v$ below we consider the generalised equation $\widetilde{\Omega}_v=\D 3(\Omega_v)$ which does not have any boundary connections. Then $G_{R(\Omega_v^\ast)}$ is isomorphic to $G_{R({\widetilde{\Omega}}_v^\ast)}$. We abuse the notation and denote $\widetilde{\Omega}_v$ by $\Omega_v$.

Consider the sequence
$$
(\Omega_v ,H)=(\Omega _{v_0}, H^{(0)})\to (\Omega _{v_1}, H^{(1)})\to\ldots  \to (\Omega _{v_i}, H^{(i)})\to \ldots,
$$
where $(\Omega _{v_{j+1}}, H^{(j+1)})$ is obtained from $(\Omega _{v_j}, H^{(j )})$ by applying the entire transformation $\D 5$ in the quadratic part of $\Omega$. Denote by $\mu_{i}$ the carrier base of the generalised equation $\Omega _{v_i}$ and consider the sequence $\mu_1,\dots,\mu_i,\ldots$

We use an argument analogous to the one given in the proof of Lemma \ref{lem:c12} to show that the number of consecutive quadratic bases in the sequence $\mu_1,\dots,\mu_i,\ldots$ is bounded above. The entire transformation applied in the quadratic part of a generalised equation, does not increase the complexity and the number of items. The number of generalised equations with a bounded number of items and bounded complexity is finite.

We now prove that if a generalised equation $\Omega_{v_{k_1}}$ repeats $r=4^{\rho_{k_1}}+1$ times, then there exist $k_j$ and $k_{j'}$ such that  $H^{(k_{j'})}<_{\VV(\Omega_{v_{k_j}})}H^{(k_j)}$.

Since, the number of different $\rho_{k_1}$-type vectors (see Definition \ref{defn:nfmatrix}) is bounded above by $r$, if the generalised equation $\Omega_{v_{k_1}}$ repeats $r$ times, there exist $k_j$ and $k_{j'}$ such that $H^{(k_j)}$ and $H^{(k_{j'})}$ have the same type vectors, i.e. satisfy condition (\ref{it:minsol3})  from Definition \ref{defn:sol<}.

Moreover, by Lemma \ref{lem:c12}, the automorphism $\pi(v_{k_j},v_{k_{j'}})$ of $G_{R(\Omega_{v_{k_j}}^\ast)}$ is invariant with respect to the non-quadratic part of $\Omega_{v_{k_j}}$.

By Remark \ref{rem:leng<}, we have $|{H^{(k_{j'})}}|< |H^{(k_j)}|$. We, therefore, have that ${H^{(k_{j'})}} <_{\VV(\Omega_{v_{k_j}})} H^{(k_j)}$, contradicting the minimality of $H^{(k_j)}$. Thus, we have proven that the number of consecutive quadratic carrier bases is bounded.

Furthermore, since whenever the carrier base is a quadratic-coefficient base the number of bases in the quadratic part decreases, there exists an integer $N$ bounded above by a computable function of the generalised equation, such that the quadratic part of $\Omega_{v_N}$ is empty.

We prove the statement of the lemma by induction on the length $i$ of the sequence. If $i=N-1$, since the application of the entire transformation to $\Omega_{v_{N-1}}$ results in a generalised equation with the empty quadratic part, the carrier $\mu_{N-1}$ is a quadratic-coefficient base and all the other bases $\nu_{N-1,1},\dots,\nu_{N-1,n_{N-1}}$ are transfer bases and are transferred from the carrier to its dual. Since the length $|H^{(N-1)}(\nu_{N-1,i})|$ of every transferred base is less than the length of the carrier, we get that
$$
d_1(H^{(N-1)})\le\sum\limits_{n=1}^{n_{N-1}}|H^{(N-1)}(\nu_{N-1,n})|\le n_{N-1}|H^{(N-1)}(\mu_{N-1})| = n_{N-1}d_2(H^{(N-1)}).
$$

By induction, we may assume that $d_1(H^{(i+1)})\leq g_{i+1}(\Omega_{v_{i+1}})\max \left\{d_2( H^{(i+1)}),1\right\}$, where $g_{i+1}$ is a certain computable function. We prove that the statement holds for $H^{(i)}$.

Suppose that the carrier base $\mu_i$ is quadratic and let $\nu_{i,1},\dots,\nu_{i,n_{i}}$ be the transfer bases of $\Omega_{v_i}$. Then
$$
|H^{(i)}(\mu_i)|-|H^{(i+1)}(\mu_i)|\le \sum\limits_{n=1}^{n_{i}}|H^{(i+1)}(\nu_{i,n})|,
$$
where $\nu_{i,n}$, $n=1,\dots, n_i$ are bases of $\Omega_{v_{i+1}}$. Thus, by induction hypothesis, we get
$$
\sum\limits_{n=1}^{n_{i}}|H^{(i+1)}(\nu_{i,n})|\le n_i g_{i+1}(\Omega_{v_{i+1}})\max \left\{d_2( H^{(i+1)}),1\right\}.
$$
Notice that the sets of quadratic-coefficient bases of $\Omega_{v_i}$ and $\Omega_{v_{i+1}}$ coincide, thus $d_2(H^{(i+1)})=d_2(H^{(i)})$. Therefore,
\begin{gather}\notag
\begin{split}
d_1(H^{(i)})=\left(|H^{(i)}(\mu_i)|\right.&\left.-|H^{(i+1)}(\mu_i)|\right)+ d_1(H^{(i+1)})\le \\
\le \left(|H^{(i)}(\mu_i)|\right.&\left.-|H^{(i+1)}(\mu_i)|\right)+ g_{i+1}(\Omega_{v_{i+1}}) \cdot\max \left\{d_2( H^{(i+1)}),1\right\}\le\\
&\le (n_i+1) \cdot g_{i+1}(\Omega_{v_{i+1}}) \cdot\max \left\{d_2( H^{(i+1)}),1\right\}=g_{i}(\Omega_{v_{i}})\cdot\max \left\{d_2( H^{(i)}),1\right\}.
\end{split}
\end{gather}

Suppose now that the carrier base $\mu_i$ is a quadratic-coefficient base and let $\nu_{i,1},\dots,\nu_{i,n_{i}}$ be the transfer bases of the generalised equation $\Omega_{v_i}$. Since the duals of the transferred quadratic bases become quadratic-coefficient and since $|H^{(i)}(\nu_{i,n})|\le |H^{(i)}(\mu_i)|$, $n=1,\dots, n_i$, we get that
\begin{gather} \notag
\begin{split}
d_2(H^{(i+1)})\le& L+|H^{(i+1)}(\mu_i)+\sum \limits_{n=1}^{n_{i}}|H^{(i)}(\nu_{n_i,n})|\le\\
\le& L+ (n_i+1) |H^{(i)}(\mu_{i})|\le (n_{i}+1)\left(L+|H^{(i)}(\mu_i)|\right)=(n_i+1)\cdot d_2(H^{(i)}),
\end{split}
\end{gather}
where $L=\sum\limits_\lambda|H^{(i)}(\lambda)|$ and the sum is taken over all bases that are quadratic-coefficient in both $\Omega_{v_i}$ and $\Omega_{v_{i+1}}$. Therefore,
\begin{gather} \notag
\begin{split}
  d_1(H^{(i)})=\left(|H^{(i)}(\mu_i)|\right.&- \left.|H^{(i+1)}(\mu_i)|\right)+ d_1(H^{(i+1)})\le \\
            \le |H^{(i)}(\mu_i)|&+g_{i+1}(\Omega_{v_{i+1}}) \cdot\max \left\{d_2( H^{(i+1)}),1\right\}\le\\
                    \le |H^{(i)}&(\mu_i)|+(n_i+1)\cdot g_{i+1}(\Omega_{v_{i+1}})\cdot\max \left\{d_2( H^{(i)}),1\right\}\le \\
&\le(n_{i}+2)\cdot g_{i+1}(\Omega_{v_{i+1}}) \cdot \max  \left\{d_2( H^{(i)}),1\right\}=g_{i}(\Omega_{v_{i}})\cdot \max \left\{d_2(H^{(i)}),1\right\}.
\end{split}
\end{gather}
The statement of the lemma follows.
\end{proof}

Recall, that by $\widetilde{\Omega}$ we denote the generalised equation obtained from $\Omega$ applying $\D 3$. Consider the section  of $\widetilde{\Omega} _{v_1}$ of the form  $[1,\alpha(\omega)]$. The section $[1,\alpha(\omega)]$ lies in the quadratic part of $\widetilde{\Omega} _{v_1}$. Let $B'$ be the set of quadratic bases that belong to $[1,\alpha(\omega)]$ and let $\VV'(\Omega_{v_1})$ be the group of automorphisms of $G_{R(\Omega^\ast)}$ that are invariant with respect to the non-quadratic part of $\widetilde{\Omega} _{v_1}$ and act identically on all the bases which do not belong to $B'$. By definition, $\VV'(\Omega_{v_1})\le \VV(\Omega_{v_1})$. Thus the solution $H^{(1)}$ minimal with respect to $\VV(\Omega_{v_1})$ is also minimal with respect to $\VV'(\Omega_{v_1})$. By Lemma \ref{2.8} we have
\begin{equation} \label{3.14}
d_1(H^{(1)})\leq f_{1}(\Omega _{v_1}) \max\left\{d_2(H^{(1)}),1\right\}.
\end{equation}

Recall that  (see Definition \ref{defn:excess})
\begin{equation} \label{eq:M}
d_{A\Sigma}(H)=\sum \limits_{i=1}^{\alpha(\omega)-1}|H_i|, \quad \psi_{A\Sigma}(H)=\sum \limits_{\mu\in \omega_1}|H({\mu})|-2d_{A\Sigma}(H).
\end{equation}

Our next goal is, using inequality (\ref{3.14}), to give an upper bound of the length of the interval $d_{A\Sigma}(H^{(1)})$ in terms of the excess $\psi _{A\Sigma}$ and the function $f_{1}(\Omega _{v_1})$. More precisely, we have the following lemma.

\begin{lem}\label{lem:3.19}
In the above notation, the following inequality holds
\begin{equation}\label{3.19}
d_{A\Sigma}(H^{(1)})\leq \max\left\{\psi_{A\Sigma}(H^{(1)})(2n f_{1}(\Omega_{v_1})+1), f_{1}(\Omega _{v_1})\right\}.
\end{equation}
\end{lem}
\begin{proof}
Denote by \glossary{name={$\gamma_i(\omega)$}, description={the number of bases from $\omega_1$ that contain $h_i$}, sort=G}$\gamma_i(\omega)$ the number of bases $\mu\in\omega_1$ containing $h_i$. Then
\begin{equation}\label{3.15}
\sum\limits_{\mu\in\omega _1}|H^{(1)}({\mu})|=\sum\limits_{i=1}^{\rho}|H_i^{(1)}| \gamma_i(\omega),
\end{equation}
where $\rho =\rho_{\Omega _{v_1}}$. Let $I=\left\{i \mid 1\le i \le\alpha(\omega)-1 \hbox{ and }\gamma_i=2\right\}$, and $J=\left\{i\mid 1\leq i\le \alpha (\omega)-1 \hbox{ and } \gamma_i> 2\right\}$. By (\ref{3.12}) we have:
\begin{equation}\label{3.16}
d_{A\Sigma}(H^{(1)})=\sum \limits_{i\in I}|H_i^{(1)}|+ \sum\limits_{i\in J}|H_i^{(1)}|= d_1(H^{(1)})+
\sum\limits_{i\in J}|H_i^{(1)}|.
\end{equation}

Let  $\lambda ,\Delta(\lambda)$ be a pair of variable quadratic-coefficient bases of the generalised equation $\widetilde{\Omega}_{v_1}$, where $\lambda$ belongs to the non-quadratic part of $\widetilde\Omega _{v_1}$.
When we apply $\D 3$ to $\Omega_{v_1}$ thereby obtaining $\widetilde\Omega _{v_1}$, the pair $\lambda ,\Delta(\lambda)$ is obtained from bases $\mu\in\omega_1$.
There are two types of quadratic-coefficient bases.
\begin{itemize}
    \item[Type 1:] variable bases $\lambda$ such that $\beta(\lambda)\le\alpha(\omega)$. In this case, since $\lambda$ belongs to the non-quadratic part of $\widetilde{\Omega}_{v_1}$, it is a product of items $\{h_i\mid i\in J\}$ and thus $|H({\lambda})|\leq\sum\limits_{i\in J}|H_i^{(1)}|$. Thus the sum of the lengths of quadratic-coefficient bases of Type 1 and their duals is bounded above by $2n\sum\limits_{i\in J}|H_i^{(1)}|$, where $n$ is the number of bases in $\Omega$.
    \item[Type 2:] variable bases $\lambda$ such that $\alpha(\lambda)\ge\alpha(\omega)$. The sum of the lengths of quadratic-coefficient bases of the second type is bounded above by $2\cdot \sum \limits_{i=\alpha(\omega)}^{\rho}|H_i^{(1)}|\gamma _i(\omega)$.
\end{itemize}
We have
\begin{equation}\label{3.17}
d_2(H^{(1)})\le 2n\sum\limits_{i\in J}|H_i^{(1)}|+2\cdot\sum\limits_{i=\alpha(\omega)}^{\rho}|H_i^{(1)}| \gamma_i(\omega).
\end{equation}
Then from (\ref{eq:M}) and (\ref{3.15}) it follows that
\begin{equation}\label{3.18}
\psi_{A\Sigma}(H^{(1)}_i)\ge \sum \limits_{i\in J}|H_i^{(1)}|+\sum\limits_{i=\alpha(\omega)}^{\rho}|H_i^{(1)}|\gamma_i(\omega).
\end{equation}
From  Equation (\ref{3.16}), using inequalities (\ref{3.14}), (\ref{3.17}), (\ref{3.18}) we get inequality (\ref{3.19}).
\end{proof}

\subsubsection*{{\rm (III.2):} Minimal solutions $H$ such that $\p(H)$ contains a prohibited subpath of type 15 fail inequality {\rm(\ref{3.18})}.}
Let the path $v_1\to v_2\to \ldots \to v_m$ corresponding to the sequence (\ref{3.11}) be $\mu$-reducing, that is  $\mu_1=\mu$ and, either there are no outgoing auxiliary edges from $v_2$ and $\mu$ occurs in the sequence $\mu _1,\ldots ,\mu _{m-1}$ at least twice, or $v_2$ does have outgoing auxiliary edges $v_2\to w_1,\dots, v_2\to w_{\nn}$ and the base $\mu$ occurs in the sequence $\mu _1,\ldots ,\mu _{m-1}$ at least $\max\limits_{1\le i\le \nn}\ss(\Omega_{w_i})$ times.

Set $\delta_i=d_{A\Sigma}(H^{(i)})-d_{A\Sigma}(H ^{(i+1)})$. We give a lower bound for $\sum \limits_{i=1}^{m-1}\delta _i$, i.e. we estimate by how much the length of a solution is reduced in a $\mu$-reducing path.

We first prove that if $\mu _{i_1}=\mu _{i_2}=\mu$, $i_1< i_2$ and $\mu _i\ne\mu$ for $i_1< i< i_2$, then
\begin{equation} \label{3.23}
\sum \limits_{i=i_1}^{i_2-1}\delta _i\ge |H^{(i_1+1)}[1,\alpha (\Delta (\mu_{i_1+1}))]|.
\end{equation}
Indeed, if $i_2=i_1+1$ then
$$
\delta _{i_1}=|H^{(i_1)}[1,\alpha(\Delta (\mu))]|=|H^{(i_1+1)}[1,\alpha (\Delta (\mu))]|.
$$
If $i_2 > i_1+1$, then $\mu _{i_1+1}\ne \mu$ and $\mu$ is a transfer base in the generalised equation $\Omega_{v_{i_1+1}}$ and thus
\begin{equation} \label{est1}
\delta_{i_1+1}+|H^{(i_1+2)}[1,\alpha(\mu)]|=|H^{(i_1+1)}[1,\alpha(\Delta(\mu _{i_1+1}))]|.
\end{equation}
Since $\mu$ is the carrier base of $\Omega_{v_{i_2}}$ we have
\begin{equation} \label{est2}
\sum\limits_{i=i_1+2}^{i_2-1}\delta _i\ge |H^{(i_1+2)}[1,\alpha (\mu)]|.
\end{equation}
From (\ref{est2}) and (\ref{est1}) we get (\ref{3.23}).

\bigskip

We want to show that every $\mu$-reducing path reduces the length of a solution $H^{(1)}$ by at least $\frac{1}{10}|H^{(1)}(\mu)|$.

\begin{lem} \label{lem:3.26}
Let $v_1\to\dots\to v_m$ be a $\mu$-reducing path, then  the following inequality holds
\begin{equation}\label{3.26}
\sum \limits_{i=1}^{m-1}\delta _i\ge\frac{1}{10} |H^{(1)}({\mu})|.
\end{equation}
\end{lem}
\begin{proof}
To prove the Lemma we consider the two cases from the definition of a $\mu$-reducing path.

Suppose first that $v_2$ does not have any outgoing auxiliary edges, i.e. the bases $\mu _2$ and $\Delta(\mu _2)$ do not intersect in the generalised equation $\Omega _{v_2}$, then (\ref{3.23}) implies that
$$
\sum\limits_{i=1}^{m-1}\delta_i\ge |H^{(2)}[1,\alpha (\Delta(\mu _2))]|\ge |H^{(2)}({\mu _2})|\ge |H^{(2)}({\mu})|=|H^{(1)}({\mu})|-\delta_1,
$$
which, in turn, implies that
\begin{equation} \label{3.24}
\sum\limits_{i=1}^{m-1}\delta _i\ge\frac{1}{2}|H^{(1)}({\mu})|.
\end{equation}

Suppose now that there are auxiliary edges $v_2 \to w_1,\ldots ,v_2 \to w_{\nn}$. Let $H^{(2)}[1,\alpha (\Delta (\mu _2))]\doteq Q$,
and $P$ be a period such that $Q\doteq P^d$ for some $d\ge 1$, then $H^{(2)}({\mu_2})$ and $H^{(2)}({\mu})$ are initial subwords of the word $H^{(2)}[1,\beta (\Delta (\mu_2))]$, which, in turn, is an initial subword of $P^\infty$.

By construction of the sequence (\ref{3.8}), relation (\ref{3.9}) fails for the vertex $v_2$, i. e. (in the notation of (\ref{3.9})):
\begin{equation}\label{3.25}
H^{(2)}({\mu})\doteq {P}^r \cdot P_1, \ P\doteq P_1\cdot P_2,\ r< \max\limits_{1\leq j\leq \nn} \ss(\Omega _{w_j}).
\end{equation}
Let $\mu _{i_1}=\mu _{i_2}=\mu$, $i_1< i_2$ and $\mu _i\ne \mu$ for $i_1< i< i_2$. If
\begin{equation} \label{3.26a}
|H^{(i_1+1)}({\mu _{i_1+1}})|\ge 2 |P|
\end{equation}
since $H^{(i_1+1)}(\Delta({\mu _{i_1+1}}))$ is a  $Q'$-periodic subword ($Q'$ is a cyclic permutation of $P$) of the $Q'$-periodic word $H^{(i_1+1)}[1,\rho_{i_1+1}+1]$ of length greater than $2|Q'|=2|P|$, it follows by Lemma 1.2.9 in \cite{1}, that $|H^{(i_1+1)}[1,\alpha (\Delta(\mu_{i_1+1}))]|\ge k|Q'|$. As $k\ne 0$ ($\mu_{i+1}$ and $\Delta(\mu_{i+1})$ do not form a pair of matched bases), so $|H^{(i_1+1)}[1,\alpha (\Delta(\mu_{i_1+1}))]|\ge |P|$. Together with (\ref{3.23}) this gives that $\sum \limits_{i=i_1}^{i_2-1}\delta _i\ge |P|$. The base $\mu$ occurs in the sequence $\mu _1,\ldots ,\mu _{m-1}$ at least $r$ times, so either (\ref{3.26a}) fails for some $i_1\le m-1$ or
$\sum \limits_{i=1}^{m-1}\delta_i\ge (r-3)|P|$.

If (\ref{3.26a}) fails, then from the inequality $|H^{(i+1)}({\mu_i})|\le |H^{(i+1)}({\mu _{i+1}})|$ and the definition of $\delta_i$ follows that
$$
\sum\limits_{i=1}^{i_1}\delta_i\ge |H^{(1)}({\mu})|-|H^{(i_1+1)}({\mu _{i_1+1}})|\ge (r-2)|P|.
$$
hence in both cases $\sum \limits_{i=1}^{m-1}\delta_i\ge (r-3)|P|$.

Notice that for $i_1=1$, inequality (\ref{3.23}) implies that $\sum\limits_{i=1}^{m-1}\delta _i\ge |Q|\ge |P|$; so $\sum\limits_{i=1}^{m-1}\delta_i\ge \max\{1,r-3\}|P|$. Together with (\ref{3.25}) this implies that
$$
\sum\limits_{i=1}^{m-1}\delta _i\ge \frac{1}{5}|H^{(2)}({\mu})|=\frac{1}{5}(|H^{(1)}({\mu})|-\delta _1).
$$
Finally we get that
$$
\sum \limits_{i=1}^{m-1}\delta _i\ge\frac{1}{10} |H^{(1)}({\mu})|.
$$
From the above inequality and  inequality (\ref{3.24}), we see that for a $\mu$-reducing path inequality (\ref{3.26}) always holds.
\end{proof}

We thereby have shown that in any $\mu$-reducing path the length of the solution is reduced by at least $\frac{1}{10}$ of the length of the carrier  base $\mu$.

\bigskip

Notice that by property (\ref{it:prp3}) from Definition \ref{defn:proh15}, we can assume that the carrier bases $\mu_i$ and their duals $\Delta(\mu_i)$ belong to the active part $A\Sigma =[1,\alpha(\omega)]$. Then, by Lemma \ref{lem:excess} and by construction of the path (\ref{3.11}), we have that $\psi_{A\Sigma}(H^{(1)})=\dots=\psi_{A\Sigma}(H^{(m)})=\dots$ We denote this number by $\psi_{A\Sigma}$.

\begin{lem}
Let $v_1\to v_2\to \ldots \to v_m$ be a prohibited path of type 15. By definition,  $v_1\to v_2\to \ldots \to v_m$ can be presented in the form (\ref{3.7}). Let $\kappa$ be the length of the subpath $\p_1\s_1\ldots \p_l\s_l$. Then there exists a carrier base $\mu \in \omega$ such that the following inequality holds
\begin{equation}\label{3.27}
|H^{(\kappa)}({\mu })|\ge\frac{1}{2n}\psi _{A\Sigma},
\end{equation}
where $n$ is the number of bases in $\Omega$.
\end{lem}
\begin{proof}
From the definition of $\psi _{A\Sigma}$, see (\ref{eq:M}), we get that $\sum\limits_{\mu \in\omega_1}|H^{(m)}({\mu})|\ge \psi _{A\Sigma}$, hence the inequality $|H^{(m)}({\mu})|\geq\frac{1}{2n}\psi _{A\Sigma}$ holds for at least one base $\mu\in\omega_1$. Since $H^{(m)}({\mu})\doteq \left(H^{(m)}({\Delta (\mu)})\right)^{\pm 1}$, we may assume that $\mu\in\omega\cup\tilde{\omega}$.

If $\mu\in\omega$, then inequality (\ref{3.27}) trivially holds.

If $\mu\in\tilde{\omega}$, then  by the third condition in the definition of a prohibited path of type 15 (see Definition \ref{defn:proh15}) there exists $\kappa\leq i\leq m$ such that $\mu$ is a transfer base of $\Omega _{v_i}$. Hence, $|H^{(\kappa)}({\mu_i})|\ge |H^{(i)}({\mu _i})|\ge |H^{(i)}({\mu})|\ge |H^{(m)}({\mu})|\ge \frac{1}{2n}\psi _{A\Sigma}$.
\end{proof}

\bigskip

Finally, from conditions (\ref{it:prp2}) and (\ref{it:prp1}) in the definition of a prohibited path of type 15, from the inequality $|H^{(i)}({\mu})|\ge |H^{(\kappa)}({\mu})|$,  $1\leq i\leq \kappa$, and from inequalities (\ref{3.26}) and (\ref{3.27}), it follows that
\begin{equation}\label{3.28}
\sum\limits_{i=1}^{\kappa-1}\delta _i\ge \max\left\{\frac{1}{20n}\psi _{A\Sigma},1\right\}\cdot(40n^2f_{1}+20n+1).
\end{equation}

By Equation (\ref{3.22}), the sum in the left part of the inequality (\ref{3.28}) equals $d_{A\Sigma}(H^{(1)})-d_{A\Sigma}(H^{(\kappa)})$, hence
\begin{equation}\label{eq:3.29}
d_{A\Sigma}(H^{(1)})\ge \max \left\{\frac{1}{20n}\psi _{A\Sigma},1\right\}\cdot(40n^2f_{1} +20n +1),
\end{equation}
which contradicts the statement of Lemma \ref{lem:3.19}.

Therefore, the assumption that there are prohibited subpaths (\ref{3.11}) of type 15 in the path (\ref{3.8}) led to a contradiction. Hence, the path (\ref{3.8}) does not contain prohibited subpaths. This implies that $v_i\in T_0(\Omega )$ for all $(\Omega _{v_i},H^{(i)})$ in (\ref{3.8}).

In particular, we have shown that final leaves $w$ of the tree $T(\Omega)$ are, in fact, leaves of the tree $T_0(\Omega)$. Naturally, we call  such leaves the \index{leaf!final of the tree $T_0$}\emph{final leaves of $T_0(\Omega)$}.

\subsubsection*{\textbf{{\rm (C):} The pair $(\Omega_w,H^{[w]})$, where $\tp(w)=2$, satisfies the properties required in Proposition \ref{3.4}}}

For all $i$, either $v_i= v_{i+1}$  and $|H^{[i+1]}|<|H^{[i]}|$, or $v_i\to v_{i+1}$ is an edge of a finite tree $T_0(\Omega)$. Hence the sequence (\ref{3.8}) is finite. Let $(\Omega _{w}, H^{[w]})$ be its final term. We show that $(\Omega _{w},H^{[w]})$ satisfies the properties required in the proposition.

Property (\ref{it:prop1}) follows directly from  the construction of $H^{[w]}$.

We now prove that property (\ref{it:prop2}) holds. Let $\tp (w)=2$ and suppose that $\Omega _w$ has non-constant non-active sections. It follows from the construction of (\ref{3.8}) that if $[j,k]$ is an active section of $\Omega_{v_{i-1}}$ and is a non-active section of $\Omega_{v_i}$ then $H^{[i]}[j,k]\doteq H^{[i+1]}[j,k]\doteq\ldots \doteq H^{[w]}[j,k]$. Therefore, (\ref{3.9}) and the definition of $\ss(\Omega _v)$ imply that the word $h_1\ldots h_{\rho_w}$ can be subdivided into subwords $h[1,i_1],\ldots ,h[i_{l'-1},i_{\rho_w}]$, such that for any $l$ either $H^{[w]}[i_l,i_{l+1}]$ has length $1$, or the word $h[i_l,i_{l+1}]$ does not appear in basic, factor and coefficient equations, or
\begin{equation}\label{3.29}
H^{[w]}[i_l,i_{l+1}]\doteq P_l^r \cdot P_l';\quad P_l\doteq P_l'P_l''; \quad r\ge \rho_w \max\limits_{\langle \P,R\rangle  } \left\{ f_{0}(\Omega_w, \P, R)\right\},
\end{equation}
where $P_l$ is a period, and the maximum is taken over all regular periodic structures of $\widetilde{\Omega}_w$. Therefore, if we choose $P_{l}$ of maximal length, then $\widetilde{\Omega }_w$ is  singular  with respect to the periodic structure $\P(H^{[w]},P_l)$. Indeed, suppose that it is regular with respect to this periodic structure. Then, as in $H^{[w]}[i_l,i_{l+1}]$ one has $i_{l+1}-i_l\le \rho_w$, so (\ref{3.29}) implies that there exists $h_k$ such that $|H^{[w]}_k|\ge f_{0}(\Omega_w, \P, R)$. By Lemma \ref{lem:23-2}, this contradicts the minimality of the solution $H^{[w]}$.

This finishes the proof of Proposition \ref{3.4}.

\section{From the coordinate group $G_{R(\Omega^*)}$ to proper quotients: \newline The decomposition tree $T_{\dec}(\Omega )$}\label{5.5.5}

We proved in the previous section that for every solution $H$ of a generalised equation $\Omega$, the path $\p(H)$ associated to the solution $H$ ends in a final leaf $v$ of the tree $T_0(\Omega )$, $\tp(v)=1,2$. Furthermore, if $\tp(v)=2$ and the generalised equation $\Omega_{v}$ contains non-constant non-active sections, then the generalised equation $\Omega _v$ is singular with respect to the periodic structure ${\P}(H^{(v)},P)$, see Proposition \ref{3.4}.

The essence of the decomposition tree $T_{\dec}(\Omega)$ is that to every solution $H$ of $\Omega$ one can associate the path $\p(H)$ in \glossary{name={$T_{\dec}(\Omega)$}, description={the decomposition tree of $\Omega$}, sort=T} $T_{\dec}(\Omega)$ such that either all sections of the generalised equation $\Omega_u$ corresponding to the leaf $u$ of $T_{\dec}(\Omega)$  are non-active constant sections or the coordinate group of $\Omega_u$ is a proper quotient of $G_{R(\Omega^*)}$.

We summarise the results of this section in the proposition below.
\begin{prop}\label{prop:dectree}
For a generalised equation $\Omega=\Omega_{v_0}$, one can effectively construct a finite oriented rooted at $v_0$ tree $T_{\dec}$, $T_{\dec}=T_{\dec}(\Omega_{v_0})$, such that:
\begin{enumerate}
\item The tree $T_0(\Omega)$ is a subtree of the tree $T_{\dec}$.
\item To every vertex $v$ of $T_{\dec}$ we assign a finitely generated group of automorphisms $A(\Omega_v)$.
\item For any solution $H$ of a generalised equation $\Omega $ there exists a leaf $u$ of the tree $T_{\dec}$, $\tp(u)=1,2$, and a solution $H^{[u]}$ of the generalised equation $\Omega _u$ such that
\begin{itemize}
    \item  $\pi_H= \sigma_0 \pi(v_0,v_1)\sigma_{1} \ldots  \pi(v_{n-1},u)\sigma_n \pi_{H^{[u]}}$, where $\sigma_i\in A(\Omega_{v_i})$;
    \item  if $\tp(u)=2$, then all non-active sections of $\Omega_u$ are constant sections.
\end{itemize}
\end{enumerate}
\end{prop}

To obtain $T_{\dec}(\Omega )$ we add some edges labelled by proper epimorphisms (to be described below) to the final leaves of type $2$ of $T_0(\Omega)$ so that the corresponding generalised equations contain non-constant non-active sections.

The idea behind the construction of this tree is the following. Lemma \ref{lem:23-1} states that given a generalised equation $\Omega$ singular with respect to a periodic structure $\langle\P, R\rangle$, there exist finitely many proper quotients of the coordinate group $G_{R(\Omega^*)}$ such that for every $P$-periodic solution $H$ of the generalised equation $\Omega$ such that $\P(H,P)=\langle\P, R\rangle$, an $\AA(\Omega)$-automorphic image $H^+$ of $H$ is a solution of a proper equation. In other words, the $G$-homomorphism $\pi_{H^+}: G_{R(\Omega^*)} \to G$ factors through one of the finitely many proper quotients of $G_{R(\Omega^*)}$.

Recall that given a coordinate group of a system of equations, the homomorphisms $\pi$ from this coordinate group to the coordinate groups of generalised equations determined by partition tables (see Equation (\ref{eq:hompt}) and discussion in Section \ref{sec:relsol2ge}), are, in general, just homomorphisms (neither injective nor surjective). Our goal here is to prove that in fact, the solution $H^+$ is a solution of one of the finitely many proper generalised equations such that the homomorphism from the coordinate group of the system of equations to the coordinate group of the generalised equation is an epimorphism.

\medskip

Let $v$ be a final leaf of type $2$ of $T_0(\Omega)$ such that $\Omega_v$ contains non-constant non-active sections. Consider a triple $(\langle \P,R\rangle,\cc,\T)$, where
\begin{itemize}
    \item $\Omega_v$ is singular with respect to the periodic structure $\langle{\P}, R\rangle$,
    \item  $\cc$ is a cycle in $\Gamma$, $\cc\in\{\cc_1,\dots,\cc_r\}$, where $\{\cc_1,\dots,\cc_r\}$ is the set of cycles from Lemma \ref{lem:23-1},
    \item and $\T$ is a partition table from the set $\PT'$ defined below.
\end{itemize}
For every such triple, we construct a generalised equation ${\Omega_v(\P, R,\cc,\T)}$, a homomorphism  $\pi_{v,\cc,\T}:G_{R(\Omega^*_v)}\to G_{R({\Omega_v(\P, R,\cc,\T)}^\ast)}$ and put an edge $v\to u$, $\Omega_u=\Omega_v(\P, R,\cc,\T)$. We then show that the homomorphism  $\pi_{v,\cc,\T}$ is a proper epimorphism and that every solution $H^+$ of $\Omega_v$ is a solution of one of the generalised equations $\Omega_v(\P, R,\cc,\T)$ for some $\cc$ and $\T$. The vertex $u$ is a leaf of $T_{\dec}(\Omega )$.

The way we proceed is the following. By part (\ref{it:23-13}) of Lemma \ref{lem:23-1}, we know that $H^+$ is a solution of the generalised equation $\{h(\mu)=h(\Delta(\mu))\mid \mu \notin \P\}$. This fact pilots the construction of the generalised equation $\check{\Omega}_v$ below. On the other hand, we consider the system of equations $\bS$ over $G$ corresponding to the bases from $\P$ and the cycle $\cc$ (see below). Since the solution $H^+$ satisfies statement (\ref{it:23-13}) of Lemma \ref{lem:23-1}, we can prove that the $G$-partition table $\T$ associated to $H^+$ satisfies property (\ref{3.31}). For the partition tables $\T$ that satisfy condition (\ref{3.31}) and the corresponding generalised equations $\Omega_\T$, the homomorphism $\pi_{\Omega_\T}:G_{R(\bS)} \to G_{R(\Omega^*_\T)}$ is surjective. We construct the generalised equation ${\Omega_v(\P, R,\cc,\T)}$ from $\Omega_\T$ and $\check{\Omega}_v$ by identifying the items they have in common.

We now formalise the construction of $G_{R({\Omega_v(\P, R,\cc,\T)}^\ast)}$ and $\pi_{v,\cc,\T}$.

Let $\cc$ be as above and suppose that $h(\cc)=h_{i_1}\cdots h_{i_k}$. Consider the system of equations $\bS \subset G[h]$ over $G$:
$$
\bS = \left\{ h(\cc_\mu)=1, h(\cc)=h_{i_1}\cdots h_{i_k}=1 \mid \mu \in \P \right\}.
$$

Below we use the notation of Section \ref{sec:consgeneq}. For the system of equations $\bS$ consider the subset $\PT'$ of the set $\PT(\bS)$ of all partition tables of $\bS$ that satisfy condition (\ref{3.31}). Consider the set of generalised equations $\{ \Omega_\T \mid \T \in \PT' \}$. For any generalised equation $\Omega_\T$, $\T\in \PT'$, since the partition table $\T$ satisfies condition (\ref{3.31}), it follows that the homomorphism $\pi_{\Omega_\T}:G_{R(\bS)} \to G_{R(\Omega^*_\T)}$ induced by the map $h_i\mapsto P_{h_i}(h',\cA)$ (see Section \ref{sec:relsol2ge} for definition) is surjective.

Let the generalised equation $\check{\Omega}_v$ be obtained from $\widetilde{\Omega}_v$ by removing all bases and items that belong to $\P$.

Construct the generalised equation $\Omega_v(\P,R,\cc,\T)$ as follows. The set of items $\tilde h$ of $\Omega_{v}(\P,R,\cc,\T)$ is the disjoint union of the items of $\check{\Omega}_v$ and $\Omega_\T$; the set of coefficient and factor equations of $\Omega_v(\P,R,\cc,\T)$ is the disjoint union of the coefficient and factor equations of $\check{\Omega}_v$ and $\Omega_\T$. The set of basic equations of $\Omega_v(\P,R,\cc,\T)$ consists of: the basic equations of $\check{\Omega}_v$, the basic equation of $\Omega_\T$, and basic equations of the form $h_k= P_{h_k}(h',\cA)$, $h_k \notin \P$, where in the left hand side of this equation $h_k$ is treated as a variable of $\check{\Omega}_v$ and $P_{h_k}(h',\cA)$, $h_k \notin \P$ is a label of a section of $\Omega_\T$. The initial and terminal terms of boundaries in $\Omega_v(\P,R,\cc,\T)$ are naturally induced from $\check{\Omega}_v$ and $\Omega_\T$.

The natural homomorphism $\pi_{v,\cc,\T}:G_{R(\Omega^*_v)}\to G_{R({\Omega_v(\P, R,\cc,\T)}^\ast)}$, induced by a map $\varpi$:
$$
h_i\mapsto \left\{
\begin{array}{ll}
P_{h_i}(h',\cA),&\hbox{ when $h_i\in \P$,} \\
h_i, & \hbox{ otherwise};
\end{array}\right.
$$
is surjective, since so is $\pi_{\Omega_\T}$. Moreover, by construction, $\pi_{v,\cc,\T}(h(\cc))=1$, and therefore $\pi_{v,\cc,\T}$ is a proper epimorphism.

We now show that for every $P$-periodic solution $H$ of $\Omega_v$ such that $\Omega_v$ is singular with respect to the periodic structure $\P(H,P)$, there exists an $\AA(\Omega_v)$-automorphic image $H^+$ of $H$ such that $H^+$ factors through  one of the solutions of the generalised equations $\Omega_v(\P,R,\cc,\T)$, for some choice of $\cc$ and $\T$. Indeed, let $H^+$  be the solution constructed in Lemma \ref{lem:23-1} and $\cc$ be one of the cycles from Lemma \ref{lem:23-1} for which $H^+(\cc)=1$. The solution $H^+$ is a solution of the system $\bS$. All equations of $\bS$ have the form $h(\cc')=1$, where $\cc'$ is a cycle in the graph of the periodic structure $\P(H,P)$. From condition (\ref{it:23-13}) of Lemma \ref{lem:23-1}, it follows that, on one hand, the word $H_k^+$, $1\le k\le \rho$, is in the normal form, and, on the other hand, that if $r_1r_2\cdots r_l=1$ is an equation  of the system $\bS$ and $H_{j_1}^+,\dots, H_{j_l}^+$ are the respective components of $H^+$, then the word $H_{j_1}^+\cdots H_{j_l}^+$ is trivial in the free group $F(\cA)$. Furthermore, in the product of two consecutive subwords $H_{j_r}^+$ and $H_{j_{r+1}}^+$ either there is no cancellation, or one of these words cancels completely, that is either $H_{j_r}^+\doteq W\cdot {\left(H_{j_{r+1}}^+\right)}^{-1}$ or $H_{j_{r+1}}\doteq {\left(H_{j_{r}}\right)}^{-1}\cdot W$. Let $\T$ be the partition table corresponding to the cancellation scheme of the system $\bS$. The argument above shows that $\T$ satisfies condition (\ref{3.31}).

Let $v_0=v$,  $\Omega_{v_1}=\Omega_v(\P, R,\cc,\T)$.

The solution $H^+$ induces a solution $H^+_\bS$ of the system $\bS$. By Lemma \ref{le:R1}, there exists a solution $H^+_{\Omega_\T}$ of the generalised equation $\Omega_\T$ such that the following diagram
$$
\xymatrix@C3em{
G_{R(\bS)} \ar[rd]_{\pi_{H^+_{\bS}}}  \ar[rr]  &     &  G_{R(\Omega_\T^\ast)}  \ar[ld]^{\pi_{H^+_{\Omega_\T}}} \\
                                                &  G &
}
$$
is commutative.

As $H^+$ satisfies condition (\ref{it:23-13}) of Lemma \ref{lem:23-1}, and the $H_i$'s are in the normal form, so the $H^+_i$'s are in the normal form. In particular, $H^+_{\Omega_\T}$ is a solution of the generalised equation $\Omega_\T$.

On the other hand, using again part (\ref{it:23-13}) of Lemma \ref{lem:23-1}, if we remove the components $\{H^+_k\mid h_k\in \P\}$ of the solution $H^+$, we get a solution $\check{H}^+$ of the generalised equation $\check{\Omega}_v$.

Combining the solutions $\check{H}^+$ and $H^+_{\Omega_\T}$ we get a solution $H^{(v_1)}$ of the generalised equation $\Omega_v(\P, R,\cc,\T)$.  By part (\ref{it:23-13}) of Lemma \ref{lem:23-1}, the solution $H^+$ satisfies the type constraints. Therefore, from the construction it follows that $H^{(v_1)}$ is a solution of $\Omega_v(\P, R,\cc,\T)$ and
$$
\pi_{H^{+}}=\pi(v_0,v_1) \pi_{H^{(v_1)}}.
$$
We thereby have shown that for every $P$-periodic solution $H$ of $\Omega_v$ such that $\Omega_v$ is singular with respect to the periodic structure $\P(H,P)$, there exists an $\AA(\Omega_v)$-automorphic image $H^+$ of $H$ such that $H^+$ factors through one of the solutions of the generalised equations $\Omega_v(\P,R,\cc,\T)$.

To the root vertex $v_0$ of the decomposition tree  $T_{\dec(\Omega)}$ we associate the group of automorphisms $\Aut(\Omega)$ of the coordinate group $G_{R(\Omega_{v_0})}$, see Definition \ref{defn:Aut}. To the vertices $v$ such that $v$ is a leaf of $T_0(\Omega)$ but not of $T_{\dec}(\Omega)$ (those vertices to which we added edges), we associate the group of automorphisms generated by the groups $\AA(\Omega_v)$, see Definition \ref{defn:AA}, corresponding to all periodic structures on $\Omega_v$ with respect to which $\Omega_v$ is singular. To all the other vertices of $T_{\dec}(\Omega)$ we associate the trivial group of automorphisms. We denote the automorphism group associated to a vertex $v$ of the tree $T_{\dec}(\Omega)$ by \glossary{name={$A(\Omega_v)$}, description={the automorphism group associated to a vertex $v$ of the trees $T_{\dec}(\Omega)$ and $T_{\sol}(\Omega)$}, sort=A}$A(\Omega_v)$.

Naturally, we call leaves $u$ such that the paths $\p(H)$ associated to solutions $H$ of $\Omega$ end in $u$, \index{leaf!final of the tree $T_{\dec}$}\emph{final leaves of $T_{\dec}(\Omega)$}.

\section{The solution tree $T_{\sol}(\Omega)$ and main results}\label{se:5.5}

Recall that the coordinate group $G_{R(\Omega_v^\ast)}$ associated to  a final leaf $v$  of $T_{\dec}(\Omega)$ is either a proper epimorphic image of $G_{R(\Omega^\ast)}$ or all the sections of the corresponding generalised equation $\Omega_v$ are non-active constant sections. We now use the assumption that the group $G$ is equationally Noetherian to get that any sequence of proper epimorphisms of coordinate groups is finite. An inductive argument for those leaves of $T_{\dec}(\Omega)$ that are proper epimorphic image of $G_{R(\Omega^\ast)}$ shows that we can construct a tree \glossary{name={$T_{\sol}(\Omega)$}, description={the extension tree of $\Omega$}, sort=T}$T_{\sol}$ with the property that for every leaf $v$ of $T_{\sol}$ all the sections of the generalised equation $\Omega_v$ are non-active constant sections.

We define a new transformation $L_v$, which we call a \emph{leaf-extension of the tree $T_{\dec}(\Omega)$ at the leaf $v$} in the following way. If there are active sections in the generalised equation $\Omega_v$, we take the union of two trees $T_{\dec}(\Omega)$ and $T_{\dec}(\Omega_v)$ and identify the leaf $v$ of $T_{\dec}(\Omega)$ with the root $v$ of the tree $T_{\dec}(\Omega_v)$, i.e. we extend the tree $T_{\dec}(\Omega)$ by gluing the tree $T_{\dec}(\Omega_v)$ to the vertex $v$.  If all the sections of the generalised equation $\Omega_v$ are non-active, then the vertex $v$ is a leaf and $T_{\dec}(\Omega_v)$ consists of a single vertex, namely $v$. We call such a vertex $v$ \index{terminal vertex}\emph{terminal}.

We use induction to construct the solution tree $T_{\sol}(\Omega)$. Let $v$ be a non-terminal leaf of $T^{(0)} = T_{\dec}(\Omega)$. Apply the transformation $L_v$ to obtain a new tree $T^{(1)} = L_v(T_{\dec}(\Omega))$. If in $T^{(1)}$ there exists a non-terminal leaf $v_1$, we apply the transformation $L_{v_1}$, and so on. By induction we construct a strictly increasing sequence of finite trees
\begin{equation} \label{eq:5.3.1}
T^{(0)} \subset  T^{(1)} \subset \ldots \subset T^{(i)}  \subset \ldots
\end{equation}
Sequence (\ref{eq:5.3.1}) is finite. Indeed, assume the contrary, i.e. the sequence is infinite and hence the union $T^{(\infty)}$ of this sequence is an infinite tree locally finite tree. By  K\"{o}nig's lemma, $T^{(\infty)}$ has an infinite branch. Observe that along any infinite branch in $T^{(\infty)}$ one has to encounter infinitely many proper epimorphisms. This derives a contradiction with the fact that $G$ is equationally Noetherian.

\begin{NB}
Note that this is the only place in the paper where we use the assumption that $G$ is equationally Noetherian.
\end{NB}

Denote by $T_{\sol}(\Omega)$ the last term of the sequence (\ref{eq:5.3.1}).

The groups of automorphisms \glossary{name={$A(\Omega_v)$}, description={the automorphism group associated to a vertex $v$ of the trees $T_{\dec}(\Omega)$, and $T_{\sol}(\Omega)$}, sort=A}$A(\Omega_v)$ associated to vertices of $T_{\sol}(\Omega)$ are induced, in a natural way, by the groups of automorphisms associated to vertices of the decomposition trees.

\begin{thm}\label{th:5.3.1}
Let $\Omega$ be a generalised equation in variables $h$ and let $T_{\sol}(\Omega,)$ be the solution tree of $\Omega$. Then the following statements hold.
\begin{enumerate}
    \item \label{it:th1}  For any solution $H$ of the generalised equation $\Omega$ there exist: a path $v_0\to v_1\to \ldots\to v_n = v$  in $T_{\sol}(\Omega)$ from the root vertex $v_0$ to a leaf $v$, a sequence of automorphisms $\sigma = (\sigma_0, \ldots, \sigma_n)$, where $\sigma_i \in  A(\Omega _{v_i})$ and a solution $H^{(v)}$ associated to the vertex $v$, such that
\begin{equation} \label{eq:5.3.3}
\pi_H=\Phi_{\sigma,H^{(v)}} = \sigma_0 \pi(v_0,v_1)\sigma_{1} \ldots  \pi(v_{n-1},v_n)\sigma_n \pi_{H^{(v)}}.
\end{equation}
    \item \label{it:th2} For any path $v_0\to v_1\to \ldots\to v_n = v$  in $T_{\sol}(\Omega)$ from the root vertex $v_0$ to a leaf $v$, any sequence of automorphisms $\sigma =(\sigma_0, \ldots, \sigma_n )$, $\sigma_i \in  A(\Omega _{v_i})$, and any solution $H^{(v)}$ associated to the vertex $v$, the homomorphism $\Phi_{\sigma,H^{(v)}}$ is a solution of $\Omega^\ast$. Moreover, every solution of $\Omega^\ast$ can be obtained in this way.
\end{enumerate}
\end{thm}
\begin{proof}
The statement follows from the construction of the tree $T_{\sol}(\Omega)$.
\end{proof}

\begin{thm}\label{ge}
Let $G=G_1\ast G_2$ be an equationally Noetherian free product of groups and let $\widehat{G}$ be a finitely generated {\rm(}$G$-{\rm)}group. Then the set of all {\rm(}$G$-{\rm)}homomorphisms  $\Hom(\widehat{G},G)$ {\rm(}$\Hom_G(\widehat{G},G)$, correspondingly{\rm)} from $\widehat{G}$ to $G$ can be described by a finite rooted tree. This tree is oriented from the root, all its vertices except for the root vertex are labelled by coordinate groups of generalised equations. To each vertex group we assign the group of automorphisms $A(\Omega_v)$. The leaves of this tree are labelled by groups of the form $F(X)\ast {G_1}_{R(S_1)}\ast {G_2}_{R(S_2)}$ {\rm(}$G\ast F(X)\ast {G_1}_{R(S_1)}\ast {G_2}_{R(S_2)}$, correspondingly{\rm)}, where ${G_i}_{R(S_i)}$ is a coordinate group over $G_i$.

Each edge of this tree is labelled by a {\rm(}$G$-{\rm)}homomorphism. Furthermore, every edge except for the edges from the root is labelled by a proper epimorphism. Each edge from the root vertex corresponds to a {\rm(}$G$-{\rm)}homomorphism from $\widehat{G}$ to the coordinate group of a generalised equation.

Every {\rm(}$G$-{\rm)}homomorphism from $\widehat{G}$ to $G$ can be written as a composition of the {\rm(}$G$-{\rm)}homomorphisms corresponding to the edges, automorphisms of the groups assigned to the vertices, and a specialisation of the variables of the generalised equation corresponding to a leaf of the tree.

This description is effective, provided that the universal Horn theory (the theory of quasi-identities) of $G$ is decidable.\end{thm}
\begin{proof}
Suppose first that $\widehat{G}$ is a finitely generated $G$-group $G$-generated by $X$, i.e. $\widehat{G}$ is generated by $G\cup X$. Let $S$ be the set of defining relations of $\widehat{G}$. We treat $S$ as a system of equations (possibly infinite) over $G$ (with coefficients from $G$). Since $G$ is equationally Noetherian, there exists a finite subsystem $S_0\subseteq S$ such that $R(S)=R(S_0)$. Since every $G$-homomorphism from $\widehat{G}=\factor{\langle G, X\rangle}{\ncl\langle S\rangle}$ to $G$ factors through $G_{R(S_0)}$, it suffices to describe the set of all homomorphisms from $G_{R(S_0)}$ to $G$. We now run the process for the system of equations $S_0$. Since there is a one-to-one correspondence between solutions of the system $S_0$ and homomorphisms from $G_{R(S_0)}$ to $G$, by Theorem \ref{th:5.3.1}, we obtain a description of $\Hom_G(G_{R(S_0)},G)$.

Suppose now that $\widehat{G}$ is not a $G$-group, $\widehat{G}=\langle X\mid S\rangle$ (note that the set $S$ may be infinite). We treat $S$ as a coefficient-free system of equations over $G$. Though, formally, coordinate groups are $G$-groups, in this case, we consider instead the group $G'_{R'(S)}=\factor{F(X)}{R'(S)}$, where $R'(S)=\bigcap\ker(\varphi)$ and the intersection is taken over all homomorphisms $\varphi$ from $F(X)$ to $G$ such that $S\subseteq \ker(\varphi)$. It is clear that $G'_{R'(S)}$ is a residually $G$ group. Since $G$ is equationally Noetherian, there exists a finite subsystem $S_0\subseteq S$ such that $R'(S)=R'(S_0)$. Every homomorphism from $\widehat{G}$ to $G$ factors through $G'_{R'(S_0)}$, it suffices to describe the homomorphisms from $G'_{R'(S_0)}$ to $G$. We now run the process for the system of equations $S_0$ and construct the solution tree, where to each vertex $v$ instead of $G_{R(\Omega_v^*)}$ we associate $G'_{R'(\Omega_v^*)}$ and the corresponding epimorphisms, homomorphisms and automorphisms are defined in a natural way. We thereby obtain a description of the set $\Hom(G'_{R'(S_0)},G)$.
\end{proof}

\begin{thm} \label{Ase}
For any system of equations $S(X)=1$ over an equationally Noetherian free product of groups $G=G_1*G_2$, there exists a finite family of nondegenerate triangular quasi-quadratic systems $C_1,\ldots, C_k$ over groups of the form $G \ast  {G_1}_{R(S_1)}\ast {G_2}_{R(S_2)}$, where ${G_i}_{R(S_i)}$
is a coordinate group over $G_i$, and morphisms of algebraic sets $p_i: V_G(C_i) \rightarrow V_G(S)$, $i = 1, \ldots,k$ such that for every $b \in V_G(S)$ there exists $i$ and $c \in V_G(C_i)$ for which $b = p_i(c)$, i.e.
$$
V_G(S) = p_1(V_G(C_1)) \cup \ldots \cup p_k(V_G(C_k)).
$$
If the universal Horn theory of $G$ is decidable, then the finite families $\{C_1,\ldots, C_k\}$ and $\{p_1,\dots, p_k\}$ can be constructed effectively.
\end{thm}
\begin{proof}
Since, by assumption $G$ is equationally Noetherian, there exists a finite subsystem $S_0$ of $S$ such that $V_G(S_0)=V_G(S)$. Note that $S_0$ can be effectively constructed if the universal Horn theory of $G$ is decidable. We consider the set of all partition tables $\PT$ of the system of equations $S_0$. By every partition table $\T\in\PT$, we construct the generalised equation $\Omega_\T$, see Section \ref{sec:consgeneq}. For every generalised equation $\Omega_\T$, we construct its solution tree $T_{\sol}(\Omega_\T)$.

By Lemma \ref{le:R1} (the notation from which we further use), it follows that there exists finitely many generalised equations $\{\Omega_\T\}$ (corresponding to the partition tables of $S_0$) such that every solution of the system $S(X)=1$ is a composition of the homomorphism $\pi_{\Omega_\T}$ and a solution of the generalised equation $\Omega_\T$. It therefore suffices to prove the statement of the theorem for solutions of generalised equations.

Consider a generalised equation $\Omega_\T=\Omega_{v_0}$. By Theorem \ref{th:5.3.1}, there are finitely many branches in the tree $T_{\sol}(\Omega_{v_0})=T_{\sol}$. To each of the branches, we will assign a finite family of homomorphisms $q_1,\dots,q_{r}$ and a finite family of NTQ systems $C_1,\dots, C_r$ such that all solutions that factor through this branch are a composition of one of the $q_i$'s and a solution of the NTQ system $C_i$. The family of homomorphisms $\{\pi_{\Omega_\T}q_i\}$ and the family of NTQ systems $\{C_i\}$ satisfy the statement of the theorem.

Consider a path $\p: v_0\to v_1\to \dots\to v_n = w$ in $T_{\sol}(\Omega_{v_0})$ from the root vertex $v_0$ to a leaf $w$.

To construct the family of NTQ systems assigned to the branch defined by $\p$, we begin from the leaf of the path $\p$, whose coordinate group, by definition, is an NTQ system and consider a subpath such that
\begin{itemize}
\item it contains at least one vertex with non-trivial group of automorphisms assigned to it,
\item all vertices of the subpath with non-trivial associated automorphism group are of the same type
\item and the canonical epimorphism from the coordinate group corresponding to the initial vertex of the subpath to the one corresponding to its terminal vertex is a proper epimorphism.
\end{itemize}
For every type of such subpath, we prove that one can construct quadratic systems of equations over the coordinate group corresponding to the terminal vertex of the path $\p$ and the corresponding homomorphisms $q_1,\dots, q_r$. The proof then follows by induction on the length of the path $\p$.

\subsubsection*{Case 1: the group of automorphisms assigned to every vertex $v_i$ is trivial.} In this case, all solutions that factor through the branch defined by the path $\p$ can be presented as a composition of the canonical epimorphism $\pi(v_0, w)$ and a solution of $\Omega_w$. Therefore, since the coordinate group of $\Omega_w$ is, by definition, an NTQ system, we set $q_1$ to be $\pi(v_0, w)$ and $C_1$ to be the coordinate group of $\Omega_w$.

\medskip

We further assume that there exists $i$ so that the group of automorphisms  $\VV({\Omega_i})$ assigned to the vertex $v_i$ is non-trivial, i.e. $\tp(v_i)\in \{2,7-10, 12,15\}$. Let $l$ be maximal such number $i$. It suffices to consider the case when the canonical epimorphism  $\pi(v_{l}, w)$ is proper.


\subsubsection*{Case 2: suppose that $\tp (v_l)= 2$.} In this case, every solution $H^{[l]}$ of $\Omega_{v_l}$ that factors through the path $\s=v_l\to v_{l+1}\to \dots\to w$ is $P$-periodic and defines a singular periodic structure $\P(H^{[l]},P)$ on $\Omega_{v_{l}}$ (this periodic structure is the same for all such solutions $H^{[l]}$).

If $\Omega_{v_l}$ is not periodised with respect to the periodic structure $\P(H^{[l]},P)$, then, by Lemma \ref{lem:23-1}, any solution of $\Omega_{v_l}$ is a composition of the canonical epimorphism $\pi(v_l, v_{l+1})$ and a solution of  $\Omega_{v_{l+1}}$, which, in turn, by Case 1, is a composition of the canonical epimorphism $\pi(v_{l+1}, w)$ and a solution of $\Omega_w$. In this case we set $q_1$ to be $\pi(v_l,w)$ and $C_1$ to be the coordinate group of $\Omega_w$.

Suppose that $\Omega_{v_l}$ is periodised and singular with respect to the periodic structure $\P(H^{[l]},P)$. By Lemma \ref{2.10''}, we have that $G_{R(\Omega _{v_{l}}^\ast)}$ is isomorphic to
$$
{\langle G, \bar x\mid R(\Psi \cup \mathcal{O})\rangle},
$$
where $\Psi$ and $\mathcal{O}$ are defined in statement (\ref{spl3}) of Lemma \ref{2.10''}.

By construction, see Section \ref{5.5.5}, if $|C^{(2)}|\ge 2$, then the coordinate group $G_{R(\Omega_{v_{l+1}}^\ast)}$ is generated by $(\bar x\smallsetminus \{h(C^{(2)}\})\cup h(\cc_2)$, where $\cc_2\in C^{(2)}$. Consider the group
$$
G'=\left< G_{R(\Omega _{v_{l+1}}^\ast)}, y_1,\dots, y_{|C^{(2)}|-1}\,\left|\, R\left([y_i,y_j]=1, [y_i, h(\cc_2)]=1, , [y_i, h(C^{(1)})]=1\right)\right.\right>,
$$
where $1\le i,j\le |C^{(2)}|-1$. By construction $G'$ is a coordinate group of an NTQ system over $G_{R(\Omega _{v_{l+1}}^\ast)}$. Furthermore, all solutions of $\Omega_{v_l}$ (that factor through the subpath $\s$) factor through $G'$.

We define $C_1$ to be
$$
C_1=\left< G_{R(\Omega _{w}^\ast)}, y_1,\dots, y_{|C^{(2)}|-1}\,\left|\,
R\left(
\begin{array}{l}
\, [y_i,y_j]=1, [y_i, \pi(v_{l+1}, w)(h(\cc_2))]=1,\\
\, [y_i, \pi(v_{l+1}, w)(h(C^{(1)}))]=1
\end{array}
\right)\right.\right>.
$$
By construction, $C_1$ is a coordinate group of an NTQ system over $G_{R(\Omega _{w}^\ast)}$. All solutions that factor through the subpath $\s$,  factor through $C_1$, therefore, we set $q$ to be the composition of the natural homomorphism from $G_{R(\Omega_{v_l}^*)}$ to $G'$ and the homomorphism from $G'$ to $C_1$ (induced by the canonical epimorphism $\pi(v_{l+1},w)$.

The case when $|C^{(2)}|=1$ is analogous.

\subsubsection*{Case 3: suppose that $\tp (v_{l})\in\{7-10\}$.} Consider the maximal subpath $\s$ of $\p$ from a vertex $v_k$ to $v_l$ such that $\tp(v_{i}) \notin \{2,12,15\}$ for all $k\le i \le l$.

By Lemma \ref{7-10}, the coordinate group $G_{R(\Omega _{v_k}^\ast)}$ is isomorphic to $G_{R(\Ker(\Omega_{v_{k}})^*)}*F(Z)$. In this case we set $C$ to be $G_{R(\Omega_{w}^*)}*F(Z)$ and $q_1$ to be the natural homomorphism  $G_{R(\Ker(\Omega_{v_{k}})^*)}*F(Z)$ to $C_1$ induced by the canonical epimorphism $\pi(v_k, w)$. Clearly, $C_1$ is an NTQ system over $G_{R(\Omega_{w}^*)}$ and every solution that factors through  $\s$ factors through $C$.

\subsubsection*{Case 4: suppose that $\tp(v_l)=12$.} Consider the maximal subpath of $\p$ from a vertex $v_k$ to $v_l$ such that $\tp(v_{i}) \notin \{2,7-10,15\}$ for all $k\le i \le l$.

In this case, by Lemma 2.6 in \cite{Razborov3}, the coordinate group $G_{R(\Omega_{v_k}^\ast)}$ is isomorphic to the coordinate group of a system of quadratic equations $S_1$ over $G_{R(\Omega_{w}^*)}$. We therefore set $C_1$ to be $G_{R(\Omega_{v_k}^\ast)}$.

\subsubsection*{Case 5: suppose that $\tp(v_l)=15$.} Let $[1,j+1]$ be the quadratic part of $\Omega_{v_l}$. If $j\ne 0$, as in Case 4, we consider the coordinate group $G'$ of the system of quadratic equations defined by the quadratic part $[1,j+1]$ of $\Omega_{v_l}$ over the coordinate group $G_{R(\Omega_w^*)}$. The quadratic-coefficient bases in $[1,j+1]$ are identified with their images in the group $G_{R(\Omega_w^*)}$ via the canonical homomorphism $\pi(v_l,w)$. By construction, there exists a natural homomorphism $\phi$ from $G_{R(\Omega_{v_l}^*)}$ to $G'$ defined by the map
$$
h_i\mapsto \left\{
              \begin{array}{ll}
                h_i, & \hbox{if $1\le i\le j$;} \\
                \pi(v_l,w)(h_i), & \hbox{if $i>j$.}
              \end{array}
            \right.
$$

Furthermore, if $v_l$ has outgoing auxiliary edges, then, by Lemma \ref{2.10''}, each regular periodic structure $\langle \P, R\rangle$ on $\Omega_{v_l}$ gives rise to an isomorphism of $G_{R(\Omega_{v_l}^*)}$ and the group
$$
\left< G, \bar x\, \left| \, R\left(\Psi,
\begin{array}{l}
u_{ie}^{h(e_i)}=z_{ie},\hbox{where } e\in T,\, e\in \Sh; 1\leq i\leq m,\\
\left[u_{ie_1},u_{ie_2}\right]=1,\hbox{where } e_j\in T, \,e_j\in \Sh, \, j=1,2; 1\leq i\leq m\\
\left[h(\cc_1),h(\cc_2)\right]=1, \hbox{where } \cc_1,\cc_2\in C^{(1)}\cup C^{(2)}
\end{array}
\right)
\right.
\right>
$$
where $\Psi$ is defined in statement (\ref{spl3}) of Lemma \ref{2.10''}.

For each periodic regular periodic  $\langle \P, R\rangle$ structure on $\Omega_{v_l}$, we set
$$
C_{\langle \P, R\rangle}=\left< G', y_{1},\dots, y_{m}\, \left| \,R\left({\phi(\bar{u}^{(i)})}^{y_{i}}=\phi({\bar{z}}^{(i)})
\right)
\right.
\right>,
$$
where $1\leq i\leq m$ and $\bar u^{(i)}$, ${\bar z^{(i)}}$ are sets of variables defined in (\ref{eq:uie}). The homomorphism $q_{\langle \P, R\rangle}$ is induced by the map
$$
x\in \bar x \mapsto \left\{
                      \begin{array}{ll}
                        \phi(x), & \hbox{if $x\ne h(e_i)$, $i=1,\dots, m$;} \\
                        y_i, & \hbox{if $x=h(e_i)$, $i=1,\dots, m$.}
                      \end{array}
                    \right.
$$
By construction, $C_{\langle \P, R\rangle}$ is a coordinate group of an NTQ system over $G_{R(\Omega_w^*)}$, and solutions that factor through the path $v_l\to\dots \to w$ (and define the regular periodic structure  $\langle \P, R\rangle$) factor through $C_{\langle \P, R\rangle}$.

\medskip

As described above, the coordinate groups assigned to leaves of the tree $T_{\sol}$ are of the form $G \ast F(h_1,\dots, h_{\rho_{A}})\ast {G_1}_{R(S_1)}\ast {G_2}_{R(S_2)}$, where ${G_i}_{R(S_i)}$ is a coordinate group over $G_i$, and therefore they are NTQ systems, as required. The proof now follows by induction on the length of the path $\p$.
\end{proof}

\end{document}